\pdfoutput=1 

\documentclass[a4paper,twoside,leqno]{article}

\usepackage{a4wide}

\usepackage{amssymb}
\usepackage{amsmath}
\usepackage{mathtools}          
\usepackage{amsthm}

\usepackage{microtype}

\usepackage{mathrsfs}

\usepackage{authoraftertitle}

\usepackage{graphicx}

\usepackage{esint}                

\usepackage{enumitem}
\setlist[enumerate,1]{label=(\alph*)}
\setlist[enumerate,2]{label=(\roman*)}

\usepackage{url}
\usepackage{hyperref}

\usepackage{xcolor}

\usepackage{mathtools}
\mathtoolsset{showonlyrefs}

\usepackage{dsfont}

\newtheoremstyle{citing}
  {3pt}
  {3pt}
  {\itshape}
  {}
  {\bfseries}
  {.}
  {.5em}
  {\thmnote{#3}}
\theoremstyle{citing}

\theoremstyle{definition}

\theoremstyle{remark}

\swapnumbers
\theoremstyle{plain}

\newtheorem{theorem}{Theorem}[section]
\newtheorem{lemma}[theorem]{Lemma}
\newtheorem{corollary}[theorem]{Corollary}

\theoremstyle{remark}
\newtheorem{remark}[theorem]{Remark}
\newtheorem*{remark*}{Remark}
\newtheorem{example}[theorem]{Example}

\theoremstyle{definition}
\newtheorem{definition}[theorem]{Definition}
\newtheorem{miniremark}[theorem]{}

\theoremstyle{definition}
\swapnumbers

\usepackage{chngcntr}

\counterwithin*{equation}{theorem}
\counterwithin*{enumi}{theorem}

\newcommand{\VSob}[1]{\mathbf{W}_{#1}}
\newcommand{\VSobz}[1]{\mathbf{W}^{\diamond}_{#1}}

\newcommand{\norm}[3]{\boldsymbol{|} #1 \boldsymbol{|}_{#2;#3}}

\newcommand{\Shuffle}[2]{\operatorname{Sh}(#1,#2)}

\newcommand{\Sequence}[2]{\mathscr{S}(#1,#2)}

\newcommand{\dspace}[2]{\mathscr{D}(#1,#2)}

\newcommand{\distr}[2]{\mathscr{D}'(#1,#2)}

\newcommand{\espace}[2]{\mathscr{E}(#1,#2)}

\newcommand{\project}[1]{{#1}_\natural}

\newcommand{\perpproject}[1]{#1_\natural^\perp}

\newcommand{\grass}[2]{\mathbf{G}(#1,#2)}

\newcommand{\measureball}[2]{{#1}\,{#2}}

\newcommand{\qspace}{\mathbf{Q}}

\newcommand{\Var}[1]{\mathbf{V}_{#1}}     
\newcommand{\IVar}[1]{\mathbf{IV}_{#1}}   
\newcommand{\var}[1]{\mathbf{v}_{#1}}     

\DeclareMathOperator{\graph}{graph}     

\newcommand{\adim}{n}
\newcommand{\vdim}{d}

\newcommand{\oball}[2]{\mathbf{U}(#1,#2)}

\newcommand{\cball}[2]{\mathbf{B}(#1,#2)}

\newcommand{\sphere}[1]{\mathbb{S}^{#1}}

\ifdefined\textint
    \renewcommand{\textint}[2]{{\textstyle\int_{#1}^{#2}}}
\else    
    \newcommand{\textint}[2]{{\textstyle\int_{#1}^{#2}}}
\fi

\ifdefined\textsum
    \renewcommand{\textsum}[2]{{\textstyle\sum_{#1}^{#2}}}
\else  
    \newcommand{\textsum}[2]{{\textstyle\sum_{#1}^{#2}}}
\fi

\newcommand{\natp}{\mathscr{P}}
\newcommand{\nat}{\natp \cup \{0\}}

\newcommand{\integers}{\mathbf{Z}}

\newcommand{\R}{\mathbf{R}}

\newcommand{\LM}{\mathscr{L}}

\newcommand{\Lp}[1]{\mathbf{L}_{#1}}

\newcommand{\besicovitch}[1]{\boldsymbol{\beta}(#1)}

\newcommand{\npp}[1]{\boldsymbol{\xi}_{{#1}}}

\newcommand{\HM}{\mathscr{H}}

\newcommand{\density}{\boldsymbol{\Theta}}

\newcommand{\unitmeasure}[1]{\boldsymbol{\alpha}(#1)}

\newcommand{\restrict}{ \mathop{ \rule[1pt]{.5pt}{6pt} \rule[1pt]{4pt}{0.5pt} }\nolimits }

\newcommand{\ud}{\ensuremath{\,\mathrm{d}}}

\newcommand{\uD}{\ensuremath{\mathrm{D}}}

\newcommand{\weakD}{\ensuremath{\mathbf{D}}}

\newcommand{\id}[1]{\mathrm{id}_{#1}}

\newcommand{\lIm}{[}
\newcommand{\rIm}{]}

\newcommand{\scale}[1]{\boldsymbol{\mu}_{#1}}
\newcommand{\trans}[1]{\boldsymbol{\tau}_{#1}}

\newcommand{\tbwedge}{{\textstyle \bigwedge}}

\newcommand{\tbcup}{{{\textstyle \bigcup}}}

\newcommand{\tbcap}{{{\textstyle \bigcap}}}

\newcommand{\Clos}[1]{\mathop{\mathrm{Clos}}#1}

\newcommand{\cylinder}[4]{\mathbf{C} ( #1, #2, #3, #4 )}

\newcommand{\ellipticity}[1]{\operatorname{ellipticity}(#1)}

\DeclareMathOperator{\Aff}{A}       

\DeclareMathOperator{\Unp}{Unp}

\DeclareMathOperator{\Hom}{Hom}

\DeclareMathOperator{\ap}{ap}

\newcommand{\cnt}[1]{\mathscr{C}^{#1}}

\newcommand{\orthproj}[2]{\mathbf{O}^\ast({#1},{#2})}

\newcommand{\hoelder}[2]{\mathbf{h}_{#1}(#2)}

\newcommand{\hnorm}[4]{\boldsymbol{\|} #2 \boldsymbol{\|}^{#4}_{#1;#3}}

\DeclareMathOperator{\reach}{reach}

\DeclareMathOperator{\dmn}{dmn}

\DeclareMathOperator{\grad}{grad}

\DeclareMathOperator{\spt}{spt}

\DeclareMathOperator{\Tan}{Tan}

\DeclareMathOperator{\Nor}{Nor}

\DeclareMathOperator{\Lip}{Lip}

\DeclareMathOperator{\card}{card}

\DeclareMathOperator{\dist}{dist}

\DeclareMathOperator{\diam}{diam}

\DeclareMathOperator{\without}{\sim}

\DeclareMathOperator{\with}{:}

\DeclareMathOperator{\im}{im}

\date{\today}

\title{Rectifiability of every finite order for stationary integral varifolds}

\author{
  S{\l}awomir Kolasi{\'n}ski
}

\hypersetup{
  unicode=true,
  pdfauthor={\MyAuthor},
  pdftitle={\MyTitle},
  pdfsubject={},
  pdfkeywords={},
  pdfproducer={},
  pdfcreator={}
  pdfinfo={
    orcid={0000-0002-7122-7275},
  }
  hidelinks=true,
  colorlinks=true,
  linkcolor=blue,
  citecolor=darkgray,
  linktocpage=false,
}

\begin{document}

\maketitle

\tableofcontents


\begin{abstract}
    We give an alternative proof of a theorem of Brena, De~Lellis and
    Franceschini (arXiv:2503.00649): the support of a stationary integral
    $\vdim$~dimensional varifold in an open subset of~$\R^{\adim}$ is
    $(\HM^{\vdim},\vdim)$~rectifiable of class~$(k,\alpha)$ for every positive
    integer~$k$ and every $0 < \alpha < 1$, and in particular of
    class~$\cnt{\infty}$. Our proof is independent of theirs and proceeds by
    different means. It compares the varifold, at every scale and at
    $\HM^{\vdim}$~almost every point, with the graph of a solution of the
    Euler--Lagrange system of a smoothed area integrand, and thereby derives an
    excess-decay estimate that is faster than any power of the scale; the
    conclusion then follows from Santilli's characterisation of higher order
    rectifiability together with Whitney's extension theorem.
\end{abstract}

\section{Introduction}

Let $V$ be a $\vdim$~dimensional integral varifold in an open subset $U$ of
$\R^{\adim}$ that is stationary, i.e., $\delta V = 0$. By Allard's theorem
(see~\cite{Allard1972}) the support of the weight measure~$\|V\|$ carries
a~$\vdim$~rectifiable structure. That this structure is as smooth as one may
presently hope in the absence of a~regularity theory for $\spt\|V\|$ itself, was
established by Brena, De~Lellis and Franceschini~\cite{Brena2025a}. The~purpose
of the present paper is to give an alternative proof of their theorem, along a
route that is independent of theirs; we~claim no other novelty, and we~compare
the two arguments in detail below.

\begin{theorem}[see~\ref{thm:finite-rectifiability} and~\ref{cor:Cinfty-rectifiability}]
    \label{introthm:main}
    Let $U \subseteq \R^{\adim}$ be open and let $V \in \IVar{\vdim}(U)$ satisfy
    $\delta V = 0$. Then $\spt\|V\|$ is $(\HM^{\vdim},\vdim)$~rectifiable of
    class~$(k,\alpha)$ for every $k \in \natp$ and $0 < \alpha < 1$; in
    particular it is $(\HM^{\vdim},\vdim)$~rectifiable of class~$\cnt{\infty}$.
\end{theorem}

Here rectifiability of class~$(k,\alpha)$ (respectively~$\cnt{\infty}$) is
understood in the measure-theoretic sense: the carrier of $\|V\|$ is covered by
a~countable union of $\vdim$~dimensional submanifolds of class~$\cnt{k,\alpha}$
(respectively~$\cnt{\infty}$). Equivalently, by the characterisation of
Anzellotti--Serapioni and Santilli (\cite{Anzellotti1994a} and
\cite[1.2, 5.6]{Santilli2019d}), at
$\HM^{\vdim}$~almost every point $\spt\|V\|$ admits an approximate Taylor
expansion of order~$(k,\alpha)$ (respectively of every finite order). No
regularity of $\spt\|V\|$ as a set is asserted or used; the singular examples of
stationary varifolds are not excluded, and whether the complement of Allard's
regular set is $\HM^{\vdim}$~null remains, after more than fifty years, open.

The result is the natural endpoint of a sequence of successive improvements.
Combining Allard's interior regularity with Almgren's multiple valued functions,
as observed by Brakke~\cite{Brakke1978} (the attribution is made explicitly
in~\cite[\S1]{Brena2025a}), one obtains rectifiability of class~$(1,\alpha)$ for
every $\alpha < 1$; Menne~\cite{Menne2013} established rectifiability of
class~$\cnt2$, which is optimal in the wider class of varifolds of locally
bounded first variation. The endpoint was reached by Brena, De~Lellis and
Franceschini~\cite{Brena2025a}, whose priority we~acknowledge. The present work,
begun years before and carried out in ignorance of theirs, arrives at the same
conclusion by different means.

\subsection*{Historical note and an attribution}
\addcontentsline{toc}{subsection}{Historical note and an attribution}

The main ideas and the whole architecture for the proof presented here where
devised by Ulrich Menne some ten years ago. As a result of our collaboration,
after about three years, a draft of this article was produced, which contained
roughly 70--80\% of the contents of the present manuscript. After that, the work
came to a~standstill, not because of any serious difficulties, but rather for
practical and organisational reasons. Following the publication
of~\cite{Brena2025a} Menne relinquished authorship and allowed me to publish the
article on my own.  Getting back to work on this project was quite difficult, as
I had completely lost my enthusiasm for it. The situation changed when new tools
became available. Once I gained access to \emph{Claude~AI} and \emph{Claude
  Code} the work became easier and I could finish it; see also the
section~\emph{\nameref{sec:AI-tools}}.


\subsection*{Method of proof}
\addcontentsline{toc}{subsection}{Method of proof}

The proof rests on a single excess-decay estimate, iterated at $\HM^{\vdim}$~almost
every point. Fix a point at which $\spt\|V\|$ possesses an approximate tangent
plane~$T$. At scale~$s$ we measure the deviation of~$V$ from a smooth comparison
surface~$M_s$ by a height excess $\zeta(s)$, the scale-invariant quadratic mean
of $\dist(\cdot,M_s)$ against~$\|V\|$ over a cylinder of radius~$s$ about the point.
The comparison surface $M_s = \graph v_s$ is the graph over~$T$ of a
$\cnt{\infty}$~function $v_s$ solving the Euler--Lagrange system $L_G(v_s) = 0$
of a fixed smoothed area integrand~$G$, fitted to~$V$ at the scale~$s$. The main
estimate~\ref{thm:decay} is a geometric decay
\begin{displaymath}
    \zeta(s) \le \kappa\, \zeta(2^{20} s) \,,
\end{displaymath}
in which the factor $0 < \kappa < 1$ may be prescribed arbitrarily small at the
cost of enlarging certain structural constants. Iteration
(see~\ref{mr:iterated-decay}) converts this into $\zeta(s) \le
\kappa^{-1}(s/r)^{\alpha}\zeta(r)$ with rate $\alpha = \tfrac{1}{20}\log_2(1/\kappa)$;
hence, into decay faster than any prescribed power of the scale.

The passage from decay to smoothness is then measure-theoretic. Telescoping the
comparison jets across dyadic scales and applying Chebyshev's inequality
(\ref{lem:decay-approxdiff}) shows that decay of order $\alpha \ge k+2$ forces
$\spt\|V\|$ to be approximately differentiable of order~$(k,\alpha')$ at
$\HM^{\vdim}$~almost every point; Santilli's characterisation of higher order
rectifiability~\cite[5.6]{Santilli2019d} then yields rectifiability of
class~$(k,\alpha)$. As $k$ is arbitrary, one obtains rectifiability of every
finite order (\ref{thm:finite-rectifiability}). Finally, the approximate Taylor
expansions of all orders are available on \emph{one} common set of full measure,
so Whitney's $\cnt{\infty}$~extension theorem upgrades the conclusion
to~$\cnt{\infty}$ (\ref{lem:set-to-function}, \ref{lem:Ck-to-Cinfty}
and~\ref{cor:Cinfty-rectifiability}), by an argument we owe to Menne.

The decay estimate occupies the bulk of the paper. Its proof rests on five
ingredients: lower bounds for the reach of a graph in terms of the second
derivatives of the graph function (\S\ref{sec:reach}), which is what makes the
nearest point projection onto the comparison surface available on the region
where the estimates are run; a reverse-Caccioppoli, or tilt-excess, inequality
relative to a $\cnt2$~comparison manifold --- ``height controls the tilt'' ---
worked out in \S\ref{sec:height-tilt} and~\S\ref{sec:two-heights}; an
$\qspace_Q$~valued Lipschitz approximation of~$V$ over~$T$, with estimates for
the bad set on which $V$ is not well approximated by a graph; interior Schauder
estimates of every order for the system $L_G$ and its linearisation, developed
here in a self-contained way in the spirit of~\cite[5.2.15]{Federer1969}; and
the construction of the comparison functions $v_s$ with the prescribed
behaviour.

\subsection*{Comparison with the work of Brena, De~Lellis and Franceschini}
\addcontentsline{toc}{subsection}{Comparison with the work of Brena, De~Lellis and Franceschini}

Our strategy and that of~\cite{Brena2025a} agree in their broad architecture and
diverge in most of the technical execution. Since the two works were carried out
independently, we describe the parallel in some detail.

The common ground is the following. Both proofs replace the linear building block
of Allard's $\varepsilon$-regularity --- the best approximating \emph{plane} in
the classical Allard--De~Giorgi decay lemma --- by a best approximating
\emph{smooth solution} of a minimal-surface-type system, and both keep the ratio
of consecutive scales fixed while driving the decay factor ($\kappa$ here,
$\gamma$ in~\cite{Brena2025a}) to be arbitrarily small; this is precisely what
produces decay faster than any power, rather than at a single H\"older rate. The
idea of enlarging the class of comparison objects is not new: Savin, in
\cite{Savin2018}, approximates by harmonic quadratic polynomials in place of
linear maps and so passes from $\cnt{1,\alpha}$ to $\cnt{2,\alpha}$~estimates
for his notion of viscosity solution of the minimal surface system; both works
considered here push the device to~$\cnt{\infty}$.
Both derive the requisite energy estimate from a Caccioppoli-type
\emph{tilt-excess inequality} for~$V$ regarded as a generalised graph over the
comparison surface, both control the higher multiplicity $Q > 1$ and the possible
``holes'' by an Almgren-style $\qspace_Q$~valued Lipschitz approximation together
with a density-deficit term --- our bad-set estimates, their parameter~$\eta$,
which is vacuous when $Q = 1$ --- and both conclude by a Whitney-type patching
that exploits the mutual closeness, in every $\cnt{k}$~norm, of the comparison
surfaces attached to nearby scales.

The differences are substantial. \emph{First, the comparison surface and the
parametrisation of~$V$.} In~\cite{Brena2025a} the reference objects are genuine
classical minimal surfaces~$M$, and $V$ is written as a $\qspace_Q$~valued graph
over the \emph{normal bundle} of the curved surface~$M$; the decay lemma updates
$M$ to a new minimal surface at the finer scale, and the analysis is intrinsic to
the moving curved surface, governed by the distance function $d_M^2$ and its
elliptic inequality $\Delta_L \tfrac12 d_M^2 + \delta^2 d_M^2 \ge \tfrac14 |L -
T_{p(z)}M|^2$. In the present paper the comparison surfaces $M_s = \graph v_s$
are, throughout, graphs over the \emph{fixed} approximate tangent plane~$T$;
there is neither a normal-bundle parametrisation nor a moving frame, and the
curvature of the comparison surface enters solely through $v_s$ and its
derivatives.

\emph{Second, the comparison equation.} The reference~\cite{Brena2025a} retains
the exact minimal-surface system, regularising only the average in order to
update the surface, at the cost that the surface obtained in the limit is smooth
but not shown to be minimal (their Theorem~1.3). We instead compare against
solutions of $L_G(v_s) = 0$ for a fixed \emph{smoothed} area integrand~$G$ --- an
elliptic regularisation performed at the level of the integrand rather than of
the surface. The integrand~$G$ is uniformly elliptic and smooth, which is
exactly what our higher order estimates require and what removes the degeneracy
of the area integrand at large slopes.

\emph{Third, the elliptic estimates.} In~\cite{Brena2025a} the regularity of the
comparison surfaces is imported from the classical theory of the minimal-surface
system; here it is proved from scratch. In~\S\ref{sec:elliptic} the operators
$L_F$, $T_A$ and the negative norm~$\|\cdot\|_{-1,1}$ are developed
following~\cite{Menne2013}, and~\ref{lem:schauder-est} establishes interior
Schauder estimates of arbitrary order by a self-contained higher order tower
modelled on~\cite[5.2.15]{Federer1969}.  The tower proper ---
\ref{lem:H}, \ref{lem:ho:tower} and~\ref{thm:ho} --- appeals neither
to~\cite{Gilbarg2001} nor to Simon's scaling method~\cite{Simon1997};
\ref{rem:ho:plan} records how it stands to the latter.  Standard $\Lp{p}$~theory
is used elsewhere, in~\ref{lem:L1-T-1-est} and in the construction of the
comparison functions.  This self-containedness is forced by our need to control
the constants at all orders simultaneously.

\emph{Fourth, the route from decay to smoothness.} The reference~\cite{Brena2025a}
passes directly to $\cnt{\infty}$: their quantitative closeness of two comparison
minimal surfaces feeds a Whitney construction that produces the
$\cnt{\infty}$~graphs at once. Our argument is instead factored through the theory
of approximate differentiability: the decay first yields, via
\ref{lem:decay-approxdiff}, approximate differentiability of a definite finite
order~$(k,\alpha)$, which by Santilli's theorem~\cite{Santilli2019d} is
equivalent to rectifiability of class~$(k,\alpha)$; the $\cnt{\infty}$~statement
is obtained as a separate step from the finite-order statements on a common
set of full measure. This modularity isolates the geometric measure theoretic
content (Santilli's characterisation) from the analytic content (the decay), and
yields, as a by-product, the sharp finite-order statement of
\ref{thm:finite-rectifiability}. At the level of prerequisites, finally, both
papers build on Almgren's multiple valued functions, but they stand differently
to Menne's~\cite{Menne2013} $\cnt2$~rectifiability: \cite{Brena2025a} quotes it
as the previous state of the art, and as one of the treatments that adopt
Almgren's original approximation scheme, but does not otherwise rest on~it,
whereas here it is a~working tool --- the operators $L_F$, $T_A$ and the
norm~$\|\cdot\|_{-1,1}$ of~\S\ref{sec:elliptic} follow
\cite[3.1]{Menne2013}. The analytic core of~\cite{Brena2025a} rests on
tilt-excess and Lipschitz-approximation apparatus of De~Lellis--Spadaro type,
whereas ours continues the line of~\cite{Menne2010,Menne2013} and of Federer's
treatment of elliptic systems.

\section{Notation}

In~principle we shall follow the notation of Federer;
see~\cite[pp.~669--671]{Federer1969}. In~particular, given two sets $A,B$, we
denote with $A \without B$ their \emph{set-theoretic difference} and, for every
$a\in \R^\adim$ and~$s \in \R$ we define the functions $\trans{a}(x) = a+x$
and~$\scale{s}(x) = sx$; see~\cite[2.7.16, 4.2.8]{Federer1969}. Whenever $X$ and
$Y$ are vector-spaces, $\Lambda \in \Hom(X,Y)$, and $x \in X$ we use alternative
notations $\langle x ,\, \Lambda \rangle$ and $\Lambda x$ for the value
of~$\Lambda$ on the vector~$x$. In case $X$ is an inner product space we write
$x \bullet y$ for the \emph{scalar product} of $x,y \in X$. The symbol
$\bigodot^k X$ denotes the space of all \emph{symmetric $k$-forms} on a
vectorspace~$X$, and $\bigodot^k X$ the space of all \emph{symmetric
  $k$-vectors} on~$X$; cf.~\cite[1.9]{Federer1969}. If~$\dim X < \infty$, then
$\bigodot^k X$ and $\bigodot_k X$ are dual to each other; in this case we shall
tacitly identify $\Hom(\bigodot_kX,\R)$ with $\bigodot^k X$.

Concerning varifolds, we~shall use the notation of Allard~\cite{Allard1972}.

Following~\cite{Almgren1968} and~\cite{Almgren2000e}, if $S \in
\grass{\adim}{\vdim}$ is a~$\vdim$~dimensional linear subspace of~$\R^{\adim}$,
then $\project S \in \Hom(\R^{\adim},\R^{\adim})$ shall denote the
\emph{orthogonal projection} onto~$S$. In~particular, if $p \in
\orthproj{\adim}{\vdim}$ is such that $\im p^* = S$, then $\project S = p^*
\circ p$.

As in~\cite[4.1]{Federer1959} if $\Sigma \subseteq \R^{\adim}$, we denote by
$\Unp(\Sigma)$ the domain of the~\emph{nearest point projection} map
\begin{displaymath}
    \npp{\Sigma} = \R^{\adim} \times \Sigma \cap
    \bigl\{ (x,y) : \dist(x,\Sigma) = |x-y| ,\, \cball{x}{|x-y|} \cap \Sigma = \{ y \} \bigr\} \,.
\end{displaymath}
In case $\Sigma$ is a manifold of class~$\cnt{2}$ we write
$\mathbf{b}(\Sigma,x)$ for the second fundamental form of~$\Sigma$ at~$x \in
\Sigma$; cf.~\cite[2.5(1)]{Allard1972}.  Its norm is always that of a~bilinear
map,
\begin{displaymath}
    \| \mathbf{b}(\Sigma,x) \| = \sup \bigl\{ | \mathbf{b}(\Sigma,x)(\sigma,\tau) |
    : \sigma,\tau \in \Tan(\Sigma,x) ,\, |\sigma| = |\tau| = 1 \bigr\} \,;
\end{displaymath}
for a symmetric form with values in $\Nor(\Sigma,x)$ this exceeds the supremum
over the diagonal in general, and it is the former that is needed here.

\section{Reach of a graph}
\label{sec:reach}

\begin{lemma} \label{lemma:structure_submanifold}
        Suppose $M$ is an $m$ dimensional submanifold of class $2$ of $\mathbf
        R^n$, the linear maps
        \begin{equation*}
                Q_a(v) : \Tan (M,a) \to \Tan (M,a), \quad
                L_a(v) : \Tan (M,a) \times \Nor (M,a) \to \mathbf R^n
                \times \mathbf R^n,
        \end{equation*}
        defined for $(a,v) \in \Nor (M)$, satisfy
        \begin{equation*}
                \langle t, Q_a(v) \rangle \bullet u = \langle t \odot u,
                \mathbf b (M,a) \rangle \bullet v, \quad \langle (t,w), L_a
                (v) \rangle = ( t, w - \langle t, Q_a(v) \rangle )
        \end{equation*}
        whenever $t,u \in \Tan (M,a)$ and $w \in \Nor (M,a)$, and $\psi :
        \Nor (M) \to \mathbf R^n$ satisfies
        \begin{equation*}
                \psi (a,v) = a+v \quad \text{for $(a,v) \in \Nor (M)$}.
        \end{equation*}

        Then, the following two statements hold.
        \begin{enumerate}
                \item \label{item:structure_submanifold:Tan} If $(a,v) \in
                \Nor (M)$, then
                \begin{equation*}
                        \Tan ( \Nor (M), (a,v) ) = \im L_a(v).
                \end{equation*}
                \item \label{item:structure_submanifold:covering} If $(a,v)
                \in \Nor(M)$ and $\| \mathbf b (M,a) \| \cdot |v| < 1$, then
                $\uD \psi (a,v)$ is univalent.
        \end{enumerate}
\end{lemma}

\begin{proof} [Proof of \ref{item:structure_submanifold:Tan}]
        As \cite[5.4.11]{Federer1969} remains valid if both occurrences of
        ``class $\infty$'' are replaced by ``class $1$'', we can construct a
        relatively open neighbourhood $U$ of $a$ in $M$ and, for $i = 1,
        \ldots, n-m$, maps $\nu_i : U \to \mathbf R^n$ of class $1$ such that
        \begin{gather*}
                \text{$\nu_1(x), \ldots, \nu_{n-m}(x)$ form an orthonormal
                basis of $\Nor (M,x)$, for $x \in M$}, \\
                \text{$\im \uD \nu_i (a) \subset \Tan (M,a)$ for $i = 1,
                \ldots, n-m$}.
        \end{gather*}
        We let $R_x : \mathbf R^{n-m} \simeq \Nor (M,x)$ be defined by
        \begin{equation*}
                R_x (y) = \sum_{i=1}^{n-m} y_i \nu_i (x) \quad \text{whenever
                $x \in U$ and $y \in \mathbf R^{n-m}$}
        \end{equation*}
        and pick $b \in \mathbf R^{n-m}$ such that $R_a (b) = v$.  Next,
        considering
        \begin{equation*}
                \phi : U \times \mathbf R^{n-m} \to \Nor (M), \quad
                \text{$\phi (x,y) = (x,R_x(y))$ for $x \in U$ and $y \in
                \mathbf R^{n-m}$},
        \end{equation*}
        we note $\phi (a,b) = (a,v)$ and compute, for $t \in \Tan (M,a)$ and
        $z \in \mathbf R^{n-m}$,
        \begin{equation*}
                \langle (t,z), \uD \phi (a,b) \rangle = \left ( t, R_a(z) +
                \sum_{i=1}^{n-m} b_i \langle t, \uD \nu_i (a) \rangle \right)
        \end{equation*}
        Noting that, whenever $t,u \in \Tan (M,a)$, we have
        \begin{equation*}
                \langle t, \uD \nu_i (a) \rangle \bullet u = - \langle t
                \odot u, \mathbf b (M,a) \rangle \bullet \nu_i (a) = - \langle
                t, Q_a ( \nu_i(a) ) \rangle \bullet u,
        \end{equation*}
        we infer
        \begin{equation*}
                \uD \nu_i (a) = - Q_a ( \nu_i (a) ) \quad \text{for $i = 1,
                \ldots, n-m$}.
        \end{equation*}
        The linearity of $Q_a$ then yields
        \begin{equation*}
                \langle (t,z), \uD \phi (a,b) \rangle = ( t, R_a(z) - \langle
                t, Q_a(v) \rangle ) \quad \text{for $t \in \Tan (M,a)$ and $z
                \in \mathbf R^{n-m}$}.
        \end{equation*}
        As $L_a(v)$ is univalent and $T = \Tan (\Nor(M),(a,v))$ is an $n$
        dimensional vector subspace containing $\im \uD \phi (a,b)$ by
        \cite[3.1.21]{Federer1969}, the equation
        \begin{equation*}
                L_a(v) \circ ( \mathbf 1_{\Tan (M,a)} \times R_a ) = \uD \phi
                (a,b)
        \end{equation*}
        implies that $\uD \phi (a,b)$ is univalent with $\im L_a(v) = \im
        \uD \phi (a,b) = T$.
\end{proof}

\begin{proof} [Proof of \ref{item:structure_submanifold:covering}]
        Since $\| Q_a(v) \| < 1$ by our hypothesis, $\mathbf 1_{\Tan (M,a)} -
        Q_a(v)$ is univalent.  Therefore, employing the isomorphism $\alpha :
        \Tan (M,a) \times \Nor (M,a) \simeq \mathbf R^n$ induced by addition,
        we obtain
        \begin{equation*}
                \uD \psi (a,v) \circ L_a (v) = \alpha \circ \big ( ( \mathbf
                1_{\Tan(M,a)} - Q_a(v) ) \times \mathbf 1_{\Nor (M,a)} \big )
        \end{equation*}
        from \ref{item:structure_submanifold:Tan}; whence, we infer the
        conclusion.
\end{proof}

\begin{definition}
        \label{def:pi}
        We let $\boldsymbol \pi = 2 \inf \{ t \with t > 0, \cos (t) = 0 \}$.
\end{definition}

\begin{definition}
        \label{def:derivative-on-subsets}
        Whenever $S \subset \mathbf R$, $Y$ is a normed space, and $\gamma : S
        \to Y$.

        Then, we define $\gamma'$ to be the function whose domain consists of
        those points $a \in \Clos S$ such that the linear span of $\Tan (S,a)$
        equals $\mathbf R$ and such that there exist an open neighbourhood $U$
        of $a$ in $\mathbf R$ and a function $\beta : U \to Y$ differentiable
        at $a$ with $\beta | ( S \cap U ) = \gamma | ( S \cap U )$, and whose
        value at such $a$ equals $\beta'(a)$.
\end{definition}

\begin{remark}
        \label{rem:derivative-on-subsets}
        This extends \cite[2.9.19, 2.9.22]{Federer1969} in the spirit of
        \cite[3.1.22]{Federer1969}.  We will employ this concept for intervals
        $S$.
\end{remark}
                
\begin{definition}
        \label{def:intrinsic-metric}
        Suppose $M$ is a connected submanifold of class $1$ of $\mathbf R^n$.

        Then, by the \emph{intrinsic metric on $M$}, we mean the function
        whose domain is $M \times M$ and whose value at $(x,y) \in M \times M$
        equals the infimum of the set of numbers
        \begin{equation*}
                {\textstyle\int_a^b} | \gamma'(t) | \ud \mathscr L^1 \, t
        \end{equation*}
        corresponding to $- \infty < a < b < \infty$ and $\gamma : \{ t \with
        a \leq t \leq b \} \to M$ with $\gamma(a) = x$, $\gamma(b) = y$, and
        $\Lip \gamma < \infty$.
\end{definition}

\begin{remark}
    \label{rem:intrinsic-metric}
    The intrinsic metric on $M$ is a metric on $M$ yielding the subspace
    topology on $M$ induced by $\mathbf R^n$ by \cite[3.2.46]{Federer1969}.
    Approximation (e.g., by means of \cite[3.1.11, 3.1.20, 4.1.2]{Federer1969})
    shows that $\gamma$ could be required to be of class $1$; hence, our
    definition is in accordance with \cite[Chapter~7,
    Definition~2.4]{Carmo1992}.
\end{remark}

\begin{example} \label{example:geodesic_spheres}
        Suppose $\rho$ is the intrinsic metric on $\mathbf S^{n-1}$.  Then,
        for $u,v \in \mathbf S^{n-1}$, we have
        \begin{equation*}
                \rho (u,v) \leq \boldsymbol \pi \quad \text{and} \quad \cos (
                \rho (u,v)) = u \bullet v;
        \end{equation*}
        this follows from \cite[Chapter 3, Example 2.11 and Chapter 7, Theorem
        2.8]{Carmo1992}.
\end{example}

\begin{lemma} \label{lemma:bending_curves}
        Suppose $I$ is a compact nondegenerate subinterval of $\mathbf R$
        with $0 \in I$, $\gamma : I \to \mathbf R^n$ is of class $2$ relative
        to $I$, $\gamma(0)=0$, $\gamma'(0) = u$, $S = \{ y \with u \bullet y = 0 \}$,
        and
        \begin{equation*}
                \text{$| \gamma'(t) | = 1$ and $| \gamma''(t) | \leq 1$} \quad
                \text{whenever $0 \leq t \leq \boldsymbol \pi/2$ and $t \in
                I$}.
        \end{equation*}

        Then, whenever $0 \leq t \leq \boldsymbol \pi /2$ and $t \in I$, there
        holds
        \begin{gather*}
                \gamma'(t) \bullet u \geq \cos (t), \quad \gamma(t) \bullet u
                \geq \sin (t), \\
                | S_\natural ( \gamma'(t)) | \leq \sin (t), \quad | S_\natural
                ( \gamma(t)) | \leq 1- \cos (t).
        \end{gather*}
\end{lemma}

\begin{proof}
        Taking $\rho$ as in \ref{example:geodesic_spheres}, we have
        \begin{equation*}
                \mathbf S^{n-1} \cap \{ y \with \rho (u,y) \leq s \} = \mathbf
                S^{n-1} \cap \{ y \with y \bullet u \geq \cos (s) \}
        \end{equation*}
        whenever $0 \leq s \leq \boldsymbol \pi$.  Thus, we estimate
        \begin{equation*}
                \rho ( \gamma'(t), \gamma'(0) ) \leq {\textstyle\int_0^t} |
                \gamma''(s)| \ud \mathscr L^1 \, s \leq t, \quad \gamma'(t)
                \bullet u \geq \cos (t) ,;
        \end{equation*}
        hence, $\gamma(t) \bullet u \geq \sin (t)$ by integration.  As $1 = |
        \gamma'(t) \bullet u |^2 + | S_\natural ( \gamma'(t) ) |^2$, it
        follows $| S_\natural ( \gamma'(t)) | \leq \sin (t)$; hence, $|
        S_\natural ( \gamma(t)) | \leq 1 - \cos (t)$ by integration.
\end{proof}

\begin{remark} \label{remark:bending_curves}
        The preceding lemma will be applied in conjunction with the following
        observation: \emph{If $u \in \mathbf S^{n-1}$ and $S = \{ y \with u
        \bullet y = 0 \}$, then
        \begin{equation*}
                A = \{ x \with |x \bullet u| \geq \sin (t), | S_\natural (x) |
                \leq 1 - \cos (t) \} \subset \mathbf R^n \without \mathbf U
                (v,1)
        \end{equation*}
        whenever $v \in S$, $|v| \geq 1$, and $0 \leq t \leq \boldsymbol
        \pi/2$; moreover, if $|v|>1$, then $\mathbf U$ may be replaced by
        $\mathbf B$.} In~fact, for $x \in A$, we have $\cos (t) \leq |
        S_\natural (v) | - | S_\natural (x) | \leq | S_\natural (v-x) |$ and
        \begin{equation*}
                |x-v|^2 = | x \bullet u |^2 + | S_\natural (x-v) |^2 \geq 1
        \end{equation*}
        and the first and last inequalities are strict in case $|v|>1$.
\end{remark}

\begin{theorem} \label{thm:bending_curves}
        Suppose $I$ is a compact subinterval of $\mathbf R$ with $0 \in I$,
        $\gamma : I \to \mathbf R^n$ is a map of class $2$ relative to $I$,
        $\gamma(0) = 0$,
        $J = \dmn \gamma'$, and%
        \begin{footnote}
                {We notice that if $\card I = 1$, then $J = \varnothing$, and
                that otherwise, $J = I$.}
        \end{footnote}
        \begin{equation*}
                \text{$| \gamma'(t) | = 1$ and $| \gamma''(t) | \leq 1$} \quad
                \text{whenever $t \in J$}.
        \end{equation*}

        Then, whenever $t \in I$ and $0 \leq t \leq \boldsymbol \pi$, there
        holds
        \begin{equation*}
                \gamma(t) \notin \mathbf U (v,1) \quad
                \text{whenever $v \in \mathbf S^{n-1}$ and $\gamma'(0) \bullet
                v = 0$}.
        \end{equation*}
\end{theorem}

\begin{proof}
        We suppose $v \in \mathbf S^{n-1}$, $\gamma'(0) \bullet v = 0$, $r
        \geq 0$, and $\gamma(r) \in \mathbf U (v,1)$ and we will show that $r
        > \boldsymbol \pi$.  For this purpose, we denote by $C$ the set of
        $c$ such that $0 \leq c \leq r$ and
        \begin{equation*}
                | \gamma(c) - v | \geq | \gamma(t) - v | \quad \text{for $0
                \leq t \leq r$}.
        \end{equation*}
        Clearly, $C$ is a nonempty compact subset of $\{ t \with 0 \leq t \leq
        r \}$; hence, $s = \sup C \in C$.  We notice that
        \begin{equation*}
                |\gamma(s)-v| \geq 1 \quad \text{and} \quad s < r.
        \end{equation*}Moreover, from \ref{lemma:bending_curves} and
        \ref{remark:bending_curves}, we obtain $| \gamma(t)-v | \geq 1$
        whenever $0 \leq t \leq \boldsymbol \pi/2$ and $t \in I$; hence,
        noting $r > \boldsymbol \pi/2$ and $\{ t \with 0 \leq t \leq
        \boldsymbol \pi/2 \} \subset I$, we infer
        \begin{equation*}
                s > 0 \quad \text{and} \quad \text{if $| \gamma(s)-v | = 1$,
                then $s \geq \boldsymbol \pi/2$}.
        \end{equation*}
        As $\Tan ( \im \gamma, \gamma(s)) \subset \Tan ( \mathbf B
        (v,|\gamma(s)-v|), \gamma(s)) = \mathbf R^n \cap \{ u \with u \bullet
        (v-\gamma(s)) \geq 0 \}$, we then infer that $u = \gamma'(s)$
        satisfies $u \bullet (v-\gamma(s)) = 0$ from \cite[3.1.21]{Federer1969}.
        Therefore, twice applying \ref{lemma:bending_curves}, once with
        $\gamma(t)$ replaced by $\gamma(t-s)-\gamma(s)$ and once with $\gamma$
        replaced by $\gamma (s-t) - \gamma(s)$ yields
        \begin{equation*}
                |(\gamma(t)-\gamma(s)) \bullet u| \geq \sin |t-s| \quad
                \text{and} \quad | S_\natural ( \gamma(t)-\gamma(s) ) | \leq 1
                - \cos |t-s|
        \end{equation*}
        whenever $t \in I$ and $|t-s| \leq \boldsymbol \pi/2$, where $S = \{ y
        \with y \bullet u = 0 \}$.  Applying \ref{remark:bending_curves} with
        $v$ replaced by $v-\gamma(s)$ entails that, for such $t$, we have
        \begin{equation*}
                \gamma(t) \notin \mathbf U (v,1) \quad \text{and} \quad
                \text{if $|\gamma(s)-v|>1$, then $\gamma(t) \notin \mathbf B
                (v,1)$} \,;
        \end{equation*}
        hence, $r > s + \boldsymbol \pi/2$ and, if $| \gamma(s) - v | > 1$, then
        $0 < s - \boldsymbol \pi/2$.  Thus, $r > \boldsymbol \pi$.
\end{proof}

\begin{remark}
        \label{rem:pi-optimal}
        Examples show that $\boldsymbol \pi$ may not be replaced by
        a~larger number.
\end{remark}

\begin{definition}
        \label{def:geodesic-curve}
        Suppose $M$ is a submanifold of class $2$ of $\mathbf R^n$.

        Then, $\gamma$ is termed a \emph{geodesic curve in $M$} if and only if
        $I = \dmn \gamma$ is a nondegenerate subinterval of $\mathbf R$,
        $\gamma : I \to M$ is of class $2$ relative to $I$ and
        \begin{equation*}
                \gamma''(t) = \langle \gamma'(t) \odot \gamma'(t), \mathbf b
                (M,\gamma(t)) \rangle \quad \text{for $t \in I$}.
        \end{equation*}
\end{definition}

\begin{corollary}
        \label{cor:geodesic-leaves-ball}
        Suppose $M$ is a submanifold of class $2$ of $\mathbf R^n$, $(a,v) \in
        \Nor (M)$, $|v| \neq 0$,
        \begin{equation*}
                \| \mathbf b (M,x) \| \leq |v|^{-1} \quad \text{for $x \in M$},
        \end{equation*}
        and $\gamma : I \to M$ is a geodesic curve in $M$ with $0 \in I$,
        $\gamma (0) = a$, and $|\gamma'(t)| = 1$ for $t \in I$.

        Then, whenever $t \in I$ and $0 \leq t \leq \boldsymbol \pi |v|$, there
        holds
        \begin{equation*}
                \gamma(t) \notin \mathbf U (a+v,|v|).
        \end{equation*}
\end{corollary}

\begin{proof}
        We may assume $a = 0$.  Put $\lambda = |v|^{-1}$ and $\tilde \gamma
        (\tau) = \lambda \gamma ( \lambda^{-1} \tau )$ for $\tau \in \lambda
        I$.  Then $\tilde\gamma$ is a geodesic curve in the submanifold
        $\lambda M$ with $\tilde\gamma(0) = 0$ and $|\tilde\gamma'| = 1$;
        moreover $\mathbf b ( \lambda M, \lambda x ) = \lambda^{-1} \mathbf b
        (M,x)$; whence, $\| \mathbf b (\lambda M, \cdot) \| \leq \lambda^{-1}
        |v|^{-1} = 1$ and $|\tilde\gamma''(\tau)| = | \langle
        \tilde\gamma'(\tau) \odot \tilde\gamma'(\tau) , \mathbf b (\lambda M,
        \tilde\gamma(\tau)) \rangle | \leq 1$ for $\tau \in \lambda I$.  Since
        $\lambda v \in \mathbf S^{n-1}$ and $\tilde\gamma'(0) = \gamma'(0) \in
        \Tan (M,0)$ is orthogonal to~$v$, \ref{thm:bending_curves} yields
        $\tilde\gamma(\tau) \notin \mathbf U ( \lambda v, 1 )$ whenever $\tau
        \in \lambda I$ and $0 \leq \tau \leq \boldsymbol\pi$.  Substituting
        $\tau = \lambda t$, the constraint $0 \leq \tau \leq \boldsymbol\pi$
        becomes $0 \leq t \leq \boldsymbol\pi\lambda^{-1} = \boldsymbol\pi|v|$,
        while $\tilde\gamma(\lambda t) \notin \mathbf U (\lambda v,1)$ reads
        $\gamma(t) \notin \mathbf U (v,|v|)$; as $a = 0$, this is the
        conclusion.

        We~stress that the bound $\boldsymbol\pi|v|$ cannot be replaced
        by~$\boldsymbol\pi$: the parameter~$t$ is an arclength, and the
        normalisation $|v|=1$ used in the reduction rescales it.
\end{proof}

\begin{corollary}
    \label{cor:reach-of-graph}
    Suppose $c \in \R^{\vdim}$, $0 < r \le \infty$, $f : \cball cr \to
    \R^{\adim-\vdim}$ is of class~$\cnt{2}$, $\gamma \in (0,1)$, $\Delta \in
    (0,\infty)$,
    \begin{equation*}
        \| \uD f(x) \| \le \gamma \quad \text{and} \quad \| \uD^2 f(x) \| \le \Delta
        \quad \text{whenever $x \in \cball cr$} \,,
    \end{equation*}
    and $M = \{ (x,f(x)) : x \in \cball cr \}$, where $\cball c\infty =
    \R^{\vdim}$. Then
    \begin{equation*}
        \sup \{ \| \mathbf{b}(M,z) \| : z \in M \} \le \Delta \,,
    \end{equation*}
    and, for every $0 < \rho \le \Delta^{-1}$ with $\boldsymbol\pi\rho \le r$,
    every $y \in \R^{\adim}$ such that $\dist(y,M) < \rho$ and some point of~$M$
    nearest to~$y$ is of the form $(x,f(x))$ with $|x-c| \le r - \boldsymbol\pi\rho$
    has a~unique nearest point in~$M$. In particular, if $r = \infty$ then
    $\reach(M) \ge \Delta^{-1}$.
\end{corollary}

\begin{proof}
    Writing $\iota(x) = (x,f(x))$, the set $M = \im \iota$ is a~$\cnt{2}$~graph,
    and for $z = \iota(x)$ and unit vectors $\sigma = \uD\iota(x)w_1 /
    |\uD\iota(x)w_1|$ and $\tau = \uD\iota(x)w_2/|\uD\iota(x)w_2|$ in
    $\Tan(M,z)$, with $w_1,w_2 \in \R^{\vdim} \without \{0\}$, we have, since
    $\uD^2\iota(x)(w_1,w_2) = (0,\uD^2f(x)(w_1,w_2))$ and
    $|\uD\iota(x)w_i|^2 = |w_i|^2 + |\uD f(x)w_i|^2 \ge |w_i|^2$,
    \begin{equation*}
        |\mathbf{b}(M,z)(\sigma,\tau)|
        = \frac{|\project{\Nor(M,z)}(0,\uD^2f(x)(w_1,w_2))|}{|\uD\iota(x)w_1| \, |\uD\iota(x)w_2|}
        \le \frac{|\uD^2 f(x)(w_1,w_2)|}{|w_1| \, |w_2|} \le \Delta \,;
    \end{equation*}
    whence, $\| \mathbf{b}(M,z) \| \le \Delta$.

    From~\ref{thm:bending_curves}, applied to $t \mapsto (\beta(\rho t) -
    \beta(0))/\rho$, we~recall that a~curve $\beta : [0,L] \to \R^{\adim}$ of
    class~$\cnt{2}$ with $L \le \boldsymbol\pi\rho$, $|\beta'| = 1$, $|\beta''|
    \le \rho^{-1}$, and $\beta'(0) \bullet \nu = 0$ for some $\nu \in
    \sphere{\adim-1}$ satisfies
    \begin{equation}
        \label{eq:cor:reach:bc}
        \beta(t) \notin \oball{\beta(0)+\rho\nu}{\rho}
        \quad \text{for $0 \le t \le L$} \,.
    \end{equation}

    Fix $0 < \rho \le \Delta^{-1}$ with $\boldsymbol\pi\rho \le r$ and $y \in
    \R^{\adim}$ as in the statement, with nearest point $a = (x_a,f(x_a)) \in M$
    satisfying $|x_a - c| \le r - \boldsymbol\pi\rho$; then $y - a \in \Nor(M,a)$
    by~\cite[3.1.21]{Federer1969}. Put $t_0 = |y-a| = \dist(y,M) < \rho$. If $t_0
    = 0$ then $y = a$ and any $b \in M$ with $|y-b| \le t_0$ equals~$a$, so
    suppose $t_0 > 0$ and set $\nu = (y-a)/t_0 \in \sphere{\adim-1}$. Let $b =
    (x_b,f(x_b)) \in M$ satisfy $b \ne a$ and $|y-b| \le |y-a| = t_0$. Then $b
    \in \cball{a+t_0\nu}{t_0} \subseteq \{a\} \cup \oball{a+\rho\nu}{\rho}$, so $b
    \in \oball{a+\rho\nu}{\rho}$; hence, $|b-a| < 2\rho$ and $|x_b-x_a| \le |b-a| <
    2\rho$.

    Since $\boldsymbol\pi > 2$ and $|x_a - c| + \boldsymbol\pi\rho \le r$, both
    $[x_a,x_b] \subseteq \cball{x_a}{2\rho} \subseteq \cball cr$ and
    $\cball{x_a}{\boldsymbol\pi\rho} \subseteq \cball cr$. Thus the graph of the
    segment $t \mapsto x_a + t(x_b-x_a)$ lies in~$M$ and shows, writing $d_M$ for
    the intrinsic metric of~$M$,
    \begin{equation*}
        d_M(a,b) \le |x_b-x_a| \sqrt{1+\gamma^2} < 2\sqrt{1+\gamma^2}\,\rho < 2\sqrt 2\,\rho < \boldsymbol\pi\rho \,.
    \end{equation*}
    The graph $N$ of $f|\cball{x_a}{\boldsymbol\pi\rho}$ is compact and, by the
    segment-graph above, rectifiably connected; hence, the infimum of the lengths
    of paths in~$N$ joining $a$ to~$b$ is attained, by the Arzel\`a--Ascoli
    theorem applied to a~minimising sequence of arclength parametrised, and thus
    $1$~Lipschitz, maps.  Let $\beta : [0,\ell] \to N$ be such a~shortest path,
    parametrised by arclength, so that $|\beta'| = 1$ almost everywhere.
    (Hopf--Rinow, \cite[Chapter~7, Theorem~2.8]{Carmo1992}, is not available
    here: $N$ is a~manifold with boundary, and a~shortest path in such a~set
    might a~priori run along the boundary; it is the excursion estimate below
    that excludes this.)  As the segment-graph above lies in~$N$, $\ell <
    \boldsymbol\pi\rho$. Paths in~$N$
    issuing from~$a$ have base excursion at most their length, so $\beta$ stays
    over $\oball{x_a}{\boldsymbol\pi\rho}$, that is, in the interior of~$N$;
    being a~shortest path there, it is a~geodesic of~$M$, in particular of
    class~$\cnt{2}$, with
    \begin{equation*}
        |\beta''| = |\langle \beta' \odot \beta' ,\, \mathbf{b}(M,\beta) \rangle| \le \| \mathbf{b}(M,\beta) \| \le \Delta \le \rho^{-1} \,;
    \end{equation*}
    moreover $\beta(0) = a$ and $\beta'(0) \in \Tan(M,a)$, so $\beta'(0) \bullet
    \nu = 0$. By~\eqref{eq:cor:reach:bc} with $L = \ell$, $b = \beta(\ell) \notin
    \oball{a+\rho\nu}{\rho}$, a~contradiction. Therefore, $a$ is the unique nearest
    point of~$M$ to~$y$.

    Finally, if $r = \infty$ the condition $|x_a - c| \le r - \boldsymbol\pi\rho$
    holds vacuously, so every $y$ with $\dist(y,M) < \rho$ has a~unique nearest
    point in~$M$; taking $\rho = \Delta^{-1}$ gives $\reach(M) \ge \Delta^{-1}$.
\end{proof}

\section{Height controls the tilt}
\label{sec:height-tilt}

\begin{miniremark}
    \label{mr:setup-for-height}
    In this section we shall study the following situation. The set $M \subseteq
    \R^{\adim}$ shall be a~$\vdim$-dimensional submanifold of~$\R^{\adim}$ of
    class~$\cnt{2}$. The map $\tau : M \to \Hom(\R^{\adim},\R^{\adim})$ shall be defined by
    the formula $\tau(x) = \project{\Tan(M,x)}$ and the map $\nu : M \to
    \Hom(\R^{\adim},\R^{\adim})$ by $\nu(x) = \perpproject{\Tan(M,x)}$. Moreover, we~define
    the number
    \begin{displaymath}
        \kappa_M  = \sup \bigl\{ |\langle v, \uD \tau(y) u \rangle| : y \in M ,\, u \in \Tan(M,y) ,\, v \in \R^{\adim} ,\, |u| = |v| = 1 \bigr\}
    \end{displaymath}
    and assume $\kappa_M < \infty$.
\end{miniremark}

\begin{remark}
    \label{rem:kappaM-and-reach}
    Observe that for $x \in M$, $u,w \in \Tan(M,x)$, and $v \in \R^{\adim}$
    \begin{gather}
        \tau(x) = \uD \npp{M}(x)|\Tan(M,x) \,,
        \quad
        \bigl\langle u ,\, \uD \tau(x) w \bigr\rangle \in \Nor(M,x) \,,
        \\
        \bigl\langle v ,\, \uD \tau(x) u \bigr\rangle
        = \bigl\langle u \odot v ,\, \uD^2\npp{M}(x) \bigr\rangle \,,
        \quad
        \tau(x) \text{ is self-adjoint } \,;
    \end{gather}
    hence,
    \begin{multline}
        \bigl\langle v ,\, \uD \tau(x) u \bigr\rangle \bullet w
        = \uD \bigl[ M \ni y \mapsto \langle v ,\, \tau(y) \rangle \bullet w \bigr](x)u
        = \uD \bigl[ M \ni y \mapsto v \bullet \langle w ,\, \tau(y) \rangle \bigr](x)u
        \\
        = v \bullet \bigl\langle w , \uD \tau(x)u \bigr\rangle 
        = v \bullet \bigl\langle u \odot w , \uD^2\npp{M}(x) \bigr\rangle
        = v \bullet \bigl\langle u , \uD \tau(x)w \bigr\rangle  \,.
    \end{multline}
    Using~\cite[Lemma~3.1]{KolasnskiMenne} and~\cite[4.18]{Federer1959} we see
    that
    \begin{displaymath}
        \kappa_M
        = \sup \bigl\{ \|\mathbf{b}(M,x)\| : x \in M \bigr\}
        \,.
    \end{displaymath}
\end{remark}

\begin{lemma}
    \label{lem:ht:Dalpha}
    Let $M$, $\tau$, $\nu$, $\kappa_M$ be as in~\ref{mr:setup-for-height}.
    Define
    \begin{gather}
        U = \R^{\adim} \cap \bigl\{ x : \dist(x,M) < \min \{ \kappa_M^{-1} ,\, \reach(M) \} \bigr\} \,,
        \quad
        \beta,\gamma :  U \to \Hom(\R^{\adim},\R^{\adim}) \,,
        \\
        \alpha(x) = x - \npp{M}(x) \quad \text{for $x \in U$} \,,
        \quad
        \beta(x) = \tau(\npp{M}(x)) - \gamma(x)
        \quad \text{for $x \in U$} \,,
        \\
        \gamma(x) = \uD \left[ M \ni y \mapsto \tau(y) \alpha(x) \right](\npp{M}(x)) \circ \tau(\npp{M}(x))
         \quad \text{for $x \in U$} \,.
    \end{gather}
    Since $\dist(x,M) < \reach(M)$ for $x \in U$, the nearest point projection
    $\npp{M}$ is defined on~$U$ and of class~$\cnt{1}$ there,
    by~\cite[4.8]{Federer1959}; the constraint $\dist(x,M) < \reach(M)$ is not
    implied by $\kappa_M \dist(x,M) < 1$, since $\reach(M) \le \kappa_M^{-1}$
    may be strict.  Then
    \begin{displaymath}
        \uD \alpha(x) u = - \beta(x)^{-1} \circ \gamma(x) u + \nu(\npp{M}(x))u
        \quad \text{for $x \in U$ and $u \in \R^{\adim}$}\,,
    \end{displaymath}
    where $\beta(x)^{-1}$ denotes the inverse of the automorphism
    $\beta(x)|\Tan(M,\npp{M}(x))$ and $\im \gamma(x) \subseteq
    \Tan(M,\npp{M}(x))$.
\end{lemma}

\begin{proof}
    Since, by~\cite[4.8(2)]{Federer1959},
    \begin{displaymath}
        \langle \alpha(x), \tau(\npp{M}(x)) \rangle = 0
        \quad \text{for $x \in U$}
    \end{displaymath}
    we obtain for $x \in U$ and $u \in \Tan(M,\npp{M}(x))$
    \begin{multline}
        0 = \left\langle \uD \alpha(x) u, \tau(\npp{M}(x)) \right\rangle
        + \left\langle \alpha(x), \uD \tau(\npp{M}(x)) \circ \uD \npp{M}(x) u \right\rangle
        \\
        = \left\langle \uD \alpha(x) u, \tau(\npp{M}(x)) \right\rangle
        - \left\langle \alpha(x), \uD \tau(\npp{M}(x)) \circ \uD \alpha(x) u \right\rangle
        + \left\langle \alpha(x), \uD \tau(\npp{M}(x)) u \right\rangle
        \\
        = \left\langle \uD \alpha(x) u, \tau(\npp{M}(x)) - \uD \left[ M \ni y \mapsto \tau(y) \alpha(x) \right](\npp{M}(x)) \right\rangle
        + \left\langle \alpha(x), \uD \tau(\npp{M}(x)) u \right\rangle \,;
    \end{multline}
    hence, observing that $\beta(x)$ is an automorphism of $\Tan(M,\npp{M}(x))$ for
    $x \in U$ (because $|\alpha(x)| < \kappa_M^{-1}$) we get
    \begin{displaymath}
        \uD \alpha(x) u
        = - \left\langle \left\langle \alpha(x), \uD \tau(\npp{M}(x)) u \right\rangle , \beta(x)^{-1} \right\rangle
        = - \beta(x)^{-1} \circ \gamma(x) u
        \,.
    \end{displaymath}
    Clearly
    \begin{displaymath}
        \uD \alpha(x) u = u
        \quad \text{for $x \in U$ and $u \in \Tan(M,\npp{M}(x))^{\perp}$} 
    \end{displaymath}
    and the conclusion follows.
\end{proof}

\begin{remark}
    \label{rem:ht:inv-taylor}
    Let~$M$ and~$\kappa_M$ be as in~\ref{mr:setup-for-height}, $U$, $\alpha$,
    $\beta$, and~$\gamma$ as in~\ref{lem:ht:Dalpha}. Since $\|\gamma(x)\| \le
    \kappa_M |\alpha(x)| < 1$ for $x \in U$ we have
    \begin{displaymath}
        \beta(x)^{-1} \circ \gamma(x) = \textsum{j=1}{\infty} \gamma(x)^j
        \quad \text{for $x \in U$}\,.
    \end{displaymath}
\end{remark}

\begin{lemma}
    \label{lem:ht:comput}
    Let $M$, $\tau$, and~$\kappa_M$ be as in~\ref{mr:setup-for-height}, $U$
    and~$\alpha$ as in~\ref{lem:ht:Dalpha}. Suppose $G \subseteq U$ is open, $0
    < A < \infty$, $\varphi \in \dspace{G}{\R}$ satisfies $|\varphi| \le 1$, and
    \begin{displaymath}
        \sup \bigl\{ \| \uD \varphi(x) \| : x \in G \bigr\} \le A \,.
    \end{displaymath}
    Then
    \begin{multline}
        \varphi(x)^2 \| \project{S} - \tau(\npp{M}(x)) \|^2
        \le \uD g(x) \bullet \project{S}
        + \varphi(x) |\alpha(x)| \cdot \| \project{S} - \tau(\npp{M}(x)) \|
        \left(
            2 A + \adim \kappa_M
        \right)
        \\
        + \varphi(x)^2 | \mathbf{h}(M,\npp{M}(x)) |\cdot |\alpha(x)|
        + \frac{\vdim \kappa_M^2 |\alpha(x)|^2}{1 - \kappa_M |\alpha(x)|}
        \,.
    \end{multline}
\end{lemma}

\begin{proof}
    We~define $g : G \to \R^{\adim}$ by the formula $g(x) = \varphi(x)^2 \alpha(x)$
    and compute, as in~\cite[8.13]{Allard1972} (although, in our case the
    set~$U$ plays a different role than in~\cite{Allard1972}), for $x \in G$ and
    $S \in \grass{\adim}{\vdim}$ using~\ref{lem:ht:Dalpha}
    and~\ref{rem:ht:inv-taylor}
    \begin{displaymath}
        \uD g(x) \bullet \project{S} 
        = 2 \varphi(x) \grad \varphi(x) \bullet \project{S} \alpha(x)
        + \varphi(x)^2 \project{S} \bullet \nu(\npp{M}(x))
        - \varphi(x)^2 \textsum{j=1}{\infty} \project{S} \bullet \gamma(x)^j \,;
    \end{displaymath}
    therefore, by~\cite[8.9(3)]{Allard1972}, 
    \begin{multline}
        \varphi(x)^2 \| \project{S} - \tau(\npp{M}(x)) \|^2
        \le \varphi(x)^2 \project{S} \bullet \nu(\npp{M}(x))
        \\
        \le \uD g(x) \bullet \project{S}
        + 2 A \varphi(x) |\alpha(x)| \cdot \| \project{S} - \tau(\npp{M}(x)) \|
        + \varphi(x)^2 |\project{S} \bullet \gamma(x)| 
        + \vdim \textsum{j=2}{\infty} \| \gamma(x)^j \| \,;
    \end{multline}
    here we~used that $\gamma(x)$ is a~symmetric endomorphism with image
    in~$\Tan(M,\npp{M}(x))$ --- see the computation below --- so that, writing
    $\lambda_1,\ldots,\lambda_{\vdim}$ for its eigenvalues and
    $e_1,\ldots,e_{\vdim}$ for a~corresponding orthonormal basis
    of~$\Tan(M,\npp{M}(x))$,
    \begin{displaymath}
        | \project{S} \bullet \gamma(x)^j |
        = \Bigl| \textsum{i=1}{\vdim} \lambda_i^j \, e_i \bullet \project{S} e_i \Bigr|
        \le \vdim \, \| \gamma(x) \|^j
        \quad \text{for $j \ge 1$} \,,
    \end{displaymath}
    since $0 \le e_i \bullet \project{S} e_i \le 1$.
    For $x \in G$ and $u,v \in \Tan(M,\npp{M}(x))$ we have
    \begin{multline}
        \gamma(x) u \bullet v
        = \langle \alpha(x) , \uD \tau(\npp{M}(x)) u \rangle \bullet v
        = \uD \left[ M \ni y \mapsto \tau(y) \alpha(x) \right](\npp{M}(x)) u \bullet v
        \\
        = \uD \left[ M \ni y \mapsto \tau(y) \alpha(x) \bullet v \right](\npp{M}(x)) u
        = \uD \left[ M \ni y \mapsto \alpha(x) \bullet \tau(y) v \right](\npp{M}(x)) u
        \\
        = \langle v , \uD \tau(\npp{M}(x)) u \rangle \bullet \alpha(x)
        = \mathbf{b}(M,\npp{M}(x))(u,v) \bullet \alpha(x) \,;
    \end{multline}
    hence,
    \begin{multline}
        | \gamma(x) \bullet \project{S} |
        \le | \gamma(x) \bullet \tau(\npp{M}(x)) | + | \gamma(x) \bullet (\project{S} - \tau(\npp{M}(x))) |
        \\
        \le | \mathbf{h}(M,\npp{M}(x)) | \cdot |\alpha(x)| + \adim \kappa_M | \alpha(x) | \cdot \| \project{S} - \tau(\npp{M}(x)) \| 
    \end{multline}
    and we obtain
    \begin{multline}
        \varphi(x)^2 \| \project{S} - \tau(\npp{M}(x)) \|^2
        \le \uD g(x) \bullet \project{S}
        + 2 A \varphi(x) |\alpha(x)| \cdot \| \project{S} - \tau(\npp{M}(x)) \|
        \\
        + \varphi(x)^2 | \mathbf{h}(M,\npp{M}(x)) |\cdot |\alpha(x)|
        + \varphi(x)^2 \adim \kappa_M | \alpha(x) | \cdot \| \project{S} - \tau(\npp{M}(x)) \|
        + \frac{\vdim \kappa_M^2 |\alpha(x)|^2}{1 - \kappa_M |\alpha(x)|} 
        \\
        \le \uD g(x) \bullet \project{S}
        + \varphi(x) |\alpha(x)| \cdot \| \project{S} - \tau(\npp{M}(x)) \|
        \left(
            2 A + \adim \kappa_M
        \right)
        \\
        + \varphi(x)^2 | \mathbf{h}(M,\npp{M}(x)) |\cdot |\alpha(x)|
        + \frac{\vdim \kappa_M^2 |\alpha(x)|^2}{1 - \kappa_M |\alpha(x)|}
    \end{multline}
    as required.
\end{proof}

\begin{corollary}
    \label{cor:ht:tleh}
    Let $M$, $\tau$, and~$\kappa_M$ be as in~\ref{mr:setup-for-height}, $U$ as
    in~\ref{lem:ht:Dalpha}, $G$, $\varphi$, and~$A$ as in~\ref{lem:ht:comput},
    $V \in \Var{\vdim}(G)$. Assume
    \begin{gather}
        \mathbf{h}(M,x) = 0 \quad \text{for all $x \in M$} \,,
        \quad
        \Gamma = 2 (2 A + \adim \kappa_M) \,,
        \\
        \delta V = 0 \,,
        \quad
        \text{and} \quad
        \kappa_M \dist(x,M) \le \tfrac 12 \quad \text{for $x \in G$} \,.
    \end{gather}
    Then
    \begin{displaymath}
        \left( \int \varphi(x)^2 \| \project{S} - \tau(\npp{M}(x)) \|^2 \ud V(x,S) \right)^{1/2}
        \le \Gamma
        \left( \int \dist(x,M)^2 \ud \|V\|(x) \right)^{1/2}  \,.
    \end{displaymath}
\end{corollary}

\begin{proof}
    Let $\alpha$ be as in~\ref{lem:ht:comput} and put $g = \varphi^2\alpha$.
    Since $\varphi \in \dspace{G}{\R}$ and $\alpha$ is of class~$\cnt{1}$
    on~$U \supseteq G$, we~have $g \in \dspace{G}{\R^{\adim}}$; whence,
    \begin{displaymath}
        \textint{}{} \uD g(x) \bullet \project{S} \ud V(x,S) = \delta V(g) = 0 \,;
    \end{displaymath}
    this is the only use of the hypothesis $\delta V = 0$.  Moreover
    $\mathbf{h}(M,\npp{M}(x)) = 0$ for $x \in G$ by hypothesis, and $\kappa_M
    |\alpha(x)| \le \tfrac 12$ gives $(1 - \kappa_M|\alpha(x)|)^{-1} \le 2$.
    We~may therefore employ~\ref{lem:ht:comput} and H{\"o}lder's inequality to
    get
    \begin{multline}
        \int \varphi(x)^2 \| \project{S} - \tau(\npp{M}(x)) \|^2 \ud V(x,S)
        \\
        \le (2 A + \adim \kappa_M)
        \left( \int \varphi(x)^2 \| \project{S} - \tau(\npp{M}(x)) \|^2 \ud V(x,S) \right)^{1/2}
        \left( \int |\alpha(x)|^2 \ud \|V\|(x) \right)^{1/2} 
        \\
        + 2 \vdim \kappa_M^2 \int |\alpha(x)|^2 \ud \|V\|(x) \,.
    \end{multline}
    If
    \begin{displaymath}
        2 \vdim \kappa_M^2 \int |\alpha(x)|^2 \ud \|V\|(x) \le \frac 12 \int \varphi(x)^2 \| \project{S} - \tau(\npp{M}(x)) \|^2 \ud V(x,S) \,,
    \end{displaymath}
    then
    \begin{displaymath}
        \left( \int \varphi(x)^2 \| \project{S} - \tau(\npp{M}(x)) \|^2 \ud V(x,S) \right)^{1/2}
        \le 2 (2 A + \adim \kappa_M)
        \left( \int |\alpha(x)|^2 \ud \|V\|(x) \right)^{1/2}  \,.
    \end{displaymath}
    Otherwise,
    \begin{displaymath}
        \int \varphi(x)^2 \| \project{S} - \tau(\npp{M}(x)) \|^2 \ud V(x,S)  \le 4 \vdim \kappa_M^2 \int |\alpha(x)|^2 \ud \|V\|(x) \,,
    \end{displaymath}
    and $2 \vdim^{1/2} \kappa_M \le 2 \adim \kappa_M \le \Gamma$.
    In any case, recalling $\dist(x,M) = |\alpha(x)|$ for $x \in G$, we obtain the desired estimate.
\end{proof}

\begin{remark}
    \label{rem:tilt-height}
    Recall~\ref{rem:kappaM-and-reach} and assume
    \begin{gather}
        a \in \R^{\adim} \,,
        \quad
        T \in \grass{\adim}{\vdim} \,,
        \quad
        0 < \varepsilon < 1 \,,
        \quad
        G = \cylinder Ta{2r}{2r} \,,
        \quad
        \varphi \in \dspace{G}{\R} \,,
        \\
        \varphi(x) = 1
        \quad \text{for $x \in \cylinder Tarr$} \,,
        \quad
        r^{-1} \le \Lip \varphi \le r^{-1} + \varepsilon \,,
        \\
        G \subseteq \R^{\adim} \cap \bigl\{ x : \dist(x,M) \le \tfrac 12 \reach(M) \bigr\} \,.
    \end{gather}
    Then $r \le \tfrac 14 \reach(M) \le (4\kappa_M)^{-1}$.  Indeed, $\oball
    a{2r} \subseteq G$; picking $y \in M$ with $|a-y| = \dist(a,M)$ and a~unit
    vector $\nu \in \Nor(M,y)$ with $a = y + \dist(a,M)\nu$ when $a \notin M$,
    and any unit $\nu \in \Nor(M,a)$ when $a \in M$, we~have $\dist(y+t\nu,M) =
    t$ for $0 \le t < \reach(M)$ by~\cite[4.8]{Federer1959}; as the points
    $y + t\nu$ with $\dist(a,M) \le t < \dist(a,M) + 2r$ lie in $\oball a{2r}
    \subseteq G$, the hypothesis forces $\min\{\dist(a,M)+2r,\reach(M)\} \le
    \tfrac 12 \reach(M)$; hence, $2r \le \tfrac 12 \reach(M)$.  Since $\reach(M)
    \le \kappa_M^{-1}$ by~\ref{rem:kappaM-and-reach}
    and~\cite[4.18]{Federer1959}, this gives $2\kappa_M \le r^{-1}$ and therefore
    $\Gamma_{\ref{cor:ht:tleh}} = 2(2A_{\ref{cor:ht:tleh}} + \adim\kappa_M) \le
    (4+\adim) A_{\ref{cor:ht:tleh}} \le (4+\adim)(r^{-1} + \varepsilon)$.  In~consequence,
    since $\varepsilon$ can be arbitrarily small \ref{cor:ht:tleh}~yields
    \begin{multline}
        \Bigl( \textint{\cylinder Tarr \times \grass{\adim}{\vdim}}{}  \| \project{S} - \tau(\npp{M}(x)) \|^2 \ud V(x,S) \Bigr)^{1/2}
        \\
        \le \frac{4+\adim}{r}
        \Bigl( \textint{\cylinder Ta{2r}{2r}}{} \dist(x,M)^2 \ud \|V\|(x) \Bigr)^{1/2}  \,.
    \end{multline}
\end{remark}

\section{Comparison of two different heights}
\label{sec:two-heights}

\begin{remark}
    \label{rem:reach-of-a-graph}
    Assume $P,Q \in \grass{\adim}{\vdim}$, $p,q \in \orthproj{\adim}{\vdim}$,
    $\im p^* = P$, $\im q^* = Q$, $f : \R^{\vdim} \to \R^{\adim}$ is of
    class~$\cnt{1}$, $f(0) = 0$, $p \circ f = \id{\R^{\vdim}}$, $Q = \im \uD
    f(0)$, $\Sigma = \im f$, $\bar{F} = q \circ f : \R^{\vdim} \to
    \R^{\vdim}$.  The construction of $F$, $V$, and~$g$, the formulae for $\uD
    g$ and $\uD F^{-1}$, and every estimate below in which only first
    derivatives occur require no more than $f \in \cnt{1}$; the displays
    involving $\uD^2f$, $\uD^2g$, $\uD^2F^{-1}$, $\Lip \uD g$, $\mathbf b(\Sigma,
    \cdot)$, and $\reach(\Sigma)$ presuppose that $f$ be of class~$\cnt{2}$, and
    are used only under that hypothesis. Then $\im \uD \bar{F}(0) = \im ( q \circ \uD f(0) ) =
    \R^{\vdim}$; hence, the inverse function theorem
    (see~\cite[3.1.1]{Federer1969}) ensures the existence of an~open
    neighbourhood $U \subseteq \R^{\vdim}$ of~$0$ such that $\bar{F}|U$ is
    invertible with inverse of class~$\cnt{1}$. Set $F = \bar F|U$, $V = \im F$
    and $g = f \circ F^{-1} : V \to \R^{\adim}$. Observe that $F(0) = 0$ so that
    $V$ is an open neighbourhood of~$0$ and $f \lIm U \rIm = g \lIm V \rIm
    \subseteq \Sigma$; thus, the graphs of $(f - p^*)|U$ and $(g - q^*)|V$
    coincide. Moreover, $\uD (g - q^*)(0) = 0$ so
    \begin{displaymath}
        \mathbf{b}(\Sigma,0)(q^*u,q^*v) = \uD^2 g(0)(u,v)
        \quad \text{for $u,v \in \R^{\vdim}$} \,,
    \end{displaymath}
    where $\mathbf{b}(\Sigma,\cdot)$ denotes the second fundamental form
    of~$\Sigma$; cf.~\cite[2.5(1)]{Allard1972}. Let $\Phi :
    \mathbf{GL}(\vdim,\R) \to \mathbf{GL}(\vdim,\R)$ be given by $\Phi(A) =
    A^{-1}$ for $A \in \mathbf{GL}(\vdim,\R)$. Then
    \begin{displaymath}
        \uD \Phi(T)S = - T^{-1} \circ S \circ T^{-1}
        \quad \text{for $S \in \Hom(\R^{\vdim},\R^{\vdim})$ and $T \in \mathbf{GL}(\vdim,\R)$} \,;
    \end{displaymath}
    cf.~\cite[3.1.11, p.~220]{Federer1969}. For $u,v \in \R^{\vdim}$, $x \in U$,
    and $y = F(x)$ straightforward computations show
    \begin{gather}
        \uD g(y) = \uD f(F^{-1}(y)) \circ \uD F^{-1}(y) \,,
        \\
        \uD^2g(y)(u,v) = \uD^2f(x) \bigl( \uD F^{-1}(y)u, \uD F^{-1}(y)v \bigr)
        + \uD f(x) \bigl( \uD^2 F^{-1}(y)(u,v) \bigr) \,,
        \\
        \uD F^{-1}(y) = \Phi\bigl( \uD F \circ F^{-1} (y) \bigr) = \uD F(x)^{-1} \,,
        \\
        \uD^2 F^{-1}(y)(u,v)
        = - \uD F(x)^{-1} \circ  \uD^2F(x) \bigl( \uD F(x)^{-1} v, \uD F(x)^{-1} u \bigr)
        \,;
    \end{gather}
    hence, setting $a = \uD F(x)^{-1} u$ and $b = \uD F(x)^{-1} v$
    \begin{gather}
        \uD g(y) u = \uD f(x) a \,,
        \\
        \uD^2g(y)(u,v) = \uD^2f(x)(a,b) - \uD f(x) \circ (q \circ \uD f(x))^{-1} \circ q \circ \uD^2 f(x)(a,b) \,.
    \end{gather}

    Set $h = f - p^* : \R^{\vdim} \to P^{\perp}$ and assume $\Lip h \le \gamma$
    and $\Lip \uD h \le s^{-1}$ for some $\gamma \in (0,1)$ and $s \in
    (0,\infty)$. Using~\cite[4.2, 4.3]{KolRMI} and~\cite[8.9(3)]{Allard1972}
    we~see that $\| \project{Q} \circ \perpproject{P} \|^2 \le \frac{\gamma^2}{1
      + \gamma^2}$; hence,
    \begin{gather}
        | q \circ p^*(u) |^2 = |u|^2 - | \perpproject{Q} \circ \project{P} p^*(u) |^2 \ge |u|^2 \frac{1}{1 + \gamma^2}
        \quad \text{for $u \in \R^{\vdim}$}\,,
        \\
        |q \circ \uD f(x) u|
        \ge |q \circ p^* u| - | q \circ \uD h(x) u|
        \ge |u| \frac{1 - \gamma^2}{\sqrt{1+\gamma^2}}
        \quad \text{for $x \in U$, $u \in \R^{\vdim}$}\,,
        \\
        \Lip (g - q^*) \le \gamma \frac{1+ \sqrt{1+\gamma^2}}{1 - \gamma^2} \le \gamma \frac{4 + \gamma^2}{2 - 2 \gamma^2} 
    \end{gather}
    and for $x \in U$, $a,b \in \R^{\vdim}$, $y = F(x)$, $u = \uD F(x) a$, $v =
    \uD F(x) b$, writing $\uD f(x) = \perpproject{Q} \circ \uD f(x) +
    \project{Q} \circ \uD f(x)$, we obtain
    \begin{gather}
        \uD^2 g(y)(u,v) = ( \perpproject{Q} - L ) \circ \uD^2f(x)(a,b) \,,
        \quad \text{where} \quad
        L = \perpproject{Q} \circ \uD f(x) \circ ( q \circ \uD f(x) )^{-1} \circ q \,,
        \\
        \begin{multlined}
            \|L\| = \|L^*\| = \sup \bigl\{ |z| : z \in Q ,\, w \in Q^{\perp} ,\, |w| = 1 ,\, \uD f(x)^* z = \uD f(x)^* w \bigr\} 
            \\
            = \sup \biggl\{ \frac{| \perpproject{Q} \circ \uD f(x) u |}{| \project{Q} \circ \uD f(x) u |} : u \in \R^{\vdim} ,\, u \ne 0 \biggr\}
            \le \gamma \frac{1 + (1+ \gamma^2)^{-1/2}} { 1 - \gamma^2 ( 1+ \gamma^2)^{-1/2} }
            \le \gamma \bigl( 2 + \tfrac 72 \gamma^2 + \tfrac 38 \gamma^4 \bigr)
            \le 6 \gamma\,,
        \end{multlined}
        \\
        \Lip \uD g \le \frac 1s (1 + 6 \gamma) \frac{1 + \gamma^2}{(1 - \gamma^2)^2} \,.
    \end{gather}
    In~particular, since $L = 0$ at $x = 0$,
    \begin{displaymath}
        \| \mathbf{b}(\Sigma,0) \| = \| \uD^2 g(0) \|
        \le \frac{1+\gamma^2}{s (1 - \gamma^2)^2} \,.
    \end{displaymath}
    The reach of~$\Sigma$ is, however, governed by the second fundamental form
    at \emph{every} point of~$\Sigma$, not only at~$0$. Since $\Sigma$ is the
    graph of $h : \R^{\vdim} \to P^{\perp}$ with $\Lip h \le \gamma < 1$ and
    $\sup_{x} \| \uD^2 h(x) \| = \Lip \uD h \le s^{-1}$, the second fundamental
    form of the graph obeys $\sup_{z \in \Sigma} \| \mathbf{b}(\Sigma,z) \| \le
    s^{-1}$; hence, identifying $P$ with $\R^{\vdim} \times \{0\}$ by an
    orthogonal transformation and applying~\ref{cor:reach-of-graph} with $h$,
    $\gamma$, $s^{-1}$, $\infty$ in place of $f$, $\gamma$, $\Delta$, $r$
    (so that $\Sigma$ is an~entire graph), we~obtain
    \begin{displaymath}
        \reach(\Sigma) \ge s \ge s \frac{(1 - \gamma^2)^2}{1+\gamma^2} \,.
    \end{displaymath}
\end{remark}

\begin{lemma}
    \label{lem:convex-reach}
    Suppose $T \in \grass{\adim}{\vdim}$, $p \in \orthproj{\adim}{\vdim}$, $q \in
    \orthproj{\adim}{\adim-\vdim}$, $\im p^* = T = \ker q$, $C \subseteq
    \R^{\vdim}$ is convex and closed, $v : C \to \R^{\adim-\vdim}$ is of
    class~$\cnt{2}$, $\Delta \in (0,\infty)$, $\| \uD^2 v(x) \| \le \Delta$ for $x
    \in C$, $w = p^* + q^* \circ v$, and $M = w \lIm C \rIm$. Then the following
    hold.
    \begin{enumerate}
    \item
        \label{i:cr:strict}
        If $y \in \R^{\adim}$ and $\Delta\,|q(y) - v(x)| < 1$ for $x \in C$, then
        the function $x \mapsto |y - w(x)|^2$ is strictly convex on~$C$;
        consequently $y$ has at most one nearest point in~$M$.
    \item
        \label{i:cr:reach}
        If, moreover, $C$ is bounded and $\Lip v \le \gamma < \infty$, then $M$
        is compact and $\reach(M) \ge \Delta^{-1} - \gamma \diam C$.
    \end{enumerate}
\end{lemma}

\begin{proof}
    Since $p^*$ and $q^*$ map $\R^{\vdim}$ and $\R^{\adim-\vdim}$ isometrically
    onto the orthogonal subspaces $T$ and $T^{\perp}$, while $p \circ w =
    \id{\R^{\vdim}}$ and $q \circ w = v$, we~have for $x \in C$
    \begin{displaymath}
        F(x) = |y - w(x)|^2 = |p(y) - x|^2 + |q(y) - v(x)|^2 \,;
    \end{displaymath}
    whence, for $u \in \R^{\vdim}$,
    \begin{align*}
        \uD^2 F(x)(u,u)
        &= 2|u|^2 + 2|\uD v(x)u|^2 - 2 \bigl\langle q(y) - v(x) ,\, \uD^2 v(x)(u,u) \bigr\rangle
        \\
        &\ge 2|u|^2 \bigl( 1 - \Delta\,|q(y) - v(x)| \bigr) \,.
    \end{align*}
    Under the hypothesis of~\ref{i:cr:strict} this is positive for $u \ne 0$, so
    $F$ is strictly convex on the convex set~$C$ and has at most one minimiser;
    as the points of~$M$ nearest to~$y$ are the images $w(x)$ of the minimisers
    of~$F$, item~\ref{i:cr:strict} follows.

    For~\ref{i:cr:reach} note that $C$ is compact; hence, so is $M = w \lIm C
    \rIm$. Let $y \in \R^{\adim}$ with $\dist(y,M) < \Delta^{-1} - \gamma \diam
    C$; a~nearest point $w(x_0)$ of~$M$ to~$y$ exists, and for $x \in C$
    \begin{displaymath}
        |q(y) - v(x)| \le |q(y-w(x_0))| + |v(x_0) - v(x)| \le \dist(y,M) + \gamma \diam C < \Delta^{-1} \,,
    \end{displaymath}
    so that~\ref{i:cr:strict} yields a~unique nearest point. Thus $\{ y :
    \dist(y,M) < \Delta^{-1} - \gamma \diam C \} \subseteq \Unp(M)$, that is,
    $\reach(M) \ge \Delta^{-1} - \gamma \diam C$.
\end{proof}

\begin{lemma}
    \label{lem:heights}
    Suppose
    \begin{gather}
        p \in \orthproj{\adim}{\vdim} \,,
        \quad
        \gamma \in (0,2/3] \,,
        \quad
        f : \R^{\vdim} \to \R^{\adim} \text{ is Lipschitzian} \,,
        \quad
        p \circ f = \id{\R^{\vdim}} \,,
        \quad
        \Sigma = \im f \,,
        \\
        \Lip (f - p^*) \le \gamma \,,
        \quad
        a \in \R^{\adim} \,,
        \quad
        b,c \in \Sigma \,,
        \quad
        |a - c| = \dist(a,\Sigma) \,,
        \quad
        p(a - b) = 0 \,,
        \\
        f|\oball{p(c)}{3|a-c|} \text{ is of class~$\cnt{1}$} \,,
        \\
        \omega = \sup\bigl\{ \| \uD f(x) - \uD f(p(c)) \| :  x \in \R^{\vdim} ,\, |x- p(c)| \le 2|a-c| \bigr\}
        \,.
    \end{gather}
    The differentiability of~$f$ is thus required only near~$p(c)$; the proof
    uses no more, the construction of~$g$ resting on the inverse function
    theorem at~$p(c)$ and on the lower bound for $q \circ f$, which needs
    only~$\Lip(f-p^*) \le \gamma$.
    Then
    \begin{displaymath}
        |a-c| \le |a-b| \le \Gamma |a-c|
        \quad \text{and} \quad
        |b-c| \le \gamma \sqrt{1 + \gamma^2} |a-c|
        \,,
    \end{displaymath}
    where $\Gamma = (1 + \gamma^2)^{1/2} + \gamma \omega (1+\gamma^2)^{3/2} (1 - \gamma^2)^{-1}$.
\end{lemma}

\begin{proof}
    Clearly $|a-c| \le |a-b|$ because $c$ is a~nearest point of~$\Sigma$
    to~$a$; the estimate for $|b-c|$ will be obtained
    in~\eqref{eq:x-c-est} and~\eqref{eq:y-c-est}, so it remains to
    bound~$|a-b|$.
    
    Without loss of generality we may and shall assume that $f(0) = c = 0$. Set
    \begin{displaymath}
        P = \im p^* \,,
        \quad
        Q = \Tan(\Sigma,c) \,,
        \quad \text{and} \quad
        h = f - p^* : \R^{\vdim} \to P^{\perp} \,.
    \end{displaymath}
    Employ~\ref{rem:reach-of-a-graph} to find $U,V \subseteq \R^{\vdim}$, $q \in
    \orthproj{\adim}{\vdim}$, and $g : V \to \R^{\adim}$ of class~$\cnt{1}$ such
    that $\im q^* = Q$, $q \circ g = \id{\R^{\vdim}}$, and $f \lIm U \rIm =
    g\lIm V \rIm$; thus $\Sigma$ is a~graph of~$(f-p^*)$ over~$P$ and, locally,
    a~graph of~$(g - q^*)$ over~$Q$.  Application of~\cite[4.2, 4.3]{KolRMI} yields
    \begin{displaymath}
        \| \project{P} - \project{Q} \|^2
        = \frac{\|\uD h(0)\|^2}{1 + \|\uD h(0)\|^2} 
        \le \frac{\gamma^2}{1 + \gamma^2}  \,.
    \end{displaymath}
    Since $\Lip (f - p^*) < 1$, the map $F = q \circ f$ satisfies the lower
    bound $|F(u)-F(u')| \ge (1-\gamma^2)(1+\gamma^2)^{-1/2}|u-u'|$ established
    below and has everywhere invertible differential; being thus univalent and
    proper, it maps $\R^{\vdim}$ onto $\R^{\vdim}$.  Consequently $\Sigma$ is
    a~graph over~$Q$ in the sense that $(z+Q^{\perp}) \cap \Sigma = \{z\}$ for
    $z \in \Sigma$, and we can define~$g$ globally on~$\R^{\vdim}$. In the~sequel we
    shall assume $U = V = \R^{\vdim}$. Note also that $g(0) = 0$ and $\uD (g -
    q^*)(0) = 0$.

    Employing~\cite[4.8(2)]{Federer1959} we see that $a-c \in Q^{\perp}$ and we
    know that $a-b \in P^{\perp}$. Define $x,y \in Q$ so that
    \begin{displaymath}
        x-c = \project{Q}(b-c)
        \quad \text{and} \quad
        Q \cap (a + P^{\perp}) = \{y\} \,.
    \end{displaymath}
    In~other words $y$ is the unique point in~$Q$ such that
    $\perpproject{Q}(y-a) = c-a$. Observe that $\{b\} = \Sigma \cap (a +
    P^{\perp})$ and $b-a \in P^{\perp}$ so $(a + P^{\perp}) = (b + P^{\perp})$;
    hence, $\{y\} = Q \cap (b + P^{\perp})$. This shows that~$y$ is also the
    unique point in~$Q$ satisfying $\perpproject{Q}(y-b) = x-b$.
    We estimate
    \begin{multline}
        \label{eq:a-c-est}
        |a-c|^2 = |\perpproject{Q}(y-a)|^2
        = |y-a|^2 - | \project{Q} \circ \perpproject{P}(y-a)|^2
        \\
        \ge |y-a|^2 \bigl( 1 - \| \project{Q} - \project{P} \|^2 \bigr)
        \ge \frac{|y-a|^2}{1 + \gamma^2}\,;
    \end{multline}
    here we~used $\perpproject{P}(y-a) = y-a$ and $\| \project{Q} \circ
    \perpproject{P} \| = \| \project{Q} - \project{P} \|$,
    by~\cite[8.9(3)]{Allard1972};
    hence,
    \begin{gather}
        \label{eq:y-a-est}
        |y-a|^2
        \le |a-c|^2 (1 + \gamma^2)
        \\
        \label{eq:y-c-est}
        \text{and}\quad
        |y-c|^2 = |y-a|^2 - |a-c|^2
        \le \gamma^2 |a-c|^2 \,.
    \end{gather}
    The computation of~\eqref{eq:a-c-est}, carried out with $b$, $x$ in place
    of $a$, $c$ --- legitimate because $\perpproject{Q}(y-b) = x-b$ and $y-b \in
    P^{\perp}$ --- similarly gives
    \begin{equation}
        \label{eq:b-y-est}
        |b-y|^2
        \le |b-x|^2 \bigl( 1 - \| \project{Q} - \project{P} \|^2 \bigr)^{-1}
        \le |b-x|^2 (1 + \gamma^2) \,.
    \end{equation}
    Since $b-y \in P^{\perp}$ we also have
    \begin{equation}
        \label{eq:x-c-est}
        |x-c| \le |b-c| = | f \circ p(y) - f \circ p(c)|
        \le \Lip f |y-c| \le (1+\gamma^2)^{1/2} |y-c| \,.
    \end{equation}
    Recall~\ref{rem:reach-of-a-graph} to see that if $w,z \in \R^{\vdim}$, $|w|
    \le 2|a-c|$, and $q \circ f(w) = z$, then
    \begin{multline}
        \| \uD g(z) - q^* \|
        = \bigl\| \perpproject{Q} \circ \uD f(w) \circ (q \circ \uD f(w))^{-1} \bigr\|
        \\
        \le \frac{ (1+\gamma^2)^{1/2} }{1 - \gamma^2} \| \perpproject{Q} \circ ( \uD f(w) - \uD f(0) ) \|
        \le \frac{ (1+\gamma^2)^{1/2} }{1 - \gamma^2} \omega \,.
    \end{multline}
    The preceding estimate is applicable to every~$z$ occurring
    in~\eqref{eq:b-x-est} below. In~fact, suppose $z,w \in \R^{\vdim}$, $|z| \le
    |x-c|$, and $q \circ f(w) = z$. Setting $F = q \circ f = q \circ p^* + q
    \circ h$ we~have $F(0) = 0$; moreover, $\im \uD h(u) \subseteq P^{\perp}$ and
    $\| \uD h(u) \| \le \gamma$ for $u \in \R^{\vdim}$; whence,
    recalling~\ref{rem:reach-of-a-graph} and $\| \project{Q} \circ
    \perpproject{P} \| \le \gamma (1+\gamma^2)^{-1/2}$,
    \begin{displaymath}
        \Lip (q \circ h) \le \frac{\gamma^2}{(1+\gamma^2)^{1/2}}
        \quad \text{and} \quad
        | q \circ p^*(u) | \ge \frac{|u|}{(1+\gamma^2)^{1/2}}
        \quad \text{for $u \in \R^{\vdim}$} \,;
    \end{displaymath}
    consequently, $F$ is univalent with
    \begin{displaymath}
        |F(u) - F(u')| \ge \frac{1-\gamma^2}{(1+\gamma^2)^{1/2}} |u-u'|
        \quad \text{for $u,u' \in \R^{\vdim}$} \,.
    \end{displaymath}
    Taking $u = w$ and $u' = 0$, and combining with~\eqref{eq:x-c-est}
    and~\eqref{eq:y-c-est}, we~conclude
    \begin{displaymath}
        |w| \le \frac{(1+\gamma^2)^{1/2}}{1-\gamma^2} |z|
        \le \frac{(1+\gamma^2)^{1/2}}{1-\gamma^2} |x-c|
        \le \frac{\gamma(1+\gamma^2)}{1-\gamma^2} |a-c|
        \le 2|a-c| \,,
    \end{displaymath}
    the last inequality because $\gamma \le 2/3$.

    Thus, using the Taylor formula~\cite[3.1.11,p.~220]{Federer1969} and
    recalling that $\uD (g - q^*)(0) = 0$ and $q \circ g = \id{\R^{\vdim}}$
    we~obtain
    \begin{multline}
        \label{eq:b-x-est}
        |b-x| = |(g - q^*)(q(x))|
        = |(g-q^*)(q(x)) - (g-q^*)(q(c))|
        \\
        \le |x-c| \sup\bigl\{ \| \uD g(z) - \uD g(0) \| : |z| \le |x-c| \bigr\}
        \\
        = |x-c| \sup \bigl\{ \| \uD (g - q^*)(z) \| : |z| \le |x-c| \bigr\}
        \le |x-c| \frac{ (1+\gamma^2)^{1/2} }{1 - \gamma^2} \omega \,.
    \end{multline}
    Finally, combining~\eqref{eq:y-a-est}, \eqref{eq:b-y-est},
    \eqref{eq:b-x-est}, \eqref{eq:x-c-est}, and~\eqref{eq:y-c-est}
    we may write
    \begin{displaymath}
        |a-b| \le |a-y| + |y-b|
        \le |a-c|
        \left( (1 + \gamma^2)^{1/2} +  \frac{ (1+\gamma^2)^{3/2} \omega \gamma}{1 - \gamma^2} \right) \,.
        \qedhere
    \end{displaymath}
\end{proof}

\begin{corollary}
    \label{cor:tilt-tilt}
    Adopt the hypotheses and notation of~\ref{lem:heights} and assume in
    addition that $f|\oball{p(c)}{3|a-c|}$ is of class~$\cnt{1,1}$, $\Lip \uD f$
    denoting the Lipschitz constant of $\uD f$ there --- which contains $p(b)$,
    since $|p(b)-p(c)| \le |b-c| \le |a-c|$.  If $T \in \grass{\adim}{\vdim}$
    and $a \in \Unp(\Sigma)$ --- so that $c = \npp{\Sigma}(a)$ and, $\Sigma$
    being a~graph over~$\im p^*$, $b = f(p(a))$ --- then,
    employing~\cite[8.9(5)]{Allard1972},
    \begin{multline}
        \bigl| \| \project{T} - \project{\Tan(\Sigma,\npp{\Sigma}(a))} \|
        - \| \project{T} - \project{\Tan(\Sigma,f(p(a)))} \| \bigr|
        \\
        \le \| \project{\Tan(\Sigma,c)} - \project{\Tan(\Sigma,b)} \|
        \le \| \uD f(p(c)) - \uD f(p(b)) \|
        \le \Lip \uD f \cdot |b-c|
        \\
        \le \Lip \uD f \cdot \gamma \sqrt{1 + \gamma^2} \dist(a,\Sigma) \,.
    \end{multline}
\end{corollary}

\begin{proof}
    The displayed chain consists of four inequalities.  The first is the
    triangle inequality, the third combines the definition of $\Lip \uD f$
    with $\Lip p \le 1$, and the fourth is the second conclusion
    of~\ref{lem:heights}.  For the second, recall $\Tan(\Sigma,f(u)) = \im \uD
    f(u)$ for $u \in \R^{\vdim}$.  Since $p \circ f = \id{\R^{\vdim}}$ we~have $p \circ
    \uD f(u) = \id{\R^{\vdim}}$, whence $|\uD f(u)v| \ge |p \circ \uD f(u)v| =
    |v|$; hence, whenever $u,u',v \in
    \R^{\vdim}$ and $|\uD f(u)v| = 1$ --- so that $|v| \le 1$ ---
    \begin{displaymath}
        \dist \bigl( \uD f(u)v ,\, \im \uD f(u') \bigr)
        \le \bigl| \uD f(u)v - \uD f(u')v \bigr|
        \le \| \uD f(u) - \uD f(u') \| \,,
    \end{displaymath}
    and $\| \project{\im \uD f(u)} - \project{\im \uD f(u')} \| = \|
    \perpproject{\im \uD f(u')} \circ \project{\im \uD f(u)} \|$
    by~\cite[8.9(3)]{Allard1972}, the two planes being members
    of~$\grass{\adim}{\vdim}$.
\end{proof}

\begin{corollary}
    \label{cor:np-in-2cyl}
    Adopt the hypotheses and notation of~\ref{lem:heights} and assume in
    addition that $\gamma \le \frac 14$, that $T = \im p^*$, that $\Sigma \cap
    p^{-1}\lIm \cball 0r \rIm \subseteq \cylinder T0r{\frac 12 r}$ for some $0
    < r < \infty$, and that $a \in \cylinder T0rr \cap \Unp(\Sigma)$.  Then
    \begin{displaymath}
        c = \npp{\Sigma}(a) \in \Sigma \cap p^{-1}\lIm \cball 0{2r} \rIm
        \subseteq \cylinder T0{2r}{r} \,.
    \end{displaymath}
\end{corollary}

\begin{proof}
    Since $\perpproject{T} \circ p^* = 0$ and $\Lip \bigl( \perpproject{T}
    \circ (f - p^*) \bigr) \le \gamma$, the hypothesis applied with $u = 0$
    yields $|\perpproject{T}(f(u))| \le |\perpproject{T}(f(0))| + \gamma |u| \le
    \tfrac 12 r + \tfrac 14 \cdot 2r = r$ whenever $u \in \cball 0{2r}$; this is
    the asserted inclusion $\Sigma \cap p^{-1}\lIm \cball 0{2r} \rIm \subseteq
    \cylinder T0{2r}{r}$.  Next, $b = f(p(a))$ satisfies $p(b) = p(a) \in
    \cball 0r$ and $a - b \in T^{\perp}$; whence, $|\perpproject{T}(b)| \le
    \tfrac 12 r$ by the hypothesis and $|\perpproject{T}(a)| \le r$ because $a
    \in \cylinder T0rr$, so that $|a-b| = |\perpproject{T}(a) -
    \perpproject{T}(b)| \le \tfrac 32 r$.  From~\ref{lem:heights} and $\gamma
    \le \tfrac 14$ we~therefore obtain
    \begin{displaymath}
        |p(c)| \le |p(b)| + |b-c|
        \le r + \tfrac{\sqrt{17}}{16} |a-c|
        \le r + \tfrac 38 |a-b|
        \le r + \tfrac 38 \cdot \tfrac 32 r = \tfrac{25}{16}r < 2r \,.
        \qedhere
    \end{displaymath}
\end{proof}

\section{Some estimates for second order elliptic systems}
\label{sec:elliptic}

\begin{definition}[\protect{cf.~\cite[3.1]{Menne2013}}]
    \label{def:ell-system}
    We employ two ways of defining differential operators.
    \begin{enumerate}
    \item
        \label{i:pde:LF}
        Assume $F : \Hom(\R^{\vdim},\R^{\adim-\vdim}) \to \R \text{ is of class~$\cnt{1}$}$ and
        \begin{displaymath}
            \sup \bigl\{ \| \uD F(\sigma) \| \cdot ( 1 + \| \sigma \| )^{-1} : \sigma \in \Hom(\R^{\vdim},\R^{\adim-\vdim}) \bigr\} < \infty \,.
        \end{displaymath}
        For any open set $U \subseteq \R^{\vdim}$, $\LM^{\vdim}$~measurable
        set~$D \subseteq \R^{\vdim}$, and $u \in \VSob{1,1}(U,\R^{\adim-\vdim})$
        we define the distribution $L_{F;D}(u) \in \distr{U}{\R^{\adim-\vdim}}$
        by the formula
        \begin{displaymath}
            L_{F;D}(u)\theta = - \int_{D \cap U} \bigl\langle \uD \theta(x) ,\, \uD F(\weakD u(x)) \bigr\rangle \ud \LM^{\vdim}(x) 
            \quad \text{for $\theta \in \dspace{U}{\R^{\adim-\vdim}}$} \,.
        \end{displaymath}
        In~case $D = U$, we write $L_F(u) = L_{F;U}(u)$.

    \item
        \label{i:pde:TA}
        Assume $U \subseteq \R^{\vdim}$ is open and
        \begin{displaymath}
            A : U \to {\textstyle \bigodot^2} \Hom(\R^{\vdim},\R^{\adim-\vdim})
            \text{ is $\LM^{\vdim} \restrict U$~measurable and locally bounded} \,.
        \end{displaymath}
        For any $\LM^{\vdim}$~measurable set $D \subseteq \R^{\vdim}$ and $u \in
        \VSob{1,1}(U,\R^{\adim-\vdim})$ we define the distribution $T_{A;D}(u)
        \in \distr{U}{\R^{\adim-\vdim}}$ by
        \begin{displaymath}
            T_{A;D}(u)\theta = - \int_{D \cap U} \bigl\langle \uD \theta(x) \odot \weakD u(x) ,\, A(x) \bigr\rangle \ud \LM^{\vdim}(x)
            \quad \text{for $\theta \in \dspace{U}{\R^{\adim-\vdim}}$} \,.
        \end{displaymath}
        In~case $D = U$, we write $T_A = T_{A;U}$.
    \end{enumerate}
\end{definition}

\begin{remark}
    \label{rem:LF-domain-problem}
    Given $A$ as in~\ref{def:ell-system}\ref{i:pde:TA} we shall assume that $U$
    is the interior of the domain of~$A$; hence, $T_A$ can be regarded as a~map
    of the~type $\VSob{1,1}(U,\R^{\adim-\vdim}) \to
    \distr{U}{\R^{\adim-\vdim}}$. However, given $F$ as
    in~\ref{def:ell-system}\ref{i:pde:LF} the set $U$ is not determined
    by~$F$. If $D \subseteq \R^{\vdim}$ is $\LM^{\vdim}$~measurable, then the
    map $L_{F;D}$ can be assigned the type
    \begin{displaymath}
        \tbcup\{ \VSob{1,1}(U,\R^{\adim-\vdim}) : U \subseteq \R^{\vdim} \text{ is open} \}
        \to \tbcup\{ \distr{U}{\R^{\adim-\vdim}} : U \subseteq \R^{\vdim} \text{ is open} \} \,.
    \end{displaymath}
    Nonetheless, no ambiguity will arise: whenever we write $L_F(u)$ or
    $L_{F;D}(u)$ in the sequel, the open set~$U$ is the one on which the
    argument~$u$ is defined, and where two functions with different domains are
    compared we~restrict them to a~common domain first, as
    in~\ref{thm:decay}.
\end{remark}

\begin{remark}
    \label{rem:operator-PDEs}
    Choose orthonormal bases $e_1, \ldots, e_{\vdim}$ and $v_1, \ldots,
    v_{\adim-\vdim}$ of $\R^{\vdim}$ and $\R^{\adim-\vdim}$ and let $\omega_1,
    \ldots, \omega_{\vdim}$ and $\eta_1, \ldots, \eta_{\adim-\vdim}$ be the dual
    bases. For $i \in \{ 1,2,\ldots,\vdim \}$ let $\uD_i$ denote the directional
    derivative in direction~$e_i$. In~case $F$ and $A$ are as
    in~\ref{def:ell-system}, $A$ is of class~$\cnt{1}$, $F$ is of
    class~$\cnt{2}$, $u,w : U \to \R^{\adim-\vdim}$ are of class~$\cnt{2}$, and
    there exist $\LM^{\vdim}\restrict U$~measurable functions $f,g : U \to
    \R^{\adim-\vdim}$ such that
    \begin{multline}
        T_A(u)\theta = \textint U{} \theta(x) \bullet f(x) \ud \LM^{\vdim}(x)
        \quad \text{and} \quad
        L_F(w)\theta = \textint U{} \theta(x) \bullet g(x) \ud \LM^{\vdim}(x)
        \\ \text{for $\theta \in \dspace{U}{\R^{\adim-\vdim}}$} \,,
    \end{multline}
    then $u$ and $w$ satisfy the following systems of partial differential
    equations
    \begin{align}
        - &\textsum{i,k=1}{\vdim} \textsum{l=1}{\adim-\vdim}
        \uD_i \bigl( (\uD_k u \bullet v_l)
        \langle \omega_iv_j \odot \omega_kv_l ,\, A \rangle  \bigr)(x)
        = f(x) \bullet v_j
        \\
        \text{and} \quad
        - &\textsum{i=1}{\vdim} \uD_i \langle \omega_i v_j ,\, \uD F \circ \uD w \rangle (x) = g(x) \bullet v_j
        \qquad \text{for $x \in U$ and $j \in \{1,\ldots,\adim-\vdim\}$} \,.
    \end{align}
\end{remark}

\begin{remark}
    \label{rem:Laplacian}
    Define
    \begin{gather}
        \Upsilon \in {\textstyle \bigodot^2} \Hom(\R^{\vdim},\R^{\adim-\vdim}) \,,
        \quad
        \Upsilon(\sigma,\tau) = \sigma \bullet \tau \,,
        \\
        F(\sigma) = \tfrac 12 |\sigma|^2 \quad \text{for $\sigma \in \Hom(\R^{\vdim},\R^{\adim-\vdim})$} \,,
        \quad
        A(x) = \Upsilon \quad \text{for $x \in U$} \,.
    \end{gather}
    Then for each $u \in \VSob{1,1}(U,\R^{\adim-\vdim})$ the two distributions
    $L_F(u)$ and $T_A(u)$ are equal and
    \begin{displaymath}
        L_F(u) \theta = T_A(u)\theta = - \int_U \uD \theta(x) \bullet \weakD u(x) \ud \LM^{\vdim}(x)
        \quad
        \text{for $\theta \in \dspace{U}{\R^{\adim-\vdim}}$} \,;
    \end{displaymath}
    hence, both $L_F(u)$ and $T_A(u)$ correspond to the Laplacian of~$u$.
\end{remark}

\begin{definition}[\protect{cf.~\cite[pp.~717-718]{Menne2013}}]
    \label{def:negative-norms}
    Whenever $T \in \distr{U}{\R^{\adim-\vdim}}$, $i$ is a \emph{negative}
    integer, $1 \le p \le \infty$, $1 \le q \le \infty$, $1/p + 1/q = 1$, we set
    \begin{displaymath}
        \norm{T}{i,p}{U} = \sup \bigl\{ T(\theta) : \theta \in \dspace{U}{\R^{\adim-\vdim}} ,\, \norm{\uD^{-i} \theta}{q}{U} \le 1 \bigr\} \,.
    \end{displaymath}
\end{definition}

\begin{definition}[\protect{cf.~\cite[5.2.3]{Federer1969}}]
    \label{def:strong-ellipticity}
    Let $\Xi \in \bigodot^2 \Hom(\R^{\vdim},\R^{\adim-\vdim})$. We say that
    $\Xi$ is \emph{strongly elliptic} if there exists a positive number $c \in
    \R$ such that
    \begin{displaymath}
        \int \bigl\langle \uD \theta(x) \odot \uD \theta(x) ,\, \Xi \bigr\rangle \ud \LM^{\vdim}(x)
        \ge c \int |\uD \theta|^2 \ud \LM^{\vdim}
        \quad \text{for $\theta \in \dspace{\R^{\vdim}}{\R^{\adim-\vdim}}$} \,.
    \end{displaymath}
    Any number $c$ as above is called an~\emph{ellipticity bound} for~$\Xi$.
\end{definition}

\begin{definition}
    \label{def:elliptic-bounds}
    Let $A$, $F$, $T_A$, $L_F$ be as in~\ref{def:ell-system}, and $0 < c \le M
    < \infty$.
    \begin{enumerate}
    \item We~say that $T_A$ is~\emph{strongly elliptic with bounds $(c,M)$} if
        $c$ is an~ellipticity bound for~$A(x)$ and $\|A(x)\| \le M$ for
        $\LM^{\vdim}$ almost all $x \in U$.

    \item We say that $L_F$ is \emph{strongly elliptic with bounds $(c,M)$} if
        $c$ is an~ellipticity bound for~$\uD^2F(\sigma)$ and $\|\uD^2 F(\sigma)\|
        \le M$ for all $\sigma \in \Hom(\R^{\vdim},\R^{\adim-\vdim})$.
    \end{enumerate}
    The bounds are part of the notion; we~never use the phrase without them.
\end{definition}

\begin{remark}
    \label{rem:elliptic-euler-lagrange}
    Let $A$, $F$, $T_A$, $L_F$ be as in~\ref{def:ell-system}, $\Upsilon$ as
    in~\ref{rem:Laplacian}, and $0 < \varepsilon < 1$.
    \begin{enumerate}
    \item Assume $A$ is bounded. Then $T_A$ is the Euler-Lagrange operator
        associated with the functional mapping $u \in
        \VSob{1,2}(U,\R^{\adim-\vdim})$ into
        \begin{displaymath}
            \frac 12 \int_U \bigl\langle \weakD u(x) \odot \weakD u(x) ,\, A(x) \bigr\rangle \ud \LM^{\vdim}(x) \,.
        \end{displaymath}

    \item Strong ellipticity of $T_A$ is ensured by the condition
        \begin{displaymath}
            \| A(x) - \Upsilon \| \le \varepsilon
            \quad \text{for $\LM^{\vdim}$ almost all $x \in U$} \,.
        \end{displaymath}

    \item Assume $\Lip F < \infty$ and $U$ is bounded. Then $L_F$ is the
        Euler-Lagrange operator associated with the functional mapping $u \in
        \VSob{1,1}(U,\R^{\adim-\vdim})$ into
        \begin{displaymath}
            \int_U F(\weakD u(x)) \ud \LM^{\vdim}(x) \,.
        \end{displaymath}
        
    \item Strong ellipticity of $L_F$ follows from the condition
        \begin{displaymath}
            \| \uD^2F(\sigma) - \Upsilon \| \le \varepsilon
            \quad \text{for $\sigma \in \Hom(\R^{\vdim},\R^{\adim-\vdim})$} \,.
        \end{displaymath}
    \end{enumerate}
\end{remark}

\begin{lemma}[\protect{cf.~\cite[7.8]{Menne2012a}}]
    \label{lem:L1-T-1-est}
    Suppose $\vdim$, $\adim$, $U$, $A$, $T_A$ are as in~\ref{def:ell-system},
    $\Upsilon$ as in~\ref{rem:Laplacian},
    \begin{gather}
        a \in \R^{\vdim} \,,
        \quad
        0 < r < \infty \,,
        \quad
        U = \oball ar \,,
        \quad
        u \in \VSobz{1,1}(U,\R^{\adim-\vdim}) \,,
        \quad
        \Lip A < \infty \,.
    \end{gather}
    Then there exists a~number $0 < \varepsilon \le 1/2$ depending only
    on~$\adim$ and~$\vdim$, and for every $0 \le \Lambda < \infty$ a~number
    $\Gamma > 0$ depending only on $\adim$, $\vdim$, and~$\Lambda$, such that
    \begin{displaymath}
        \norm{u}{1}{U} \le \Gamma r \norm{T_A(u)}{-1,1}{U} \,,
    \end{displaymath}
    whenever $\| A(x) - \Upsilon \| \le \varepsilon$ for $\LM^{\vdim}$ almost
    all $x \in U$ and $r \Lip A \le \Lambda$.
\end{lemma}

\begin{proof}
    Let $p = 2\vdim$ and $q = p/(p-1)$ and assume $r=1$, the general case
    following by replacing $u$, $A$ with $u \circ \trans a \circ \scale r$, $A
    \circ \trans a \circ \scale r$, the latter having Lipschitz constant $r \Lip
    A \le \Lambda$. Set
    \begin{displaymath}
        \varepsilon = \min \bigl\{ \tfrac 12 ,\, \varepsilon_{\text{\cite[3.2]{Menne2013}}}(\adim,p) \bigr\}
    \end{displaymath}
    and assume $\| A(x) - \Upsilon \| \le \varepsilon$ for $\LM^{\vdim}$ almost
    all $x \in U$. Let $\theta \in
    \dspace{U}{\R^{\adim-\vdim}}$. Applying~\cite[Theorem~4.6]{Giusti2003} one
    obtains $\eta \in \VSobz{1,p}(U,\R^{\adim-\vdim})$ satisfying
    \begin{displaymath}
        - \int_{U} \bigl\langle \uD \zeta(x) \odot \weakD \eta(x) ,\, A(x) \bigr\rangle \ud \LM^{\vdim}(x)
        = \int_{U} \zeta(x) \bullet \theta(x) \ud \LM^{\vdim}(x)
        \quad \text{for $\zeta \in \dspace{U}{\R^{\adim-\vdim}}$} \,.
    \end{displaymath}
    For each $v \in \R^{\vdim}$ with $|v| = 1$ define $G_v \in \Lp{p}(U,
    \Hom(\R^{\vdim},\R^{\adim-\vdim}))$ by requiring
    \begin{multline}
        G_v(x) \bullet \sigma = \sigma(v) \bullet \theta(x) - \bigl\langle \sigma \odot \weakD \eta(x) ,\, \uD A(x) v \bigr\rangle
        \\
        \text{for $\sigma \in \Hom(\R^{\vdim},\R^{\adim-\vdim})$ and $\LM^{\vdim}$ almost all $x \in U$} 
    \end{multline}
    and apply again~\cite[Theorem~4.6]{Giusti2003} to find $\eta_v \in
    \VSobz{1,p}(U,\R^{\adim-\vdim})$ such that
    \begin{displaymath}
        - \int_{U} \bigl\langle \uD \zeta(x) \odot \weakD \eta_v(x) ,\, A(x) \bigr\rangle \ud \LM^{\vdim}(x)
        = - \int_{U} \bigl\langle \uD \zeta(x) ,\, G_v(x) \bigr\rangle \ud \LM^{\vdim}(x)
        \quad \text{for $\zeta \in \dspace{U}{\R^{\adim-\vdim}}$} \,;
    \end{displaymath}
    the right hand side is the pairing appropriate to~$G_v$, which takes values
    in~$\Hom(\R^{\vdim},\R^{\adim-\vdim})$, and it is precisely what results from
    differentiating the equation for~$\eta$ in the direction~$v$.
    Clearly $\eta_v(x) = \weakD \eta(x)v$ for any unit vector $v \in \R^{\vdim}$
    and $\LM^{\vdim}$~almost all $x \in U$. In~particular, we see that $\eta \in
    \VSobz{2,p}(U,\R^{\adim-\vdim})$. An application of~\cite[10.17]{Giusti2003}
    yields a number $\Delta_1$, depending only on~$\adim$, $\vdim$, and the
    ellipticity data of~$A$ --- hence, by $\|A - \Upsilon\| \le \varepsilon \le
    \tfrac12$, only on~$\adim$ and~$\vdim$ --- with
    \begin{displaymath}
        \norm{\weakD \eta_v}{1,p}{U} \le \Delta_1 \bigl(
        \norm{G_v}{p}{U} + \norm{\eta_v}{2}{U}
        \bigr) \,.
    \end{displaymath}
    Employing~\cite[3.2]{Menne2013} together with the Poincar{\'e}
    inequality~\cite[7.10]{Gilbarg2001} we~get $\Delta_2 \in \R$ depending only
    on~$\adim$ and~$\vdim$ such that
    \begin{displaymath}
        \norm{\weakD \eta}{p}{U} \le \Delta_2 \norm{\theta}{p}{U} \,.
    \end{displaymath}
    Since $\| G_v(x) \| \le |\theta(x)| + \Lip A \, \| \weakD \eta(x) \|$ for
    $\LM^{\vdim}$ almost all $x \in U$, it follows that $\norm{G_v}{p}{U} \le (1
    + \Lambda \Delta_2) \norm{\theta}{p}{U}$; therefore, we obtain a number
    $\Delta_3$ depending only on~$\adim$, $\vdim$, and~$\Lambda$ such that
    \begin{displaymath}
        \norm{\weakD^2 \eta}{p}{U} \le \Delta_3 \norm{\theta}{p}{U} \,;
    \end{displaymath}
    here the term $\norm{\eta_v}{2}{U}$ was absorbed by means of $\eta_v =
    \weakD\eta \, v$, $\norm{\weakD\eta}{2}{U} \le
    \unitmeasure{\vdim}^{1/2-1/p}\norm{\weakD\eta}{p}{U}$, and the preceding
    estimate.
    Proceeding further as in the proof of~\cite[7.8]{Menne2012a}, we~note that $p
    > \vdim$; whence, the Sobolev embedding
    theorem~\cite[Theorem~7.26(ii)]{Gilbarg2001} yields $\Delta_4$ depending only
    on~$\adim$ and~$\vdim$ with
    \begin{displaymath}
        \norm{\weakD \eta}{\infty}{U}
        \le \Delta_4 \bigl( \norm{\weakD \eta}{p}{U} + \norm{\weakD^2 \eta}{p}{U} \bigr)
        \le \Delta_4 ( \Delta_2 + \Delta_3 ) \norm{\theta}{p}{U} \,.
    \end{displaymath}
    Approximating $u$ in $\VSobz{1,1}(U,\R^{\adim-\vdim})$ and $\eta$ in
    $\VSob{1,p}(U,\R^{\adim-\vdim})$ by members of
    $\dspace{U}{\R^{\adim-\vdim}}$, and using $\eta$ as a~test function
    for~$T_A(u)$, we~infer
    \begin{displaymath}
        \textint U{} u(x) \bullet \theta(x) \ud \LM^{\vdim}(x)
        = - \textint U{} \bigl\langle \uD \eta(x) \odot \weakD u(x) ,\, A(x) \bigr\rangle \ud \LM^{\vdim}(x)
        \le \norm{T_A(u)}{-1,1}{U} \norm{\weakD \eta}{\infty}{U} \,.
    \end{displaymath}
    Taking the supremum over $\theta$ with $\norm{\theta}{p}{U} \le 1$, duality
    in~$\Lp{q}$ together with H{\"o}lder's inequality
    (cf.~\cite[2.4.16]{Federer1969}) gives
    \begin{displaymath}
        \norm{u}{1}{U}
        \le \unitmeasure{\vdim}^{1/p} \norm{u}{q}{U}
        \le \unitmeasure{\vdim}^{1/p} \Delta_4 ( \Delta_2 + \Delta_3 ) \norm{T_A(u)}{-1,1}{U} \,,
    \end{displaymath}
    which is the assertion with $\Gamma = \unitmeasure{\vdim}^{1/p} \Delta_4 (
    \Delta_2 + \Delta_3 )$.
\end{proof}

\begin{remark}
    \label{rem:L1-T-1-dependence}
    The dependence of~$\Gamma$ on the scale invariant quantity $r \Lip A$ is
    caused by the passage through second order estimates for the adjoint problem
    in the proof; in the case of constant coefficients, $A = \Upsilon$, one may
    take $\Lambda = 0$ and recovers~\cite[7.8]{Menne2012a}
    and~\cite[3.14]{Menne2013}, where $\Gamma$ depends only on~$\adim$ and the
    ellipticity bounds. We~do not know whether $\Gamma$ may be chosen
    independently of $r \Lip A$ for merely bounded measurable~$A$ satisfying
    $\| A - \Upsilon \| \le \varepsilon$.
\end{remark}

\begin{lemma}[\protect{cf.~\cite[3.16]{Menne2013}}]
    \label{lem:v-u-le-LFu-LFv}
    Suppose $\vdim$, $\adim$, $U$, $F$, $L_F$ are as in~\ref{def:ell-system},
    $\Upsilon$ as in~\ref{rem:Laplacian},
    \begin{gather}
        \| \Upsilon - \uD^2 F(\sigma) \| \le \varepsilon_{\ref{lem:L1-T-1-est}}(\adim,\vdim) \quad \text{for $\sigma \in \Hom(\R^{\vdim},\R^{\adim-\vdim})$} \,,
        \quad
        \Lambda = \Lip \uD^2 F < \infty \,,
        \\
        a \in \R^{\vdim} \,,
        \quad
        0 < r < \infty \,,
        \quad
        U = \oball ar \,,
        \\
        u, v \in \VSob{1,2}(U,\R^{\adim-\vdim}) \,,
        \quad
        u-v \in \VSobz{1,2}(U,\R^{\adim-\vdim}) \,.
    \end{gather}
    Then, for every $0 \le \Theta < \infty$, there exists $\Gamma =
    \Gamma(\adim,\vdim,\Lambda\Theta)$, which may and shall be assumed
    nondecreasing in its last argument, such that for any differentiable $w :
    U \to \R^{\adim-\vdim}$ with $\Lip \uD w < \infty$ and $r \Lip \uD w \le
    \Theta$ there holds
    \begin{displaymath}
        r^{-1-\vdim} \norm{v-u}{1}{U}
        \le \Gamma r^{-\vdim} \bigl(
        \norm{L_F(v) - L_F(u)}{-1,1}{U} 
        + \Lambda ( \norm{\weakD(u-w)}{2}{U}^2 + \norm{\weakD(v-w)}{2}{U}^2 )
        \bigr) \,.
    \end{displaymath}
\end{lemma}

\begin{proof}
    Throughout, we abbreviate $\varepsilon_0 =
    \varepsilon_{\ref{lem:L1-T-1-est}}(\adim,\vdim)$ and recall $0 < \varepsilon_0
    \le 1/2$.  We also recall that, for $\Xi \in \bigodot^2
    \Hom(\R^{\vdim},\R^{\adim-\vdim})$,
    \begin{equation}
        \label{eq:vulf:comass}
        \| \Xi \| = \sup \bigl\{ \langle \sigma \odot \tau ,\, \Xi \rangle :
        \sigma, \tau \in \Hom(\R^{\vdim},\R^{\adim-\vdim}) ,\,
        |\sigma| \le 1 ,\, |\tau| \le 1 \bigr\} \,.
    \end{equation}

    \textbf{Step 0 (preliminaries).}  From $\| \Upsilon - \uD^2 F(\sigma) \|
    \le \varepsilon_0$, valid for every $\sigma \in
    \Hom(\R^{\vdim},\R^{\adim-\vdim})$, we infer
    \begin{equation}
        \label{eq:vulf:D2F-bdd}
        \| \uD^2 F(\sigma) \| \le \| \Upsilon \| + \varepsilon_0 < \infty
        \quad \text{for $\sigma \in \Hom(\R^{\vdim},\R^{\adim-\vdim})$} \,;
    \end{equation}
    in particular $\uD F$ satisfies the growth condition
    of~\ref{def:ell-system}\ref{i:pde:LF}, so that $L_F(u)$ and~$L_F(v)$ are
    well defined elements of $\distr{U}{\R^{\adim-\vdim}}$.  Moreover, $\uD w$ is
    defined on the whole of~$U$ and Lipschitz, so that $B = \uD^2F \circ \uD w$
    below is $\LM^{\vdim}\restrict U$~measurable and bounded; and
    $\LM^{\vdim}(U) < \infty$, consequently
    $\VSobz{1,2}(U,\R^{\adim-\vdim}) \subseteq \VSobz{1,1}(U,\R^{\adim-\vdim})$;
    whence, $v - u \in \VSobz{1,1}(U,\R^{\adim-\vdim})$.

    Define the $\LM^{\vdim} \restrict U$~measurable functions $A, B : U \to
    \bigodot^2 \Hom(\R^{\vdim},\R^{\adim-\vdim})$ and the distributions $R, S
    \in \distr{U}{\R^{\adim-\vdim}}$ by
    \begin{gather}
        A(x) = \textint 01 \uD^2F(t \weakD v(x) + (1-t) \weakD u(x)) - \uD^2 F(\uD w(x)) \ud \LM^{1}(t)
        \quad \text{for $\LM^{\vdim}$ almost all $x \in U$} \,,
        \\
        B = \uD^2F \circ \uD w \,,
        \quad
        R = L_F(v) - L_F(u) \,,
        \quad
        \text{$T_A$ and $T_B$ as in~\ref{def:ell-system}} \,,
        \quad
        S = R - T_A(v-u) \,.
    \end{gather}
    By~\eqref{eq:vulf:D2F-bdd} both $A$ and $B$ are bounded, so that $T_A$
    and~$T_B$ are defined in accordance
    with~\ref{def:ell-system}\ref{i:pde:TA}, and
    \begin{equation}
        \label{eq:vulf:B-close}
        \| B(x) - \Upsilon \| \le \varepsilon_0
        \quad \text{for $\LM^{\vdim}$ almost all $x \in U$} \,.
    \end{equation}

    \textbf{Step 1 ($S = T_B(v-u)$).}  The fundamental theorem of calculus
    (cf.~\cite[3.1.11]{Federer1969}) applied to $t \mapsto \uD F(t \weakD v(x)
    + (1-t)\weakD u(x))$ gives, for $\LM^{\vdim}$ almost all $x \in U$,
    \begin{displaymath}
        \uD F(\weakD v(x)) - \uD F(\weakD u(x))
        = \bigl\langle \weakD (v-u)(x) ,\, (A+B)(x) \bigr\rangle \,;
    \end{displaymath}
    hence, $R = T_{A+B}(v-u)$ and, $C \mapsto T_C$ being linear, $S = R -
    T_A(v-u) = T_B(v-u)$.

    \textbf{Step 2 (application of~\ref{lem:L1-T-1-est}).}  Recalling $B(x) =
    \uD^2F(\uD w(x))$ and $\Lambda = \Lip \uD^2 F$, we~have
    \begin{displaymath}
        \Lip B \le \Lambda \Lip \uD w < \infty \,;
        \quad \text{hence,} \quad
        r \Lip B \le \Lambda \Theta \,.
    \end{displaymath}
    Together with~\eqref{eq:vulf:B-close}, this permits us
    to~apply~\ref{lem:L1-T-1-est} with $B$, $v-u$, $\Lambda\Theta$ in place
    of~$A$, $u$, $\Lambda$, which yields
    \begin{equation}
        \label{eq:vulf:L1-est}
        r^{-1-\vdim} \norm{v-u}{1}{U}
        \le \Gamma_{\ref{lem:L1-T-1-est}}(\adim,\vdim,\Lambda\Theta) r^{-\vdim} \norm{S}{-1,1}{U} \,,
    \end{equation}
    since $S = T_B(v-u)$ by Step~1.

    \textbf{Step 3 (pointwise bound for $A$).}  Writing $\uD w(x)$, appearing
    in the definition of~$A(x)$, as $t \uD w(x) + (1-t) \uD w(x)$ we obtain
    \begin{multline}
        \label{eq:vulf:A-pointwise}
        \|A(x)\| \le \Lambda \textint 01 t |\weakD (v-w)(x)| + (1-t) |\weakD (u-w)(x)| \ud \LM^{1}(t)
        \\
        = \tfrac 12 \Lambda \bigl( |\weakD (v-w)(x)| + |\weakD (u-w)(x)| \bigr)
        \quad \text{for $\LM^{\vdim}$~almost all $x \in U$} \,.
    \end{multline}

    \textbf{Step 4 (bound for $\norm{S}{-1,1}{U}$).}  Let $\theta \in
    \dspace{U}{\R^{\adim-\vdim}}$ satisfy $\norm{\uD \theta}{\infty}{U} \le 1$.
    Using~\eqref{eq:vulf:comass}, then $|\weakD (v-u)(x)| \le |\weakD
    (v-w)(x)| + |\weakD (u-w)(x)|$, and finally~\eqref{eq:vulf:A-pointwise},
    we~estimate
    \begin{multline}
        \bigl| T_A(v-u) \theta \bigr|
        \le \textint U{} \| A(x) \| \, |\weakD (v-u)(x)| \ud \LM^{\vdim}(x)
        \\
        \le \tfrac 12 \Lambda \textint U{} \bigl( |\weakD (v-w)(x)| + |\weakD (u-w)(x)| \bigr)^2 \ud \LM^{\vdim}(x) \,;
    \end{multline}
    hence, by the definition of~$S$ and the triangle inequality,
    \begin{equation}
        \label{eq:vulf:S-est}
        \norm{S}{-1,1}{U}
        \le \norm{R}{-1,1}{U}
        + \tfrac 12 \Lambda \textint{U}{} \bigl( |\weakD (v-w)(x)| + |\weakD (u-w)(x)| \bigr)^2 \ud \LM^{\vdim}(x) \,.
    \end{equation}

    \textbf{Step 5 (conclusion).}  Since $(\alpha + \beta)^2 \le 2 \alpha^2 + 2
    \beta^2$ for $\alpha, \beta \in \R$, we have
    \begin{displaymath}
        \tfrac 12 \textint{U}{} \bigl( |\weakD (v-w)(x)| + |\weakD (u-w)(x)| \bigr)^2 \ud \LM^{\vdim}(x)
        \le \norm{\weakD(u-w)}{2}{U}^2 + \norm{\weakD(v-w)}{2}{U}^2 \,.
    \end{displaymath}
    Combining this with~\eqref{eq:vulf:S-est}, recalling $R = L_F(v) - L_F(u)$,
    and inserting the result into~\eqref{eq:vulf:L1-est}, we conclude
    \begin{displaymath}
        r^{-1-\vdim} \norm{v-u}{1}{U}
        \le \Gamma_{\ref{lem:L1-T-1-est}}(\adim,\vdim,\Lambda\Theta) r^{-\vdim} \bigl(
        \norm{L_F(v) - L_F(u)}{-1,1}{U}
        + \Lambda \bigl( \norm{\weakD(u-w)}{2}{U}^2 + \norm{\weakD(v-w)}{2}{U}^2 \bigr)
        \bigr) \,,
    \end{displaymath}
    which is the assertion with $\Gamma = \Gamma_{\ref{lem:L1-T-1-est}}(\adim,\vdim,\Lambda\Theta)$.
    \qedhere
\end{proof}

\begin{remark}
    \label{rem:vulf:Lip-Dw}
    In~\cite[3.16]{Menne2013} the comparison function~$w$ is affine, so that $B =
    \uD^2F(\uD w)$ is constant and the estimate of~\cite[3.14]{Menne2013} may be
    invoked with a~constant depending only on the ellipticity data.  Admitting
    $w$ with $\Lip \uD w < \infty$, as we~do here, renders $B$ variable and
    thereby introduces the dependence of~$\Gamma$ on $r \Lip \uD w$;
    cf.~\ref{rem:L1-T-1-dependence}.  In the applications
    of~\ref{lem:v-u-le-LFu-LFv} in the proof of the main theorem the
    function~$w$ is one of the comparison functions~$v_{\varrho}$, for which
    $r \Lip \uD v_{\varrho} \le 1$, so that $\Theta = 1$ is admissible.
\end{remark}





\begin{lemma}[\protect{cf.~\cite[3.15]{Menne2013}}]
    \label{lem:W12est}
    Suppose $F$, $L_F$ are as in~\ref{def:ell-system},
    \begin{gather}
        0 < c < M < \infty \,,
        \quad
        a \in \R^{\vdim} \,,
        \quad
        0 < r < \infty \,,
        \\
        \|\uD^2 F(\sigma)\| \le M 
        \quad \text{and} \quad
        \uD^2F(\sigma)(\tau,\tau) \ge c |\tau|^2
        \quad \text{for $\sigma, \tau \in \Hom(\R^{\vdim}, \R^{\adim - \vdim})$} \,,
        \\
        u,v \in \VSob{1,2}(\oball ar,\R^{\adim - \vdim})
        \quad \text{satisfy} \quad
        u - v \in \VSobz{1,2}(\oball ar,\R^{\adim - \vdim}) \,.
    \end{gather}
    Then for every function $w \in \VSob{1,2}(\oball ar,\R^{\adim - \vdim})$
    such that $L_F(w) = 0$ there holds
    \begin{displaymath}
        \norm{\weakD (v-u)}{2}{a,r} \le c^{-1}
        \bigl( M \norm{\weakD (u-w)}{2}{a,r}
        + \norm{L_F(v)}{-1,2}{a,r} \bigr) \,.
    \end{displaymath}
\end{lemma}

\begin{proof}
    Since $L_F(w) = 0$ for $\theta \in \dspace{\oball ar}{\R^{\adim - \vdim}}$
    we have
    \begin{multline}
        L_F(v)(\theta) = - \int_{\oball ar} \langle \uD \theta(x), \uD F(\weakD v(x)) - \uD F(\weakD w(x)) \rangle \ud \LM^{\vdim}(x)
        \\
        = - \int_{\oball ar} \langle \uD \theta(x) \odot \weakD (v-w)(x), A(x) \rangle \ud \LM^{\vdim}(x) \,,
        \\
        \text{where} \quad
        A(x) = \int_0^1 \uD^2 F(t \weakD v(x) + (1-t) \weakD w(x)) \ud \LM^1(t) \,.
    \end{multline}
    This implies for $\theta \in \dspace{\oball ar}{\R^{\adim - \vdim}}$
    \begin{multline}
        \int_{\oball ar} \langle \uD \theta(x) \odot \weakD (v-u)(x), A(x) \rangle \ud \LM^{\vdim}(x)
        \\
        = - \int_{\oball ar} \langle \uD \theta(x) \odot \weakD (u-w)(x), A(x) \rangle \ud \LM^{\vdim}(x)
        - L_F(v)(\theta)
        \,.
    \end{multline}
    Since $u - v \in \VSobz{1,2}(\oball ar,\R^{\adim - \vdim})$ we may
    approximate $v-u$ with $\theta$ and obtain
    \begin{displaymath}
        c ( \norm{\weakD(v-u)}{2}{a,r} )^2
        \le \bigl( M \norm{\weakD(u-w)}{2}{a,r} + \norm{L_F(v)}{-1,2}{a,r} \bigr)
        \norm{\weakD(v-u)}{2}{a,r} \,.
        \qedhere
    \end{displaymath}
\end{proof}

\subsection*{Estimates for the higher order derivatives}
\addcontentsline{toc}{subsection}{Estimates for the higher order derivatives}


\begin{remark}
    \label{rem:ho:plan}
    The remainder of this section is organised as follows.
    \ref{def:ho:seq}--\ref{lem:ho:pde} fix Federer's notation for iterated
    partial derivatives, introduce the class~$W(c,M)$ of admissible
    coefficients, and record the systems satisfied by the derivatives $\uD_{s}f$
    of a solution together with a closed formula for their
    inhomogeneities~$\Omega_s$; \ref{mr:H:cited} collects the results
    of~\cite{Federer1969} which will be used.  \ref{lem:H} is the interior
    Schauder estimate on which everything rests, and \ref{rem:H:iterate} the
    form in which it is iterated.  \ref{lem:ho:tower} is the quantitative
    version of the inductive step in the proof of~\cite[5.2.15]{Federer1969},
    \ref{thm:ho} the quantitative version of~\cite[5.2.15]{Federer1969} itself,
    and \ref{lem:schauder-est} its specialisation to the situation of the
    present paper.

    \emph{Comparison with~\cite{Simon1997}.}  Take there $\kappa =
    (1,\ldots,1)$ and $I = J = \{ \gamma : |\gamma| = 1 \}$, so that the
    operator is $u \mapsto \textsum{i,j}{} \uD_j (A_{ij} \uD_i u)$ and the
    seminorm $[\cdot]_{-J,\alpha}$ is the infimum of $\textsum{i}{}
    [\Omega_i]_{\alpha}$ over the representations $\textsum{i}{} \uD_i
    \Omega_i$ of the right hand side.  Then \cite[Theorem~3]{Simon1997},
    combined with the interpolation inequality \cite[1.5]{Simon1997}, is the
    estimate~\eqref{eq:H:concl} of~\ref{lem:H}, and \cite[\S9]{Simon1997}
    observes that the passage to systems requires none but notational changes.
    Simon's hypotheses are in fact weaker than ours: no smallness of the
    $\delta$~H{\"o}lder seminorm of~$A$ is required, only a bound, and
    hypoellipticity of the operator obtained by freezing the coefficients
    replaces the Legendre--Hadamard condition of~\ref{def:ho:W}.  His theorem
    is also far more general, covering $\kappa$~homogeneous hypoelliptic
    operators of arbitrary order and general boundary value problems, and its
    proof is much shorter: a blow-up and compactness argument reduces the
    estimate to a Liouville theorem for the frozen operator.  That proof
    is, however, by contradiction; the constant is produced by a compactness
    argument, no bound for it is obtained, and its uniformity is asserted only
    over compact families of operators.  The higher order estimates
    of~\cite[\S8]{Simon1997} are then derived, as they are here, by applying
    the first order estimates to the derivatives of the solution.  We~follow
    Federer instead, whose \cite[5.2.14]{Federer1969} is proved by explicit
    potential theoretic estimates, because this keeps the present paper
    self-contained and its constants explicit.
\end{remark}

\begin{definition}
    \label{def:ho:seq}
    Let $k$ be a positive integer and let $p$ be a non-negative integer.  We
    denote by $\Sequence{k}{p}$ the set of
    all functions of the type $\{1,\ldots,p\} \to \{1,\ldots,k\}$, and we let
    $\Sequence{k}{0}$ be the one-element set whose member is the unique function
    $\varnothing \to \{1,\ldots,k\}$; cf.~\cite[1.10.6]{Federer1969}.  Fixing
    the standard basis $e_1,\ldots,e_{\vdim}$ of~$\R^{\vdim}$, for a
    function~$f$ of class~$\cnt{p}$ defined on an open subset of~$\R^{\vdim}$
    with values in a normed space and for $s \in \Sequence{\vdim}{p}$ we
    abbreviate
    \begin{displaymath}
        \uD_i f = \langle e_i ,\, \uD f \rangle \,,
        \quad
        \uD_s f = \uD_{s(1)} \uD_{s(2)} \cdots \uD_{s(p)} f
        = \langle e_{s(1)} \odot \cdots \odot e_{s(p)} ,\, \uD^p f \rangle \,,
    \end{displaymath}
    with the convention $\uD_s f = f$ in case $p = 0$.
\end{definition}

\begin{definition}
    \label{def:ho:W}
    For $\Upsilon \in \bigodot^2 \Hom(\R^{\vdim},\R^{\adim-\vdim})$ we define
    \begin{displaymath}
        \ellipticity{\Upsilon} = \inf \bigl\{
        \langle (\xi\,y) \odot (\xi\,y) \,, \Upsilon \rangle :
        \xi \in \Hom(\R^{\vdim},\R) \,,\ y \in \R^{\adim-\vdim} \,,\
        |\xi| = 1 = |y| \bigr\} \,,
    \end{displaymath}
    and, given $0 < c < M < \infty$, we~let
    \begin{displaymath}
        W(c,M) = {\textstyle \bigodot^2} \Hom(\R^{\vdim},\R^{\adim-\vdim}) \cap
        \bigl\{ \Upsilon : \ellipticity{\Upsilon} > c \,,\ \| \Upsilon \| < M \bigr\} \,.
    \end{displaymath}
\end{definition}

\begin{miniremark}
    \label{mr:ho:Omega}
    Suppose $U \subseteq \R^{\vdim}$ is open, $q$ is a non-negative integer,
    $f : U \to
    \R^{\adim-\vdim}$ is of class~$\cnt{q+1}$, $A : U \to {\textstyle
      \bigodot^2}\Hom(\R^{\vdim},\R^{\adim-\vdim})$, $\Omega : U \to
    \Hom(\R^{\vdim},\R^{\adim-\vdim})$ and $h : U \to
    \R^{\adim-\vdim}$ are of class~$\cnt{q}$.  Whenever $j \in \{1,\ldots,q\}$,
    $s \in \Sequence{\vdim}{j}$, $x \in U$, and $\sigma \in
    \Hom(\R^{\vdim},\R^{\adim-\vdim})$, writing $s' = s|\{1,\ldots,j\} \without
    \{j\}$, we~define $h_{s'} = \uD_{s'} h$ and the map $\Omega_s : U \to
    \Hom(\R^{\vdim},\R^{\adim-\vdim})$ recursively by setting $\Omega_{s'} =
    \Omega$ in case $s' \in \Sequence{\vdim}{0}$ and
    \begin{equation}
        \label{eq:ho:Omega-recursion}
        \Omega_s(x) \bullet \sigma
        = \uD_{s(j)} \Omega_{s'}(x) \bullet \sigma
        - \bigl\langle \uD \uD_{s'} f(x) \odot \sigma \,, \uD_{s(j)} A(x) \bigr\rangle \,.
    \end{equation}
\end{miniremark}

\begin{lemma}
    \label{lem:ho:pde}
    Adopt the hypotheses and notation of~\ref{mr:ho:Omega} and assume
    \begin{equation}
        \label{eq:ho:base-pde}
        \textint{U}{} \bigl\langle \uD \theta \odot \uD f \,, A \bigr\rangle \ud \LM^{\vdim}
        = \textint{U}{} \theta \bullet h + \uD \theta \bullet \Omega \ud \LM^{\vdim}
        \quad \text{for $\theta \in \dspace{U}{\R^{\adim-\vdim}}$} \,.
    \end{equation}
    Then, for every $s \in \Sequence{\vdim}{q}$, writing $h_s = \uD_s h$:
    \begin{enumerate}
    \item
        \label{i:ho:pde}
        $\textint{U}{} \bigl\langle \uD \theta \odot \uD \uD_s f \,, A \bigr\rangle \ud \LM^{\vdim}
        = \textint{U}{} \theta \bullet h_s + \uD \theta \bullet \Omega_s \ud \LM^{\vdim}$
        for $\theta \in \dspace{U}{\R^{\adim-\vdim}}$;
    \item
        \label{i:ho:Omegas-explicit}
        for $\sigma \in \Hom(\R^{\vdim},\R^{\adim-\vdim})$ and $x \in U$,
        \begin{displaymath}
            \sigma \bullet \Omega_s(x) = \sigma \bullet \uD_s \Omega(x)
            - \textsum{j=0}{q-1} \textsum{\eta \in \Shuffle{j}{q-j}}{}
            \bigl\langle \uD \uD_{s \circ \eta|\{1,\ldots,j\}} f(x) \odot \sigma \,,
            \uD_{s \circ \eta|\{j+1,\ldots,q\}} A(x) \bigr\rangle \,.
        \end{displaymath}
    \end{enumerate}
\end{lemma}

\begin{proof}
    We argue by induction on~$q$; the case $q = 0$ is~\eqref{eq:ho:base-pde}
    with $\Omega_{s} = \Omega$, $h_s = h$.  Let $q \ge 1$, fix $s \in
    \Sequence{\vdim}{q}$, put $i = s(q)$ and $s' = s|\{1,\ldots,q\} \without
    \{q\} \in \Sequence{\vdim}{q-1}$, and assume~\ref{i:ho:pde}
    and~\ref{i:ho:Omegas-explicit} hold with $q-1$ and~$s'$ in place of $q$
    and~$s$.  Since $\uD_{s'} f$, $\Omega_{s'}$, $A$ are of class~$\cnt{1}$ and
    $\uD_i \uD_{s'} f = \uD_s f$, for $\theta \in \dspace{U}{\R^{\adim-\vdim}}$
    the map $x \mapsto \bigl\langle \uD \theta(x) \odot \uD \uD_{s'} f(x) \,,
    A(x) \bigr\rangle$ has compact support; whence, $\textint{U}{} \uD_i \bigl[
    \bigl\langle \uD \theta \odot \uD \uD_{s'} f \,, A \bigr\rangle \bigr] \ud
    \LM^{\vdim} = 0$ and therefore
    \begin{displaymath}
        \textint{U}{} \bigl\langle \uD \uD_i \theta \odot \uD \uD_{s'} f \,, A \bigr\rangle
        + \bigl\langle \uD \theta \odot \uD \uD_{s} f \,, A \bigr\rangle
        + \bigl\langle \uD \theta \odot \uD \uD_{s'} f \,, \uD_i A \bigr\rangle
        \ud \LM^{\vdim} = 0 \,.
    \end{displaymath}
    Using~\ref{i:ho:pde} with $s'$ in place of~$s$ and $\uD_i \theta$ in place
    of~$\theta$, and integrating by parts, we~obtain
    \begin{multline}
        \textint{U}{} \bigl\langle \uD \theta \odot \uD \uD_{s} f \,, A \bigr\rangle \ud \LM^{\vdim}
        = - \textint{U}{} \uD_i \theta \bullet h_{s'}
        + \uD \uD_i \theta \bullet \Omega_{s'}
        + \bigl\langle \uD \theta \odot \uD \uD_{s'} f \,, \uD_i A \bigr\rangle
        \ud \LM^{\vdim}
        \\
        = \textint{U}{} \theta \bullet \uD_i h_{s'}
        + \uD \theta \bullet \uD_i \Omega_{s'}
        - \bigl\langle \uD \theta \odot \uD \uD_{s'} f \,, \uD_i A \bigr\rangle
        \ud \LM^{\vdim}
        = \textint{U}{} \theta \bullet h_{s} + \uD \theta \bullet \Omega_{s} \ud \LM^{\vdim} \,,
    \end{multline}
    where we used $\uD_i h_{s'} = \uD_s h = h_s$ and, by~\eqref{eq:ho:Omega-recursion},
    $\uD_i \Omega_{s'} \bullet \sigma - \bigl\langle \uD \uD_{s'} f \odot \sigma \,,
    \uD_i A \bigr\rangle = \Omega_s \bullet \sigma$.  This
    settles~\ref{i:ho:pde}.  Finally, substituting the inductive
    form~\ref{lem:ho:pde}\ref{i:ho:Omegas-explicit} of~$\Omega_{s'}$
    into~\eqref{eq:ho:Omega-recursion} and distributing the
    derivative~$\uD_{i}$ over the factor $\uD_{s'} \Omega$ and over each product
    $\uD \uD_{\cdots} f \odot \uD_{\cdots} A$ by the product rule, one obtains,
    upon grouping the resulting terms according to the shuffle that assigns
    each of the derivatives $\uD_{s(1)},\ldots,\uD_{s(q)}$ either to the
    $f$-factor or to the $A$-factor, the
    formula~\ref{lem:ho:pde}\ref{i:ho:Omegas-explicit} at level~$q$; the derivative landing
    on $\uD_{s'}\Omega$ produces the term $\sigma \bullet \uD_s \Omega$, while
    the shuffles with $j$ derivatives on the $f$-factor and $q-j$ on the
    $A$-factor produce the double sum.
\end{proof}


\begin{miniremark}
    \label{mr:H:cited}
    Throughout this subsection $\vdim,\adim \in \integers$, $0 < \vdim < \adim$,
    $0 < \delta < 1$, and $0 < c < M < \infty$; the set $W(c,M)$ is as
    in~\ref{def:ho:W}.  Following \cite[5.2.1]{Federer1969} we~write
    $\norm{f}{2}{b,r}$ for the $\Lp{2}$~norm of~$f$ over $\oball br$ --- this is
    the quantity denoted $|f|_{b,r}$ there --- and $\hoelder{\delta}{f}$ for the
    least~$t$ such that $|f(x)-f(y)| \le t|x-y|^{\delta}$ for $x,y \in \dmn f$.
    In~addition, for $f$ defined on $\cball br$ and $0 < t < \infty$, we~abbreviate
    \begin{displaymath}
        \hnorm{\delta}{f}{b,r}{t} = \norm{f}{\infty}{b,r} + t^{\delta} \hoelder{\delta}{f|\cball br} \,,
    \end{displaymath}
    and record the two properties of this quantity which will be used without
    further comment: it is non-decreasing in~$t$ and in~$r$, and it is
    submultiplicative, in the sense that
    \begin{equation}
        \label{eq:hnorm:mult}
        \hnorm{\delta}{f \bullet g}{b,r}{t} \le \hnorm{\delta}{f}{b,r}{t} \, \hnorm{\delta}{g}{b,r}{t}
    \end{equation}
    whenever $\bullet$ is a bilinear pairing of norm at most~$1$.

    We~employ four results of Federer:
    \begin{enumerate}
    \item
        \label{i:cited:closure}
        the \emph{closure lemma} \cite[5.2.2]{Federer1969}, applied
        componentwise as in \cite[p.~558]{Federer1969};
    \item
        \label{i:cited:garding}
        \emph{G{\aa}rding's inequality} \cite[5.2.3, p.~535]{Federer1969}, in the
        sharpened form recorded there --- where the value $A(b)$ occurring in
        \cite{Federer1969} is replaced by an arbitrary strongly elliptic
        $\Upsilon$, which the one line proof given there permits: if $U
        \subseteq \R^{\vdim}$ is open,
        $\Upsilon$ is strongly elliptic with ellipticity bound~$c$, and
        $\|A(x)-\Upsilon\| \le d$ for $x \in U$, then $\textint{U}{}
        \bigl\langle \uD\theta \odot \uD\theta ,\, A \bigr\rangle \ud
        \LM^{\vdim} \ge (c-d) \bigl( \norm{\uD\theta}{2}{U} \bigr)^{2}$ for
        $\theta \in \dspace{U}{\R^{\adim-\vdim}}$;
    \item
        \label{i:cited:basic}
        the \emph{basic estimate} \cite[5.2.3, p.~536]{Federer1969};
    \item
        \label{i:cited:hoelder}
        the \emph{H{\"o}lder estimate for linear systems}
        \cite[5.2.14]{Federer1969}, whose constants we~denote by $\varepsilon_0
        = \varepsilon_{\text{\cite[5.2.14]{Federer1969}}}(\vdim,\adim,\delta,c,M)$
        and $\Delta_0 =
        \Gamma_{\text{\cite[5.2.14]{Federer1969}}}(\vdim,\adim,\delta,c,M)$.
    \end{enumerate}
    Since \ref{i:cited:garding}--\ref{i:cited:hoelder} are phrased in terms
    of strong ellipticity, whereas $W(c,M)$ is defined in~\ref{def:ho:W} by the
    Legendre--Hadamard condition, we~record that the two agree with the same
    bound; this is the converse implication alluded to in \cite[5.2.3,
    p.~534]{Federer1969}.  Indeed, if $\ellipticity{\Upsilon} \ge c$ and $\theta
    \in \dspace{\R^{\vdim}}{\R^{\adim-\vdim}}$, then the Fourier transform
    of~$\uD\theta$ at~$\xi$ equals $2\pi \sqrt{-1} \, \xi \, \hat\theta(\xi)$;
    whence, writing $\hat\theta = a + \sqrt{-1} \, b$ with $a$ and~$b$ real
    valued and using Plancherel's theorem together with the homogeneity of the
    expression defining $\ellipticity{\Upsilon}$,
    \begin{multline}
        \textint{}{} \bigl\langle \uD\theta \odot \uD\theta ,\, \Upsilon \bigr\rangle \ud \LM^{\vdim}
        = 4\pi^2 \textint{}{} \bigl\langle (\xi \, a(\xi)) \odot (\xi \, a(\xi))
        + (\xi \, b(\xi)) \odot (\xi \, b(\xi)) ,\, \Upsilon \bigr\rangle \ud \LM^{\vdim}(\xi)
        \\
        \ge 4\pi^2 c \textint{}{} |\xi|^2 \bigl( |a(\xi)|^2 + |b(\xi)|^2 \bigr) \ud \LM^{\vdim}(\xi)
        = c \bigl( \norm{\uD\theta}{2}{} \bigr)^2 \,.
    \end{multline}
\end{miniremark}

\begin{lemma}[interior estimate; cf.~\protect{\cite[5.2.14]{Federer1969}}]
    \label{lem:H}
    Let $\vdim,\adim,\delta,c,M$ be as in~\ref{mr:H:cited} and set
    \begin{displaymath}
        \varepsilon = \varepsilon(\vdim,\adim,\delta,c,M) = \min \{ \varepsilon_0 \,, c/4 \} \,.
    \end{displaymath}
    There exists $\Gamma = \Gamma(\vdim,\adim,\delta,c,M) < \infty$ such that
    the following holds.  Suppose
    \begin{enumerate}
    \item
        $b \in \R^{\vdim}$ and $0 < r < r' < \infty$; put $\tau = r'-r$;
    \item
        \label{i:H:assume-epsilon}
        $A : \cball{b}{r'} \to W(c,M)$ is of class~$\cnt{0,\delta}$ and
        $(r')^{\delta} \hoelder{\delta}{A|\cball{b}{r'}} \le \varepsilon$;
    \item
        \label{i:H:g}
        $g : \cball{b}{r'} \to \R^{\adim-\vdim}$ is of class~$\cnt{1}$ with
        $\hoelder{\delta}{\uD g|\cball{b}{r'}} < \infty$, and $\Omega :
        \cball{b}{r'} \to \Hom(\R^{\vdim},\R^{\adim-\vdim})$ is of
        class~$\cnt{0,\delta}$;
    \item
        \label{i:H:pde}
        $\textint{}{} \bigl\langle \uD\theta \odot \uD g ,\, A \bigr\rangle \ud
        \LM^{\vdim} = \textint{}{} \uD\theta \bullet \Omega \ud \LM^{\vdim}$ for
        $\theta \in \dspace{\oball{b}{r'}}{\R^{\adim-\vdim}}$.
    \end{enumerate}
    Then
    \begin{gather}
        \label{eq:H:L2}
        \hnorm{\delta}{\uD g}{b,r}{\tau} \le \Gamma \Bigl(
        \tau^{-\vdim/2} \norm{\uD g}{2}{b,r'}
        + \tau^{\delta} \hoelder{\delta}{\Omega|\cball{b}{r'}} \Bigr) \,,
        \\
        \label{eq:H:concl}
        \hnorm{\delta}{\uD g}{b,r}{\tau} \le \Gamma \Bigl(
        \tau^{-1} \norm{g}{\infty}{b,r'} + \hnorm{\delta}{\Omega}{b,r'}{\tau} \Bigr) \,.
    \end{gather}
\end{lemma}

\begin{proof}
    Put $\varsigma = \tau/3$ and let $z \in \cball br$; then $\cball
    z{2\varsigma} \subseteq \cball b{r'}$ and $2\varsigma \le r'$.

    \emph{Step 1 (localisation).}  For $x \in \cball{z}{\varsigma}$ we~have $\|
    A(x) - A(z) \| \le \hoelder{\delta}{A|\cball b{r'}} \varsigma^{\delta} \le
    \varepsilon (\varsigma/r')^{\delta} \le \varepsilon \le c/4$; since $A(z)$ is
    strongly elliptic with ellipticity bound~$c$, as recorded at the end
    of~\ref{mr:H:cited}, the inequality
    of~\ref{mr:H:cited}\ref{i:cited:garding} yields
    \begin{displaymath}
        \textint{}{} \bigl\langle \uD\theta \odot \uD\theta ,\, A \bigr\rangle \ud \LM^{\vdim}
        \ge \tfrac{3c}{4} \bigl( \norm{\uD\theta}{2}{z,\varsigma} \bigr)^{2}
        \quad \text{for $\theta \in \dspace{\oball{z}{\varsigma}}{\R^{\adim-\vdim}}$} \,.
    \end{displaymath}
    Consequently \ref{mr:H:cited}\ref{i:cited:basic}, applied
    to~\ref{i:H:pde} on $\oball z\varsigma$ with $\lambda = 0$, $\chi = 0$, $\mu
    = 3c/4$, and with inner radius $\varsigma/2$, gives
    \begin{equation}
        \label{eq:H:caccioppoli}
        \tfrac{3c}{4} \norm{\uD g}{2}{z,\varsigma/2}
        \le \norm{\Omega}{2}{z,\varsigma}
        + (3cM/4)^{1/2} \, 2 \varsigma^{-1} \norm{g}{2}{z,\varsigma} \,.
    \end{equation}
    Moreover, $(\varsigma/2)^{\delta} \hoelder{\delta}{A|\cball{z}{\varsigma/2}}
    \le \varepsilon \le \varepsilon_0$; whence,
    \ref{mr:H:cited}\ref{i:cited:hoelder}, applied to~\ref{i:H:pde} on
    $\oball{z}{\varsigma/2}$ with inner radius $\varsigma/4$, gives
    \begin{equation}
        \label{eq:H:schauder}
        (\varsigma/4)^{\vdim/2+\delta} \hoelder{\delta}{\uD g|\cball{z}{\varsigma/4}}
        \le \Delta_0 \Bigl( \norm{\uD g}{2}{z,\varsigma/2}
        + (\varsigma/2)^{\vdim/2+\delta} \hoelder{\delta}{\Omega|\cball{z}{\varsigma/2}} \Bigr) \,.
    \end{equation}

    \emph{Step 2 (passage to the supremum norm).}  Since, for $y \in
    \cball{z}{\varsigma/4}$,
    \begin{displaymath}
        |\uD g(y)| \le \bigl( \unitmeasure{\vdim} (\varsigma/4)^{\vdim} \bigr)^{-1/2}
        \norm{\uD g}{2}{z,\varsigma/4}
        + (\varsigma/2)^{\delta} \hoelder{\delta}{\uD g|\cball{z}{\varsigma/4}} \,,
    \end{displaymath}
    dividing~\eqref{eq:H:schauder} by $\varsigma^{\vdim/2}$ produces $\Delta_1 =
    \Delta_1(\vdim,\adim,\delta,c,M)$ with
    \begin{equation}
        \label{eq:H:local}
        \hnorm{\delta}{\uD g}{z,\varsigma/4}{\varsigma}
        \le \Delta_1 \Bigl( \varsigma^{-\vdim/2} \norm{\uD g}{2}{z,\varsigma/2}
        + \varsigma^{\delta} \hoelder{\delta}{\Omega|\cball{z}{\varsigma/2}} \Bigr) \,.
    \end{equation}

    \emph{Step 3 (patching).}  As $\norm{\uD g}{\infty}{b,r} \le \sup \{ \norm{\uD
      g}{\infty}{z,\varsigma/4} : z \in \cball br \}$ and, for $x,y \in \cball
    br$, either $|x-y| \le \varsigma/4$ --- whence $x,y \in \cball
    x{\varsigma/4}$ --- or $|x-y| > \varsigma/4$, we~get
    \begin{displaymath}
        \hoelder{\delta}{\uD g|\cball{b}{r}}
        \le \sup_{z \in \cball br} \hoelder{\delta}{\uD g|\cball{z}{\varsigma/4}}
        + 2 (\varsigma/4)^{-\delta} \norm{\uD g}{\infty}{b,r} \,,
    \end{displaymath}
    and therefore, recalling $\tau = 3\varsigma$,
    \begin{equation}
        \label{eq:H:patch2}
        \hnorm{\delta}{\uD g}{b,r}{\tau}
        \le \bigl( 1 + 3^{\delta} ( 1 + 2 \cdot 4^{\delta} ) \bigr)
        \sup_{z \in \cball br} \hnorm{\delta}{\uD g}{z,\varsigma/4}{\varsigma}
        \le 28 \sup_{z \in \cball br} \hnorm{\delta}{\uD g}{z,\varsigma/4}{\varsigma} \,.
    \end{equation}

    \emph{Step 4 (conclusion).}  Since $\cball{z}{\varsigma} \subseteq
    \cball{b}{r'}$, the right hand sides of~\eqref{eq:H:local}
    and~\eqref{eq:H:caccioppoli} only increase when the balls
    $\cball{z}{\varsigma/2}$ and $\cball{z}{\varsigma}$ occurring in them are
    replaced by $\cball{b}{r'}$; both estimates therefore persist after that
    replacement.
    Combining~\eqref{eq:H:patch2} with~\eqref{eq:H:local} and recalling that
    $\varsigma^{-\vdim/2} = 3^{\vdim/2}\tau^{-\vdim/2}$ and
    $\varsigma^{\delta} \le \tau^{\delta}$, we~obtain~\eqref{eq:H:L2}.  If
    instead we~first insert~\eqref{eq:H:caccioppoli} into~\eqref{eq:H:local}
    and then estimate
    \begin{displaymath}
        \norm{h}{2}{z,\varsigma} \le \bigl( \unitmeasure{\vdim} \varsigma^{\vdim} \bigr)^{1/2} \norm{h}{\infty}{z,\varsigma}
        \quad \text{for $h \in \{ g ,\, \Omega \}$} \,,
    \end{displaymath}
    we~arrive at~\eqref{eq:H:concl}.
\end{proof}

\begin{remark}
    \label{rem:H:iterate}
    Suppose $U = \oball{b}{\varrho}$, the maps $u$, $A$, and~$\Omega$ are of
    class~$\cnt{\infty}$ on $\cball{b}{\varrho}$ --- the smoothness of $A$
    and~$\Omega$ being what allows $\Omega_s$ to be formed at \emph{every}
    order in~\ref{mr:ho:Omega} --- and $u$, $A$, $\Omega$, $h = 0$ satisfy the
    base equation~\eqref{eq:ho:base-pde}, where $A$ takes its values
    in~$W(c,M)$ and satisfies~\ref{lem:H}\ref{i:H:assume-epsilon}
    on~$\cball{b}{\varrho}$.
    Applying~\ref{lem:H} with $g = \uD_{s} u$ and with $\Omega_{s}$ in place
    of~$\Omega$ --- which is legitimate by~\ref{lem:ho:pde}\ref{i:ho:pde} ---
    and using $\norm{\uD_{s} u}{\infty}{b,r'} \le \norm{\uD \uD_{s'}
      u}{\infty}{b,r'}$, one~obtains: whenever $q \in \natp$, $s \in
    \Sequence{\vdim}{q}$, $s' = s|\{1,\ldots,q\} \without \{q\}$, $0 < r < r'
    \le \varrho$, and $\tau = r'-r$, there holds
    \begin{equation}
        \label{eq:H}
        \tau \, \hnorm{\delta}{\uD \uD_{s} u}{b,r}{\tau}
        \le \Gamma \Bigl( \norm{\uD \uD_{s'} u}{\infty}{b,r'} + \tau \, \omega_{s}(r') \Bigr) \,,
        \quad \text{where} \quad
        \omega_{s}(r',\tau) = \hnorm{\delta}{\Omega_{s}}{b,r'}{\tau} \,,
    \end{equation}
    the second argument of~$\omega_s$ being carried explicitly because the
    weight is not determined by the radius: below it is taken once as
    $\varrho-r'$ and once as $r'-r$.
    with $\Gamma = \Gamma_{\ref{lem:H}}(\vdim,\adim,\delta,c,M)$
    \emph{independent of~$q$}.  The conversion
    \begin{displaymath}
        \norm{\uD^{q+1} u}{\infty}{} \le \vdim^{(q+1)/2} \max_{s \in \Sequence{\vdim}{q}} \norm{\uD \uD_s u}{\infty}{}
    \end{displaymath}
    is performed \emph{once}, at the very end; it is precisely to keep~$\Gamma$
    independent of~$q$ that~\eqref{eq:H} is stated componentwise in~$s$.
\end{remark}


\begin{lemma}[higher order estimates for linear systems]
    \label{lem:ho:tower}
    Let $\vdim,\adim,\delta,c,M$ and $\varepsilon$ be as in~\ref{lem:H} and let
    $q \in \natp$.  There exists $\Gamma =
    \Gamma(\vdim,\adim,\delta,c,M,q) < \infty$ such that the following
    holds.  Suppose
    \begin{enumerate}
    \item
        $b \in \R^{\vdim}$ and $0 < \varrho < \infty$;
    \item
        \label{i:tower:A}
        $A : \cball{b}{\varrho} \to W(c,M)$ is of class~$\cnt{q-1,\delta}$ and
        $\varrho^{\delta} \hoelder{\delta}{A|\cball{b}{\varrho}} \le \varepsilon$;
    \item
        \label{i:tower:Omega}
        $\Omega : \cball{b}{\varrho} \to \Hom(\R^{\vdim},\R^{\adim-\vdim})$ is
        of class~$\cnt{q-1,\delta}$;
    \item
        \label{i:tower:f}
        $f : \cball{b}{\varrho} \to \R^{\adim-\vdim}$ is of class~$\cnt{1}$,
        $\hoelder{\delta}{\uD f|\cball{b}{\varrho}} < \infty$, and $f$, $A$,
        $\Omega$, $h = 0$ satisfy~\eqref{eq:ho:base-pde} with $U =
        \oball{b}{\varrho}$.
    \end{enumerate}
    Then, $\Omega_s$ being defined as in~\ref{mr:ho:Omega},
    \begin{enumerate}[resume]
    \item
        \label{i:tower:reg}
        $f$ is of class~$\cnt{q}$ on $\oball{b}{\varrho}$ and
        $\hoelder{\delta}{\uD^{q} f|K} < \infty$ for every compact $K \subseteq
        \oball{b}{\varrho}$;
    \item
        \label{i:tower:Omega-s}
        for $p \in \{0,1,\ldots,q-1\}$ and $s \in \Sequence{\vdim}{p}$ the
        map~$\Omega_s$ is of class~$\cnt{q-1-p}$ on $\oball{b}{\varrho}$ with
        $\hoelder{\delta}{\uD^{q-1-p}\Omega_s|K} < \infty$ for every compact $K
        \subseteq \oball{b}{\varrho}$, and
        \begin{displaymath}
            \textint{}{} \bigl\langle \uD\theta \odot \uD \uD_{s} f \,, A \bigr\rangle \ud \LM^{\vdim}
            = \textint{}{} \uD\theta \bullet \Omega_{s} \ud \LM^{\vdim}
            \quad \text{for $\theta \in \dspace{\oball{b}{\varrho}}{\R^{\adim-\vdim}}$} \,;
        \end{displaymath}
    \item
        \label{i:tower:est}
        abbreviating
        \begin{displaymath}
            \mathcal{A} = \textsum{i=1}{q-1} \varrho^{i} \hnorm{\delta}{\uD^{i}A}{b,\varrho}{\varrho} \,,
            \quad
            \mathcal{O} = \textsum{i=0}{q-1} \varrho^{i} \hnorm{\delta}{\uD^{i}\Omega}{b,\varrho}{\varrho} \,,
        \end{displaymath}
        there holds, whenever $0 < r < \varrho$,
        \begin{displaymath}
            \textsum{p=1}{q} \varrho^{p-1} \hnorm{\delta}{\uD^{p} f}{b,r}{\varrho}
            \le \Gamma \, \Bigl( \frac{\varrho}{\varrho-r} \Bigr)^{q} (1+\mathcal{A})^{q} \Bigl(
            \hnorm{\delta}{\uD f}{b,\varrho}{\varrho} + \mathcal{O} \Bigr) \,.
        \end{displaymath}
    \end{enumerate}
\end{lemma}

\begin{proof}
    \emph{Part~1: \ref{i:tower:reg} and the equation of~\ref{i:tower:Omega-s}, by induction
      on~$q$.}  For $q = 1$ both are contained in~\ref{i:tower:f}.  Let $q \ge
    2$ and assume the assertions proved with $q-1$ in place of~$q$; since
    $\cnt{q-1,\delta} \subseteq \cnt{q-2,\delta}$, they apply to the present
    data and yield that $f$ is of class~$\cnt{q-1}$ on $\oball{b}{\varrho}$,
    with $\hoelder{\delta}{\uD^{q-1}f|K} < \infty$ for compact $K \subseteq
    \oball b\varrho$, and that the equation of~\ref{i:tower:Omega-s} holds for $s \in
    \Sequence{\vdim}{p}$ with $p \le q-2$.

    We~first note that $\Omega_{s}$ is of class~$\cnt{1}$ with
    $\hoelder{\delta}{\uD\Omega_s|K} < \infty$ for compact $K \subseteq \oball
    b\varrho$ whenever $s \in \Sequence{\vdim}{q-2}$.  Indeed,
    by~\ref{lem:ho:pde}\ref{i:ho:Omegas-explicit} the map $\Omega_s$ is the sum
    of $\uD_s \Omega$, which is of class~$\cnt{1,\delta}$ by~\ref{i:tower:Omega},
    and of products $\uD\uD_{\sigma} f \odot \uD_{\varsigma} A$ with $\sigma \in
    \Sequence{\vdim}{j}$ and $\varsigma \in \Sequence{\vdim}{q-2-j}$ for some
    $j \in \{0,\ldots,q-3\}$; the first factor is of class~$\cnt{q-2-j}$ with
    $\delta$~H{\"o}lder top derivative --- as $f$ is of class~$\cnt{q-1}$ and
    $j+1 \le q-2$ --- and the second is of class~$\cnt{j+1,\delta}$
    by~\ref{i:tower:A}; both exponents are at least~$1$.

    Now fix $s \in \Sequence{\vdim}{q-2}$, $i \in \{1,\ldots,\vdim\}$, and
    numbers $0 < r < r' < \varrho' < \varrho$, put $g = \uD_{s}f$, and define,
    for integers $\nu > (\varrho'-r')^{-1}$, maps on $\cball{b}{r'}$ by
    \begin{gather}
        g_{\nu}(x) = \nu \bigl( g(x + \nu^{-1}e_i) - g(x) \bigr) \,,
        \quad
        \zeta_{\nu}(x) = \nu \bigl( \Omega_{s}(x+\nu^{-1}e_i) - \Omega_{s}(x) \bigr) \,,
        \\
        \sigma \bullet \eta_{\nu}(x) = \nu \bigl\langle \uD \uD_{s} f(x+\nu^{-1}e_i) \odot \sigma \,,
        A(x+\nu^{-1}e_i) - A(x) \bigr\rangle \,.
    \end{gather}
    Testing the equation of~\ref{i:tower:Omega-s} with $\theta_{\nu}(x) = \nu \bigl(
    \theta(x-\nu^{-1}e_i) - \theta(x) \bigr)$, where $\theta \in
    \dspace{\oball{b}{r'}}{\R^{\adim-\vdim}}$, and translating the
    variable of integration --- the computation is that
    of~\cite[p.~558]{Federer1969} --- we~obtain
    \begin{equation}
        \label{eq:tower:nu-pde}
        \textint{}{} \bigl\langle \uD\theta \odot \uD g_{\nu} \,, A \bigr\rangle \ud \LM^{\vdim}
        = \textint{}{} \uD\theta \bullet ( \zeta_{\nu} - \eta_{\nu} ) \ud \LM^{\vdim}
        \quad \text{for $\theta \in \dspace{\oball{b}{r'}}{\R^{\adim-\vdim}}$} \,.
    \end{equation}
    Writing $\nu \bigl( F(x+\nu^{-1}e_i) - F(x) \bigr) = \textint{0}{1} \uD_i
    F(x+t\nu^{-1}e_i) \ud \LM^1(t)$ for $F = g$, for $F = \Omega_{s}$, and for
    $F = A$, and using~\eqref{eq:hnorm:mult}, we~infer that, for $0 < t <
    \infty$,
    \begin{gather}
        \label{eq:tower:nu-bounds}
        \norm{g_{\nu}}{\infty}{b,r'} \le \norm{\uD \uD_{s} f}{\infty}{b,\varrho'} \,,
        \quad
        \hnorm{\delta}{\zeta_{\nu}}{b,r'}{t} \le \hnorm{\delta}{\uD_i\Omega_{s}}{b,\varrho'}{t} \,,
        \\
        \hnorm{\delta}{\eta_{\nu}}{b,r'}{t} \le
        \hnorm{\delta}{\uD\uD_{s} f}{b,\varrho'}{t} \, \hnorm{\delta}{\uD_i A}{b,\varrho'}{t} \,,
    \end{gather}
    all right hand sides being finite by the preceding paragraph, because
    $\cball{b}{\varrho'}$ is a compact subset of $\oball{b}{\varrho}$.  Since $A$ satisfies~\ref{lem:H}\ref{i:H:assume-epsilon} on
    $\cball{b}{r'}$, applying~\ref{lem:H} to~\eqref{eq:tower:nu-pde} with $g_\nu$
    in place of~$g$, with $\zeta_\nu-\eta_\nu$ in place of~$\Omega$, and with
    inner and outer radii $r$ and~$r'$, we~conclude that $\hnorm{\delta}{\uD
      g_{\nu}}{b,r}{r'-r}$ is bounded independently of~$\nu$
    by~\eqref{eq:tower:nu-bounds}.  As $g_{\nu}
    \to \uD_i g$ uniformly on $\cball{b}{r}$, \ref{mr:H:cited}\ref{i:cited:closure}
    shows that $\uD_i \uD_{s} f = \uD_i g$ is of class~$\cnt{1}$ on
    $\oball{b}{r}$ with $\hoelder{\delta}{\uD\uD_i g|\cball{b}{r}} < \infty$,
    and that $\uD g_{\nu} \to \uD\uD_i g$ uniformly on compact subsets of
    $\oball br$.  Since $r < \varrho$, $s$ and~$i$ were arbitrary, $f$ is of
    class~$\cnt{q}$ on $\oball{b}{\varrho}$ with $\hoelder{\delta}{\uD^{q}f|K} <
    \infty$ for compact $K \subseteq \oball b\varrho$.  Finally, letting $t \in
    \Sequence{\vdim}{q-1}$ be defined by $t|\{1,\ldots,q-2\} = s$ and $t(q-1) =
    i$, the maps $\zeta_{\nu} - \eta_{\nu}$ converge uniformly on
    $\cball{b}{r}$ to~$\Omega_{t}$, by~\eqref{eq:ho:Omega-recursion} and the
    continuity of $\uD_i\Omega_s$, $\uD\uD_sf$ and~$\uD_iA$; passing to the
    limit in~\eqref{eq:tower:nu-pde} therefore gives the equation
    of~\ref{i:tower:Omega-s} for~$t$.  This completes the induction.

    \emph{Part~2: \ref{i:tower:Omega-s}.}  Let $p \in \{0,\ldots,q-1\}$ and $s
    \in \Sequence{\vdim}{p}$.  By~\ref{lem:ho:pde}\ref{i:ho:Omegas-explicit},
    $\Omega_s$ is the sum of $\uD_s\Omega$, of class~$\cnt{q-1-p,\delta}$
    by~\ref{i:tower:Omega}, and of products $\uD\uD_{\sigma}f \odot
    \uD_{\varsigma}A$ with $\HM^0(\dmn\sigma) = j$ and $\HM^0(\dmn\varsigma) =
    p-j$, $0 \le j \le p-1$; by Part~1 the first factor is of
    class~$\cnt{q-1-j}$ with $\delta$~H{\"o}lder top derivative on compact
    subsets, and by~\ref{i:tower:A} the second is of class~$\cnt{q-1-p+j,\delta}$.
    Since $\min \{ q-1-j ,\, q-1-p+j \} \ge q-1-p$ for $0 \le j \le p-1$, the
    assertion follows.

    \emph{Part~3: the estimate of~\ref{i:tower:est}.}  Let $0 < r < \varrho$, put
    $\tau = (\varrho-r)/(q+1)$ and $r_p = \varrho - (p+1)\tau$ for $p \in
    \{0,\ldots,q\}$, so that $r_0 < \varrho$ and $r_q = r$, and set
    \begin{displaymath}
        N_p = \max \bigl\{ \hnorm{\delta}{\uD\uD_{s}f}{b,r_{p-1}}{\tau} : s \in \Sequence{\vdim}{p-1} \bigr\}
        \quad \text{for $p \in \{1,\ldots,q\}$} \,,
    \end{displaymath}
    which is finite by Part~1 because $\cball{b}{r_{p-1}}$ is a compact subset
    of $\oball{b}{\varrho}$, and which for $p = 1$ satisfies $N_1 =
    \hnorm{\delta}{\uD f}{b,r_0}{\tau} \le \hnorm{\delta}{\uD
      f}{b,\varrho}{\varrho}$.  Let $p \in \{2,\ldots,q\}$ and
    $s \in \Sequence{\vdim}{p-1}$.  Applying~\ref{lem:H}, by way
    of~\eqref{eq:H:concl} and of the equation of~\ref{i:tower:Omega-s}, with $g = \uD_{s}f$,
    with~$\Omega_{s}$ in place of~$\Omega$, and with inner and outer radii
    $r_{p-1}$ and~$r_{p-2}$, and estimating $\norm{\uD_{s}f}{\infty}{b,r_{p-2}}
    \le \norm{\uD\uD_{s'}f}{\infty}{b,r_{p-2}} \le N_{p-1}$, where $s' =
    s|\{1,\ldots,p-1\} \without \{p-1\}$, we~obtain
    \begin{equation}
        \label{eq:tower:rec}
        N_{p} \le \Gamma_1 \bigl( \tau^{-1} N_{p-1} + W_{p-1} \bigr) \,,
        \quad \text{where} \quad
        W_{p-1} = \max \bigl\{ \hnorm{\delta}{\Omega_{s}}{b,r_{p-2}}{\tau} : s \in \Sequence{\vdim}{p-1} \bigr\}
    \end{equation}
    and $\Gamma_1 = \Gamma_{\ref{lem:H}}(\vdim,\adim,\delta,c,M)$.  Next, by
    \ref{lem:ho:pde}\ref{i:ho:Omegas-explicit}, by~\eqref{eq:hnorm:mult}, by
    $\HM^0(\Shuffle{j}{p-1-j}) = \binom{p-1}{j}$, and by
    $|\uD_{\sigma}F| \le |\uD^{i}F|$ for $\sigma \in \Sequence{\vdim}{i}$,
    \begin{displaymath}
        W_{p-1} \le \mathcal{O}_{p-1} + \textsum{j=0}{p-2} {\textstyle \binom{p-1}{j}} N_{j+1} \mathcal{A}_{p-1-j} \,,
    \end{displaymath}
    where $\mathcal{O}_{i} = \hnorm{\delta}{\uD^{i}\Omega}{b,\varrho}{\tau}$ and
    $\mathcal{A}_{i} = \hnorm{\delta}{\uD^{i}A}{b,\varrho}{\tau}$; here we~used
    that $r_{p-2} \le r_{j}$, hence $\cball b{r_{p-2}} \subseteq \cball b{r_j}$,
    for $j \le p-2$, so that
    $\hnorm{\delta}{\uD\uD_{\sigma}f}{b,r_{p-2}}{\tau} \le N_{j+1}$ whenever
    $\sigma \in \Sequence{\vdim}{j}$ and $j \le p-2$.  Inserting this
    into~\eqref{eq:tower:rec}, multiplying by~$\tau^{p-1}$, and abbreviating $n_p = \tau^{p-1}N_p$, $a_i =
    \tau^{i}\mathcal{A}_i$, $o_i = \tau^{i}\mathcal{O}_i$, and $a = \max \{ a_i
    : 1 \le i \le q-1 \}$, this becomes
    \begin{displaymath}
        n_{p} \le \Gamma_1 \Bigl( n_{p-1} + o_{p-1} + 2^{q} a \textsum{j=1}{p-1} n_{j} \Bigr)
        \le \Gamma_1 (1 + 2^{q} a) \Bigl( \textsum{j=1}{p-1} n_{j} + o_{p-1} \Bigr) \,;
    \end{displaymath}
    whence, by induction on~$p$,
    \begin{displaymath}
        \textsum{j=1}{q} n_{j} \le \bigl( 1 + \Gamma_1 (1+2^{q}a) \bigr)^{q}
        \Bigl( n_1 + \textsum{i=0}{q-1} o_{i} \Bigr) \,.
    \end{displaymath}
    Since $\tau \le \varrho$ we~have $a \le \mathcal{A}$ and $\sum_{i} o_i \le
    \mathcal{O}$, and since $\varrho / \tau = (q+1) \varrho/(\varrho-r)$ and
    $\norm{\uD^{p}f}{\infty}{} \le \vdim^{p/2} \max \{
    \norm{\uD\uD_{s}f}{\infty}{} : s \in \Sequence{\vdim}{p-1} \}$, together
    with the analogous inequality for the $\delta$~H{\"o}lder seminorms and with
    $r_q = r$, the left hand side of the estimate of~\ref{i:tower:est} is at most
    $\bigl( (q+1) \varrho/(\varrho-r) \bigr)^{q-1+\delta} \vdim^{q/2}
    \sum_{j=1}^{q} n_j$.
    Recalling $n_1 \le \hnorm{\delta}{\uD f}{b,\varrho}{\varrho}$
    we~obtain the estimate of~\ref{i:tower:est}.
\end{proof}


\begin{theorem}[\protect{cf.~\cite[5.2.15]{Federer1969}}]
    \label{thm:ho}
    Let $\vdim,\adim,\delta,c,M$ and $\varepsilon$ be as in~\ref{lem:H} and
    suppose
    \begin{enumerate}
    \item
        \label{i:ho:q}
        $q$ is an integer, $q \ge 2$;
    \item
        \label{i:ho:G}
        $V$ is an open subset of $\R^{\vdim} \times \R^{\adim-\vdim} \times
        \Hom(\R^{\vdim},\R^{\adim-\vdim})$ and $G : V \to \R$ is of
        class~$\cnt{q+1}$;
    \item
        \label{i:ho:f}
        $U$ is an open subset of~$\R^{\vdim}$, $f : U \to \R^{\adim-\vdim}$ is
        of class~$\cnt{1}$, $\hoelder{\delta}{\uD f} < \infty$, and $\psi(x) =
        \bigl( x ,\, f(x) ,\, \uD f(x) \bigr) \in V$ whenever $x \in U$;
    \item
        \label{i:ho:el}
        $\textint{U}{} \bigl\langle (0,\theta(x),\uD\theta(x)) \,, \uD G[\psi(x)]
        \bigr\rangle \ud \LM^{\vdim}x = 0$ whenever $\theta \in
        \dspace{U}{\R^{\adim-\vdim}}$;
    \item
        \label{i:ho:A}
        $A : U \to {\textstyle \bigodot^2} \Hom(\R^{\vdim},\R^{\adim-\vdim})$
        satisfies $\bigl\langle \sigma \odot \tau ,\, A(x) \bigr\rangle =
        \bigl\langle (0,0,\sigma) \odot (0,0,\tau) \,, \uD^2 G[\psi(x)]
        \bigr\rangle$ whenever $x \in U$ and $\sigma,\tau \in
        \Hom(\R^{\vdim},\R^{\adim-\vdim})$;
    \item
        \label{i:ho:ball}
        $b \in \R^{\vdim}$, $0 < \varrho < \infty$, $\cball{b}{\varrho}
        \subseteq U$, and $Z$ is a convex subset of~$V$ such that $\psi \lIm
        \cball{b}{\varrho} \rIm \subseteq Z$ and
        \begin{displaymath}
            \alpha_p = \sup \bigl\{ \| \uD^{p} G(w) \| : w \in Z \bigr\} < \infty
            \quad \text{for $p \in \{0,1,\ldots,q+1\}$} \,;
        \end{displaymath}
    \item
        \label{i:ho:ell}
        $\bigl\langle \sigma \odot \tau ,\, \Upsilon(w) \bigr\rangle = \bigl\langle
        (0,0,\sigma) \odot (0,0,\tau) \,, \uD^2 G(w) \bigr\rangle$ defines a
        map $\Upsilon : Z \to W(c,M)$;
    \item
        \label{i:ho:small}
        $\varrho^{\delta} \alpha_3 \hoelder{\delta}{\psi|\cball{b}{\varrho}} \le \varepsilon$.
    \end{enumerate}
    Then, abbreviating
    \begin{displaymath}
        E = \hnorm{\delta}{\uD f}{b,\varrho}{\varrho} \,,
        \quad
        \Lambda = \max \{ 1 ,\, \varrho \} \cdot \max \{ 1 ,\, \alpha_0 ,\, \ldots ,\, \alpha_{q+1} \} \,,
        \quad
        \varrho_p = \varrho ( \tfrac12 + 2^{-p} ) \,,
    \end{displaymath}
    the following statements hold.
    \begin{enumerate}[resume]
    \item
        \label{i:ho:class}
        $f$ is of class~$\cnt{q}$ on $\oball{b}{\varrho}$ and
        $\hoelder{\delta}{\uD^{q}f|K} < \infty$ for every compact $K \subseteq
        \oball{b}{\varrho}$.
    \item
        \label{i:ho:Omega-s}
        To each $s \in \Sequence{\vdim}{p}$ with $p \in \{1,\ldots,q-1\}$
        corresponds $\Omega_s : \oball{b}{\varrho} \to
        \Hom(\R^{\vdim},\R^{\adim-\vdim})$ of class~$\cnt{q-p}$ such that
        $\hoelder{\delta}{\uD^{q-p}\Omega_s|K} < \infty$ for every compact $K
        \subseteq \oball{b}{\varrho}$ and
        \begin{displaymath}
            \textint{}{} \bigl\langle \uD\theta \odot \uD\uD_{s} f \,, A \bigr\rangle \ud \LM^{\vdim}
            = \textint{}{} \uD\theta \bullet \Omega_{s} \ud \LM^{\vdim}
            \quad \text{for $\theta \in \dspace{\oball{b}{\varrho}}{\R^{\adim-\vdim}}$} \,.
        \end{displaymath}
    \item
        \label{i:ho:Omega-1}
        $\Omega_{(i)}(x) \bullet \sigma = \bigl\langle (0,\sigma(e_i),0) \,, \uD
        G[\psi(x)] \bigr\rangle - \bigl\langle (0,0,\sigma) \odot (e_i, \uD_i
        f(x), 0) \,, \uD^2 G[\psi(x)] \bigr\rangle$ whenever $i \in
        \{1,\ldots,\vdim\}$, $x \in \oball{b}{\varrho}$ and $\sigma \in
        \Hom(\R^{\vdim},\R^{\adim-\vdim})$.
    \item
        \label{i:ho:Omega-rec}
        $\Omega_{s}$ is obtained from $\Omega_{(i)}$ by the
        recursion~\eqref{eq:ho:Omega-recursion} of~\ref{mr:ho:Omega}, applied
        with $\uD_i f$ in place of~$f$.
    \item
        \label{i:ho:est}
        Defining $\Gamma_1 = E$ and, recursively,
        \begin{displaymath}
            \Gamma_p = \Delta \bigl( \Lambda ( 1 + \Gamma_{p-1} ) \bigr)^{(p+1)^2}
            \quad \text{for $p \in \{2,\ldots,q\}$} \,,
        \end{displaymath}
        where $\Delta = \Delta(\vdim,\adim,\delta,c,M,q) < \infty$ is a suitable
        constant, there holds
        \begin{displaymath}
            \textsum{p=1}{q} \varrho^{p-1} \hnorm{\delta}{\uD^{p}f}{b,\varrho_q}{\varrho} \le \Gamma_q \,;
        \end{displaymath}
        in particular $\cball{b}{\varrho/2} \subseteq \cball{b}{\varrho_q}$
        yields the same bound on $\cball{b}{\varrho/2}$.
    \item
        \label{i:ho:homog}
        In~case $G$ is independent of its first two variables one has
        $\Omega_{(i)} = 0$; and if, in~this case, \ref{i:ho:small} is replaced
        by the requirement that $\alpha_3 E \le \varepsilon$ and $E \le 1$, then
        \ref{i:ho:class}--\ref{i:ho:Omega-rec} remain valid and the estimate above
        holds with $\Gamma_q = \Delta_0 \, E$, where $\Delta_0 =
        \Delta_0(\vdim,\adim,\delta,c,M,q,\Lambda_0) < \infty$ and $\Lambda_0 =
        \max \{ 1, \alpha_2, \ldots, \alpha_{q+1} \}$.
    \end{enumerate}
\end{theorem}

\begin{proof}
    We~abbreviate $\ell = \adim-\vdim$ and $H =
    \varrho^{\delta}\hoelder{\delta}{\psi|\cball{b}{\varrho}}$, and note for
    later use that, since $\hoelder{\delta}{f|\cball b\varrho} \le
    (2\varrho)^{1-\delta} \norm{\uD f}{\infty}{b,\varrho}$ and
    $\hoelder{\delta}{\mathrm{id}|\cball b\varrho} \le (2\varrho)^{1-\delta}$,
    \begin{equation}
        \label{eq:ho:H-vs-E}
        H \le 2\varrho ( 1 + \norm{\uD f}{\infty}{b,\varrho} )
        + \varrho^{\delta}\hoelder{\delta}{\uD f|\cball b\varrho}
        \le 2 \Lambda (1+E) \,.
    \end{equation}
    Observe also that $W(c,M)$ is convex, the function
    $\ellipticity{\cdot}$ being an infimum of linear functions and the norm
    being convex.

    \emph{Step 1: the case $q = 2$.}  Fix $i \in \{1,\ldots,\vdim\}$, put $R =
    \tfrac78\varrho$ and $\tau = \tfrac18\varrho$, and define, for integers $\nu
    > 8/\varrho$, maps on $\cball{b}{R}$ by
    \begin{gather}
        f_{\nu}(x) = \nu \bigl( f(x+\nu^{-1}e_i) - f(x) \bigr) \,,
        \quad
        \psi_{\nu}(x) = \nu \bigl( \psi(x+\nu^{-1}e_i) - \psi(x) \bigr) = \bigl( e_i, f_{\nu}(x), \uD f_{\nu}(x) \bigr) \,,
        \\
        \phi_{\nu}(t,x) = (1-t) \psi(x) + t \psi(x+\nu^{-1}e_i) \,,
        \\
        \bigl\langle \sigma \odot \tau ,\, A_{\nu}(x) \bigr\rangle
        = \textint{0}{1} \bigl\langle (0,0,\sigma) \odot (0,0,\tau) \,,
        \uD^2 G[\phi_{\nu}(t,x)] \bigr\rangle \ud \LM^1 t \,,
        \\
        \sigma \bullet P_{\nu}(x)
        = \textint{0}{1} \bigl\langle (0,\sigma(e_i),0) \,,
        \uD G[\psi(x+t\nu^{-1}e_i)] \bigr\rangle \ud \LM^1 t \,,
        \\
        \sigma \bullet Q_{\nu}(x)
        = \textint{0}{1} \bigl\langle (0,0,\sigma) \odot (e_i, f_{\nu}(x), 0) \,,
        \uD^2 G[\phi_{\nu}(t,x)] \bigr\rangle \ud \LM^1 t \,;
    \end{gather}
    these are well defined because $Z$ is convex, so that $\phi_{\nu}(t,x) \in
    Z$.  Testing~\ref{i:ho:el} with $\theta_{\nu}(x) = \nu \bigl(
    \theta(x-\nu^{-1}e_i) - \theta(x) \bigr)$, where $\theta \in
    \dspace{\oball{b}{R}}{\R^{\ell}}$, and translating the variable of
    integration --- the computation is that of~\cite[p.~556]{Federer1969} ---
    one obtains
    \begin{equation}
        \label{eq:ho:base-nu}
        \textint{}{} \bigl\langle \uD\theta \odot \uD f_{\nu} \,, A_{\nu} \bigr\rangle \ud \LM^{\vdim}
        = \textint{}{} \uD\theta \bullet ( P_{\nu} - Q_{\nu} ) \ud \LM^{\vdim}
        \quad \text{for $\theta \in \dspace{\oball{b}{R}}{\R^{\ell}}$} \,.
    \end{equation}
    By~\ref{i:ho:ell} and the convexity of $W(c,M)$ the map~$A_{\nu}$ takes its
    values in $W(c,M)$, and
    \begin{displaymath}
        R^{\delta} \hoelder{\delta}{A_{\nu}|\cball{b}{R}}
        \le \varrho^{\delta} \alpha_3 \hoelder{\delta}{\psi|\cball{b}{\varrho}} \le \varepsilon
    \end{displaymath}
    by~\ref{i:ho:small}, since $\hoelder{\delta}{\phi_{\nu}(t,\cdot)|\cball bR}
    \le \hoelder{\delta}{\psi|\cball b\varrho}$ for $0 \le t \le 1$.  Moreover
    $\norm{f_{\nu}}{\infty}{b,R} \le \norm{\uD f}{\infty}{b,\varrho}$,
    $\hoelder{\delta}{f_{\nu}|\cball bR} \le \hoelder{\delta}{\uD
      f|\cball{b}{\varrho}}$, and therefore, using $\tau \le \varrho$ and
    $|(e_i,y,0)| \le 1 + |y|$,
    \begin{equation}
        \label{eq:ho:PQ}
        \hnorm{\delta}{P_{\nu}-Q_{\nu}}{b,R}{\tau}
        \le \alpha_1 + \alpha_2 H + (1+E)(\alpha_2 + \alpha_3 H) + \alpha_2 E
        \le 8 \Lambda^2 (1+E)^2
    \end{equation}
    by~\eqref{eq:ho:H-vs-E}.  Since $f_{\nu}$ is of class~$\cnt1$ with
    $\hoelder{\delta}{\uD f_{\nu}|\cball bR} \le 2\nu\hoelder{\delta}{\uD
      f|\cball b\varrho} < \infty$, we~may apply~\ref{lem:H}
    to~\eqref{eq:ho:base-nu}, with $A_{\nu}$, $f_{\nu}$ and $P_{\nu}-Q_{\nu}$ in
    place of $A$, $g$ and~$\Omega$, and with inner and outer radii $\varrho_2 =
    \tfrac34\varrho$ and~$R$; by~\eqref{eq:H:concl} and~\eqref{eq:ho:PQ},
    \begin{equation}
        \label{eq:ho:base-est}
        \hnorm{\delta}{\uD f_{\nu}}{b,\varrho_2}{\tau}
        \le \Gamma_{\ref{lem:H}} \bigl( 8 \varrho^{-1} E + 8 \Lambda^2 (1+E)^2 \bigr) \,,
    \end{equation}
    a bound independent of~$\nu$.  Since $f_{\nu} \to \uD_i f$ uniformly on
    $\cball{b}{\varrho_2}$, \ref{mr:H:cited}\ref{i:cited:closure} shows that
    $\uD_i f$ is of class~$\cnt1$ on $\oball{b}{\varrho_2}$, that
    $\hoelder{\delta}{\uD\uD_i f|\cball{b}{\varrho_2}} < \infty$ with the bound
    \eqref{eq:ho:base-est}, and that $\uD f_{\nu} \to \uD\uD_i f$ uniformly on
    compact subsets of $\oball{b}{\varrho_2}$.  Since $A_{\nu} \to A$ and, by
    \ref{i:ho:Omega-1}, $P_{\nu} - Q_{\nu} \to \Omega_{(i)}$ uniformly on
    $\cball{b}{\varrho_2}$, passing to the limit in~\eqref{eq:ho:base-nu} yields
    \begin{equation}
        \label{eq:ho:first-pde}
        \textint{}{} \bigl\langle \uD\theta \odot \uD\uD_i f \,, A \bigr\rangle \ud \LM^{\vdim}
        = \textint{}{} \uD\theta \bullet \Omega_{(i)} \ud \LM^{\vdim}
        \quad \text{for $\theta \in \dspace{\oball{b}{\varrho_2}}{\R^{\ell}}$} \,,
    \end{equation}
    which is~\ref{i:ho:Omega-s} for $p = 1$, and identifies $\Omega_{(i)}$
    as in~\ref{i:ho:Omega-1}.  Finally, as $\norm{\uD^2f}{\infty}{} \le \vdim
    \max_i \norm{\uD\uD_i f}{\infty}{}$ and $\varrho^{\delta} =
    8^{\delta}\tau^{\delta}$, \eqref{eq:ho:base-est} gives $\Delta_1 =
    \Delta_1(\vdim)$ with
    \begin{equation}
        \label{eq:ho:F2}
        \varrho \hnorm{\delta}{\uD^2 f}{b,\varrho_2}{\varrho}
        \le \Delta_1 \Gamma_{\ref{lem:H}} \bigl( E + \varrho \Lambda^2 (1+E)^2 \bigr)
        \le \Delta_2 \bigl( \Lambda (1+E) \bigr)^{3} \,,
    \end{equation}
    with $\Delta_2 = \Delta_2(\vdim,\adim,\delta,c,M)$, since $\varrho \le
    \Lambda$ and $\Lambda \ge 1$.

    The hypotheses \ref{i:ho:ball}--\ref{i:ho:small} continue to hold with $b$
    and~$\varrho$ replaced by~$b'$ and~$\varrho'$ whenever $\cball{b'}{\varrho'}
    \subseteq \cball{b}{\varrho}$; applying the preceding argument at every $b'
    \in \oball{b}{\varrho}$ with $\varrho' = \varrho - |b'-b|$, we~conclude that
    $f$ is of class~$\cnt2$ on $\oball{b}{\varrho}$, that
    $\hoelder{\delta}{\uD^2f|K} < \infty$ for every compact $K \subseteq
    \oball{b}{\varrho}$, and that~\eqref{eq:ho:first-pde} holds for $\theta \in
    \dspace{\oball{b}{\varrho}}{\R^{\ell}}$.  Together
    with~\ref{i:ho:Omega-1} --- which is the definition of~$\Omega_{(i)}$ ---
    and with~\eqref{eq:ho:F2} this proves the theorem for $q = 2$; note
    $\Gamma_2 = \Delta(\Lambda(1+E))^{9} \ge E + \Delta_2 (\Lambda(1+E))^{3}$
    for a suitable~$\Delta$.

    \emph{Step 2: induction on~$q$.}  Let $q \ge 3$ and assume the theorem
    proved with $q-1$ in place of~$q$; as $G$ is of class~$\cnt{q}$ and
    $\alpha_p < \infty$ for $p \le q$, it applies to the present data and yields
    that $f$ is of class~$\cnt{q-1}$ on $\oball{b}{\varrho}$ with
    $\hoelder{\delta}{\uD^{q-1}f|K} < \infty$ for compact $K \subseteq
    \oball{b}{\varrho}$, that~\eqref{eq:ho:first-pde} holds on
    $\oball{b}{\varrho}$, and that
    \begin{equation}
        \label{eq:ho:Fq-1}
        F_{q-1} := \textsum{p=1}{q-1} \varrho^{p-1} \hnorm{\delta}{\uD^{p}f}{b,\varrho_{q-1}}{\varrho}
        \le \Gamma_{q-1} \,.
    \end{equation}
    Consequently $\psi$ is of class~$\cnt{q-2,\delta}$ on every closed ball
    $\cball{b}{\varrho'}$ with $\varrho' < \varrho$; hence, so are $A = \Upsilon
    \circ \psi$ and, by~\ref{i:ho:Omega-1}, $\Omega_{(i)}$.
    Fix $i \in \{1,\ldots,\vdim\}$ and apply~\ref{lem:ho:tower}, at order $q-1$
    and on $\cball{b}{\varrho_{q-1}}$, to $g = \uD_i f$ and to $A$,
    $\Omega_{(i)}$ --- admissible by~\eqref{eq:ho:first-pde},
    by~\ref{i:ho:ell}, and because $\varrho_{q-1}^{\delta}
    \hoelder{\delta}{A|\cball{b}{\varrho_{q-1}}} \le \varrho^{\delta} \alpha_3
    \hoelder{\delta}{\psi|\cball b\varrho} \le \varepsilon$.  Its
    conclusion~\ref{lem:ho:tower}\ref{i:tower:reg} gives that $\uD_i f$ is of
    class~$\cnt{q-1}$ on $\oball{b}{\varrho_{q-1}}$ with $\delta$~H{\"o}lder top
    derivative on compact subsets, and
    \ref{lem:ho:tower}\ref{i:tower:Omega-s} gives~\ref{i:ho:Omega-s}
    and~\ref{i:ho:Omega-rec}; since $i$~was arbitrary and since the same
    argument applies with $b,\varrho$ replaced by $b',\varrho'$ as at the end of
    Step~1, we~obtain~\ref{i:ho:class} on all of $\oball{b}{\varrho}$.  The
    sharper class assertion in~\ref{i:ho:Omega-s} follows, as in Part~2 of the
    proof of~\ref{lem:ho:tower}, from
    \ref{lem:ho:pde}\ref{i:ho:Omegas-explicit}: for $s \in \Sequence{\vdim}{p}$
    the map $\Omega_s$ is the sum of derivatives of order $p-1$ of
    $\Omega_{(i)}$, which is of class~$\cnt{q-1,\delta}$ on compact subsets of
    $\oball b\varrho$ by~\ref{i:ho:Omega-1} and~\ref{i:ho:class}, and of
    products $\uD\uD_{\sigma}\uD_i f \odot \uD_{\varsigma}A$ --- the recursion
    being applied with $\uD_i f$ in place of~$f$ --- with $\sigma \in
    \Sequence{\vdim}{j}$, $j \le p-2$, and $\varsigma \in
    \Sequence{\vdim}{p-1-j}$; the first factor is of class~$\cnt{q-j-2}$ and the
    second of class~$\cnt{q-p+j,\delta}$, so that all summands are of
    class~$\cnt{q-p}$ with $\delta$~H{\"o}lder top derivative on compact
    subsets.

    It remains to prove the estimate of~\ref{i:ho:est}.  Abbreviate $\Pi = \Lambda(1 +
    F_{q-1})$.  Since $\uD\psi = (\mathrm{id}_{\R^{\vdim}}, \uD f, \uD^2 f)$ and
    $\uD^{j}\psi = (0, \uD^{j}f, \uD^{j+1}f)$ for $j \ge 2$, \eqref{eq:ho:Fq-1}
    yields
    \begin{equation}
        \label{eq:ho:psi-bound}
        \varrho^{j} \hnorm{\delta}{\uD^{j}\psi}{b,\varrho_{q-1}}{\varrho} \le 2 \Pi
        \quad \text{for $1 \le j \le q-2$} \,.
    \end{equation}
    Hence, by the chain rule \cite[3.1.11]{Federer1969} applied to $\Upsilon$
    and~$\psi$, by~\ref{i:ho:ball}, by~\eqref{eq:ho:psi-bound}, and
    by~\eqref{eq:hnorm:mult}, there is
    $\Delta_3 = \Delta_3(\vdim,\adim,q)$ such that the quantities $\mathcal{A}$
    and~$\mathcal{O}$ of~\ref{lem:ho:tower}\ref{i:tower:est}, formed at order
    $q-1$ on $\cball{b}{\varrho_{q-1}}$ from $A$ and~$\Omega_{(i)}$, satisfy
    \begin{equation}
        \label{eq:ho:AO}
        \mathcal{A} + \mathcal{O} \le \Delta_3 \, \Pi^{q+1} \,;
    \end{equation}
    here we~used that $\uD^{j}\Omega_{(i)}$ is, by~\ref{i:ho:Omega-1}, a
    universal polynomial expression in $\uD^{p}G \circ \psi$ with $p \le j+2$,
    in $\uD^{l}\psi$ with $l \le j$, and in $\uD^{l}\uD_i f$ with $l \le j$, and
    that $\varrho_{q-1} \ge \varrho/2$.  Applying \ref{lem:ho:tower}\ref{i:tower:est} with $g =
    \uD_i f$, with $\varrho_{q-1}$ in place of~$\varrho$ and with $r =
    \varrho_q$ --- so that $\varrho_{q-1}/(\varrho_{q-1}-\varrho_q) \le 2^{q-1}
    + 2$ --- multiplying by~$\varrho$, taking the maximum over~$i$, and
    using $\varrho \hnorm{\delta}{\uD\uD_i f}{b,\varrho_{q-1}}{\varrho} \le
    \varrho \hnorm{\delta}{\uD^2 f}{b,\varrho_{q-1}}{\varrho} \le F_{q-1}$,
    we~find $\Delta_4 = \Delta_4(\vdim,\adim,\delta,c,M,q)$ with
    \begin{displaymath}
        \textsum{p=2}{q} \varrho^{p-1} \hnorm{\delta}{\uD^{p}f}{b,\varrho_q}{\varrho}
        \le \Delta_4 \bigl( 1 + \Delta_3 \Pi^{q+1} \bigr)^{q-1} \bigl( F_{q-1} + \Lambda \mathcal{O} \bigr) \,,
    \end{displaymath}
    where the factor~$\Lambda$ absorbs the factor~$\varrho$ multiplying
    $\mathcal{O}$.  Adding the term $p = 1$, which is at most~$F_{q-1}$, and
    using~\eqref{eq:ho:AO} together with $\Pi \ge 1$ and $F_{q-1} \le \Pi$,
    we~obtain $\Delta_5 = \Delta_5(\vdim,\adim,\delta,c,M,q)$ with
    \begin{displaymath}
        \textsum{p=1}{q} \varrho^{p-1} \hnorm{\delta}{\uD^{p}f}{b,\varrho_q}{\varrho}
        \le \Delta_5 \, \Pi^{(q+1)(q-1)} \, \Pi^{q+2}
        = \Delta_5 \, \Pi^{q^2+q+1}
        \le \Delta_5 \, \Pi^{(q+1)^2} \,.
    \end{displaymath}
    Since $\Pi = \Lambda(1+F_{q-1}) \le \Lambda(1+\Gamma_{q-1})$
    by~\eqref{eq:ho:Fq-1}, this is the estimate of~\ref{i:ho:est}, with the
    recursion stated there, provided $\Delta \ge \Delta_5$.

    \emph{Step 3: the case of \ref{i:ho:homog}.}  Suppose $G$ is independent of
    its first two variables.  Then $\uD G[\psi(x)]$ has vanishing components in
    the first two variables, so the first summand in~\ref{i:ho:Omega-1}
    vanishes; the second vanishes as well, because $\uD^2 G[\psi(x)]$ pairs only
    the third components and the second argument $(e_i,\uD_if(x),0)$ has
    vanishing third component.  Thus $\Omega_{(i)} = 0$.  Moreover, $A =
    \Upsilon' \circ \uD f$, where $\Upsilon'$ is the map induced by~$\uD^2 G$ on
    the third variable, so that
    \begin{displaymath}
        \varrho^{\delta} \hoelder{\delta}{A|\cball{b}{\varrho}}
        \le \alpha_3 \varrho^{\delta} \hoelder{\delta}{\uD f|\cball b\varrho}
        \le \alpha_3 E \le \varepsilon \,,
    \end{displaymath}
    which is the hypothesis of~\ref{lem:H} and of~\ref{lem:ho:tower} needed in
    Steps~1 and~2; the sole use of~\ref{i:ho:small} there was to secure it.
    Step~1 goes through unchanged and furnishes~\eqref{eq:ho:first-pde},
    together with the fact that $f$ is of class~$\cnt{2}$ on
    $\oball{b}{\varrho}$; since now $\Omega_{(i)} = 0$, applying~\ref{lem:H}
    to~$\uD_i f$ and~\eqref{eq:ho:first-pde}, with inner and outer radii
    $\varrho_2$ and~$R$, gives
    \begin{displaymath}
        \hnorm{\delta}{\uD\uD_i f}{b,\varrho_2}{\tau}
        \le \Gamma_{\ref{lem:H}} \, \tau^{-1} \norm{\uD_i f}{\infty}{b,R}
        \le 8 \Gamma_{\ref{lem:H}} \, \varrho^{-1} E \,;
    \end{displaymath}
    whence, $\varrho \hnorm{\delta}{\uD^2f}{b,\varrho_2}{\varrho} \le \Delta_6
    E$ in place of~\eqref{eq:ho:F2}.  In Step~2 one has $\mathcal{O} = 0$, and
    since $\uD^{j}\psi$ may there be replaced by $\uD^{j}\uD f$ --- so that
    $\varrho^{j}\hnorm{\delta}{\uD^{j}\uD f}{b,\varrho_{q-1}}{\varrho} \le
    F_{q-1}$ for $1 \le j \le q-2$, without the factor $\max\{1,\varrho\}$ ---
    the bound~\eqref{eq:ho:AO} improves to $\mathcal{A} \le \Delta_3 \Lambda_0
    (1+F_{q-1})^{q}$.  Consequently, if $E \le 1$ and if $F_{q-1} \le \Delta_0'
    E \le \Delta_0'$ by the inductive hypothesis, then
    \begin{displaymath}
        \textsum{p=1}{q} \varrho^{p-1} \hnorm{\delta}{\uD^{p}f}{b,\varrho_q}{\varrho}
        \le F_{q-1} + \Delta_4 \bigl( 1 + \Delta_3 \Lambda_0 (1+\Delta_0')^{q} \bigr)^{q-1} F_{q-1}
        \le \Delta_0 E
    \end{displaymath}
    with $\Delta_0 = \Delta_0(\vdim,\adim,\delta,c,M,q,\Lambda_0)$, the estimate
    being \emph{linear} in~$E$ because the inhomogeneity $\mathcal{O}$ vanishes.
\end{proof}

\begin{lemma}
    \label{lem:schauder-est}
    Assume
    \begin{gather}
        \vdim, \adim, k \in \natp \,,
        \quad
        \vdim \le \adim \,,
        \quad
        k \ge 2 \,,
        \quad
        r \in (0,\infty) \,,
        \quad
        \eta, \lambda \in (0,1) \,,
        \\
        G : \Hom(\R^{\vdim},\R^{\adim - \vdim}) \to \R
        \quad \text{is of class~$\cnt{\infty}$ and satisfies}
        \\
        \lambda |\tau|^2 \le \langle \tau \odot \tau ,\, \uD^2 G(\sigma) \rangle \le \lambda^{-1} |\tau|^2
        \quad \text{for $\sigma,\tau \in \Hom(\R^{\vdim},\R^{\adim - \vdim})$} \,,
        \\
        u : \oball 0r \to \R^{\adim - \vdim}
        \quad \text{is of class~$\cnt{\infty}$ and satisfies} \quad
        L_{G}(u) = 0 \,,
        \\
        r^{1/2} \hoelder{1/2}{\uD u|{\cball 0r}} + \norm{\uD u}{\infty}{0,r} \le \eta \,.
    \end{gather}
    There exist $\eta_0 = \eta_0(\adim,\vdim,\lambda,G) \in (0,1)$ and $\Gamma =
    \Gamma(\adim,\vdim,\lambda,k,G)$, depending only on the indicated
    quantities, such that if $\eta \le \eta_0$ then
    \begin{equation}
        \label{eq:schauder-iter}
        \sum_{i=1}^{k} \Bigl( r^{i-1} \norm{\uD^i u}{\infty}{0, r/4}
        + r^{i-1/2} \hoelder{1/2}{\uD^i u|{\cball{0}{r/4}}} \Bigr)
        \le \Gamma \eta \,.
    \end{equation}
\end{lemma}

\begin{proof}
    We~may assume $\vdim < \adim$, the case $\vdim = \adim$ being trivial.  This
    is the case of~\ref{thm:ho}\ref{i:ho:homog}, with $q = k$, $\delta = 1/2$,
    $c = \lambda/2$, $M = 2\lambda^{-1}$, $b = 0$, and $\varrho = r/2$.  Indeed,
    put $\ell = \adim-\vdim$, let
    \begin{gather}
        V = \R^{\vdim} \times \R^{\ell} \times \Hom(\R^{\vdim},\R^{\ell}) \,,
        \quad
        \tilde G(x,y,\tau) = G(\tau) \quad \text{for $(x,y,\tau) \in V$} \,,
        \\
        Z = \R^{\vdim} \times \R^{\ell} \times \bigl\{ \tau : |\tau| \le 1 \bigr\} \,,
        \quad
        \Lambda_G^{(p)} = \sup \bigl\{ \| \uD^{p} G(\tau) \| : |\tau| \le 1 \bigr\} \,,
    \end{gather}
    and note that $Z$ is convex, that $\tilde G$ is of class~$\cnt{k+1}$, that
    $\alpha_p = \Lambda_G^{(p)} < \infty$ for $p \in \{0,\ldots,k+1\}$ --- so
    that $\Lambda_0 = \max\{1,\Lambda_G^{(2)},\ldots,\Lambda_G^{(k+1)}\}$
    depends only on $k$ and~$G$ --- and that $\tilde G$ is independent of its
    first two variables.  Since $L_G(u) = 0$ means, by~\ref{def:ell-system},
    that $\textint{}{} \langle \uD\theta ,\, \uD G(\uD u) \rangle \ud
    \LM^{\vdim} = 0$ for $\theta \in \dspace{\oball 0r}{\R^{\ell}}$, and since
    $\uD \tilde G[\psi(x)]$ has vanishing components in the first two variables,
    hypothesis~\ref{thm:ho}\ref{i:ho:el} holds; hypotheses
    \ref{thm:ho}\ref{i:ho:q}--\ref{i:ho:A} are immediate, and $A = \uD^2 G \circ
    \uD u$.  The ellipticity hypothesis on~$G$ yields
    $\ellipticity{\uD^2G(\sigma)} \ge \lambda > c$ and $\|\uD^2G(\sigma)\| \le
    \lambda^{-1} < M$, and $c < M$ because $\lambda < 1$; hence,
    \ref{thm:ho}\ref{i:ho:ell} holds with the map $\Upsilon$ induced by $\uD^2
    G$.  Finally $\norm{\uD u}{\infty}{0,r} \le \eta \le 1$ gives $\psi \lIm
    \cball{0}{\varrho} \rIm \subseteq Z$, and
    \begin{displaymath}
        E = \hnorm{1/2}{\uD u}{0,\varrho}{\varrho}
        \le \norm{\uD u}{\infty}{0,r} + r^{1/2} \hoelder{1/2}{\uD u|\cball 0r} \le \eta \,.
    \end{displaymath}
    Choosing $\eta_0 = \eta_0(\adim,\vdim,\lambda,G) \in (0,1)$ so small that
    $\Lambda_G^{(3)} \eta_0 \le \varepsilon_{\ref{lem:H}}(\vdim,\adim,\tfrac12,c,M)$
    --- a requirement which does \emph{not} involve~$k$ --- the hypotheses
    of~\ref{thm:ho}\ref{i:ho:homog} are met, and~\ref{thm:ho}\ref{i:ho:est} yields
    \begin{displaymath}
        \textsum{p=1}{k} \varrho^{p-1} \hnorm{1/2}{\uD^{p}u}{0,\varrho/2}{\varrho}
        \le \Delta_0 \, E \le \Delta_0 \, \eta
        \quad \text{with} \quad
        \Delta_0 = \Delta_0(\vdim,\adim,\tfrac12,c,M,k,\Lambda_0) \,.
    \end{displaymath}
    Since $r/4 = \varrho/2$ --- a radius \emph{independent of~$k$}, because
    $\varrho_k \ge \varrho/2$ in~\ref{thm:ho}\ref{i:ho:est} --- since $r^{p-1} =
    2^{p-1}\varrho^{p-1}$, and since
    $r^{1/2} = 2^{1/2}\varrho^{1/2}$, the left hand side
    of~\eqref{eq:schauder-iter} is at most $2^{k}$ times the left hand side of
    the last display.  Thus~\eqref{eq:schauder-iter} holds with $\Gamma =
    2^{k}\Delta_0$, which depends only on $\adim$, $\vdim$, $\lambda$, $k$
    and~$G$.
\end{proof}

\section{Auxiliary results for the construction of approximating functions}

\begin{definition}
    \label{def:area-integrand}
    We define $\Phi : \Hom(\R^{\vdim},\R^{\adim - \vdim}) \to \R$ by the formula
    \begin{displaymath}
        {\textstyle \Phi(\tau) = \bigl( \sum_{i=0}^{\vdim} |\tbwedge_i \tau|^2 \bigr)^{1/2} }
        \quad \text{for $\tau \in \Hom(\R^{\vdim}, \R^{\adim - \vdim})$} \,.
    \end{displaymath}
\end{definition}

\begin{remark}
    \label{rem:area-integrand}
    Note that the map $\R^{\adim} \times \Hom(\R^{\vdim}, \R^{\adim - \vdim})
    \ni (x,\tau) \mapsto \Phi(\tau)$ is the nonparametric integrand associated
    to the area integrand; see~\cite[5.1.1, 5.1.9]{Federer1969}. Observe also
    that $\Phi(0) = 1$, $\uD \Phi(0) = 0$, and $\langle \sigma \odot \rho ,\,
    \uD^2 \Phi(0) \rangle = \sigma \bullet \rho$ whenever $\sigma, \rho \in
    \bigodot^1(\R^{\vdim}, \R^{\adim - \vdim})$.
\end{remark}

\begin{lemma}
    \label{lem:fake-integrand}
    Let $0 < \varepsilon \le 1$ and $\Phi$ be as in~\ref{def:area-integrand}.
    There exist $0 < \eta \le 1$ and $G : \Hom(\R^{\vdim},\R^{\adim - \vdim})
    \to \R$ of class~$\cnt{\infty}$ such that for $\sigma,\tau \in
    \Hom(\R^{\vdim},\R^{\adim - \vdim})$ and $i \in \{0,1,2\}$ there holds
    \begin{gather}
        \uD^i \Phi(\sigma) = \uD^i G(\sigma)
        \quad \text{given $|\sigma| \le \eta$} \,,
        \quad
        \| \uD^2G(\sigma) - \uD^2 \Phi(0) \|
        \le \varepsilon \,,
        \\
        G(\sigma) = 1 + \tfrac 12 |\sigma|^2
        \quad \text{given $|\sigma| > 2 \eta$} \,,
        \quad
        \langle \tau \odot \tau  ,\, \uD^2 G(\sigma) \rangle \ge (1 - \varepsilon) |\tau|^2 \,,
        \quad
        \| \uD^2 G(\sigma) \| \le 1 + \varepsilon \,.
    \end{gather}
\end{lemma}

\begin{proof}
    Let $\Gamma = \Gamma_{\text{\cite[3.21]{Menne2013}}}$ be the number depending
    on~$\adim$ and~$\vdim$ appearing in~\cite[3.21]{Menne2013}. Since $\Phi$ is
    of class~$\cnt{\infty}$ we see that there exists $0 < \eta \le 1$ such that
    \begin{displaymath}
        s = \sup \left\{ \| \uD^2 \Phi(\sigma) - \uD^2 \Phi(0) \| : \sigma \in \Hom(\R^{\vdim},\R^{\adim - \vdim}) ,\, |\sigma| \le 2 \eta \right\}
        \le  \varepsilon\Gamma^{-1} \,,
    \end{displaymath}
    We employ~\cite[3.21]{Menne2013} with $\Hom(\R^{\vdim},\R^{\adim - \vdim})$,
    $2$, $\infty$, $\Phi$, $0$, $2 \eta$ in place of $H$, $k$, $l$, $\Phi$, $a$,
    $\delta$ to obtain the function $G : \Hom(\R^{\vdim},\R^{\adim - \vdim}) \to
    \R$; its conclusions read $\uD^i G = \uD^i \Phi$ on $\cball 0\eta$ for $i \in
    \{0,1,2\}$, $\| \uD^2 G(\sigma) - \uD^2 \Phi(0) \| \le \Gamma s \le
    \varepsilon$ for all~$\sigma$, and $G$ agrees outside $\cball 0{2\eta}$ with
    a~polynomial function of degree at most~$2$, which by the construction
    carried out in the proof of~\cite[3.21]{Menne2013} is the second order
    Taylor polynomial of~$\Phi$ at~$0$.  Recalling~\ref{rem:area-integrand},
    that polynomial is $\sigma \mapsto 1 + \tfrac 12 |\sigma|^2$, and $\uD^2
    \Phi(0)$ is the form $(\sigma,\tau) \mapsto \sigma \bullet \tau$, whose norm
    is~$1$; whence $G$ has all the desired properties.
\end{proof}

\begin{lemma}
    \label{lem:approx-function}
    Let $\adim, \vdim \in \natp$ satisfy $0 < \vdim \le \adim$ and $\Phi$ be as
    in~\ref{def:area-integrand}. There exist $0 < \kappa \le 1$, $0 <
    \varepsilon < \min \{ 1/2 ,\, \varepsilon_{\ref{lem:L1-T-1-est}}(\adim,\vdim) \}$,
    and $G : \Hom(\R^{\vdim},\R^{\adim - \vdim}) \to \R$ of class~$\cnt{\infty}$
    satisfying
    \begin{gather}
        \uD^i \Phi(\sigma) = \uD^i G(\sigma)
        \quad \text{for $i \in \{0, 1, 2 \}$ and $\sigma \in \Hom(\R^\vdim, \R^{\adim - \vdim})$ given $|\sigma| \le \kappa$} \,,
        \\
        (1 - \varepsilon)|\tau|^2 \le \langle \tau \odot \tau ,\, \uD^2G(\sigma) \rangle \le (1 + \varepsilon)|\tau|^2
        \quad \text{for $\sigma,\tau \in \Hom(\R^\vdim, \R^{\adim - \vdim})$} \,,
    \end{gather}
    such that whenever $0 < r < \infty$ and $g : \R^{\vdim} \to \R^{\adim -
      \vdim}$ satisfies $\Lip g \le \kappa$, then there exists a~map $u \in
    \espace{\oball 0r}{\R^{\adim - \vdim}}$ such that
    \begin{gather}
        \label{eq:LGu-and-u-g}
        L_G(u) = 0 \,,
        \quad
        u - g \in \VSobz{1,2}(\oball 0r,\R^{\adim - \vdim}) \,,
        \\
        \label{eq:LPhiu}
        L_{\Phi}(u)(\theta) = 0
        \quad
        \text{for $\theta \in \dspace{\oball 0r}{\R^{\adim-\vdim}}$ with $\spt \theta \subseteq \oball 0{r/4}$} \,,
        \\
        \label{eq:Du-g}
            \norm{\weakD (u-g)}{2}{0,r} \le 3 \norm{\weakD (g-w)}{2}{0,r}
            \quad
            \text{for $w \in \VSob{1,2}(\oball 0r,\R^{\adim - \vdim})$ with $L_G(w) = 0$} \,.
    \end{gather}
    Moreover, there is a positive finite number $\Gamma_0 =
    \Gamma_0(\adim,\vdim)$ with
    \begin{equation}
        \label{eq:u-lo-est}
        \norm{\uD u}{\infty}{0, r/4}
        + r^{1/2} \hoelder{1/2}{\uD u|{\cball{0}{r/4}}}
        \le \Gamma_0 r^{-\vdim/2} \norm{\uD g}{2}{0,r} \,,
    \end{equation}
    and, for any $k \in \natp$, a positive finite number $\Gamma =
    \Gamma(\adim,\vdim,k)$ such that
    \begin{equation}
        \label{eq:u-hi-est}
        \textsum{i=1}{k} \Bigl( r^{i-1} \norm{\uD^i u}{\infty}{0, 2^{-4}r}
        + r^{i-1/2} \hoelder{1/2}{\uD^i u|{\cball{0}{2^{-4}r}}} \Bigr)
        \le \Gamma r^{-\vdim/2} \norm{\uD g}{2}{0,r} \,.
    \end{equation}
    The radius $2^{-4}r$ in~\eqref{eq:u-hi-est} does not depend on~$k$.
\end{lemma}

\begin{proof}
    Employ~\cite[3.6]{Menne2013} and~\ref{lem:L1-T-1-est} to obtain
    \begin{displaymath}
        \varepsilon = \tfrac 12 \min \bigl\{ \varepsilon_{\text{\cite[3.6]{Menne2013}}}(\adim,2,2\vdim) ,\, \varepsilon_{\ref{lem:L1-T-1-est}}(\adim,\vdim) ,\, \tfrac 12 \bigr\} \,.
    \end{displaymath}
    Then use~\ref{lem:fake-integrand} to get
    \begin{displaymath}
        \eta = \eta_{\ref{lem:fake-integrand}}(\varepsilon)
        \quad \text{and} \quad
        G = G_{\ref{lem:fake-integrand}}(\varepsilon) \,.
    \end{displaymath}
    Let $\Delta_1 \in \R$ be the~positive finite number depending only
    on~$\vdim$ and~$\adim$ appearing in the Sobolev-Morrey inequality
    (e.g.~\cite[(7.42)]{Gilbarg2001}). Define
    \begin{gather*}
        \Delta_2 = 1 + 2^{\vdim+4} + 2^{2\vdim+5} \,,
        \quad
        \Delta_3 = \Delta_1 \Delta_2 \Gamma_{\text{\cite[3.6]{Menne2013}}}(\adim,2,2\vdim)
        2^{\vdim+3} \unitmeasure{\vdim}^{1/2} \,,
        \quad
        \Delta_4 = \Delta_3 + 2^{\vdim+2} \unitmeasure{\vdim}^{-1/2} \,,
        \\
        \eta_0 = \eta_{0,\ref{lem:schauder-est}}(\adim,\vdim,1-\varepsilon,G) \,,
        \quad \text{and} \quad
        \kappa = \frac{\min \{ \eta ,\, \eta_0 \}}{2 \Delta_4 \unitmeasure{\vdim}^{1/2}} \,;
    \end{gather*}
    note $2 \Delta_4 \unitmeasure{\vdim}^{1/2} \ge 2^{\vdim+3} \ge 1$, whence
    $\kappa \le \eta$.
    Assume $g : \R^{\vdim} \to \R^{\adim - \vdim}$ satisfies $\Lip g \le
    \kappa$. Apply the direct method of the calculus of variations
    (cf.~\cite[Theorems~4.5 and~4.6 and Remark~4.1]{Giusti2003}) to get the
    function $u \in \VSob{1,2}(\oball 0r,\R^{\adim - \vdim})$ such that
    \begin{displaymath}
        L_G(u) = 0 
        \quad \text{and} \quad
        u - g \in \VSobz{1,2}(\oball 0r,\R^{\adim - \vdim}) \,.
    \end{displaymath}
    Recall~\ref{lem:fake-integrand} to see that $\| \uD^2 G(\sigma)\| \le
    (1+\varepsilon)$ and $\langle \tau \odot \tau ,\, \uD^2G(\sigma) \rangle \ge
    (1 - \varepsilon)|\tau|^2$ whenever $\sigma,\tau \in \Hom(\R^\vdim, \R^{\adim -
      \vdim})$. We use~\ref{lem:W12est} with $g$, $u$, $G$ in place of $u$, $v$,
    $F$ respectively to get~\eqref{eq:Du-g}.
    From the choice of $\varepsilon$ and~\cite[3.6]{Menne2013} it follows that
    \begin{multline}
        \label{eq:D2u2d}
        \norm{\weakD^2 u}{2\vdim}{y,s/2} \le \Gamma_{\text{\cite[3.6]{Menne2013}}} s^{-2-\vdim + 1/2}\norm{u-P}{1}{y,s}
        \\ \text{for $y \in \oball 0r$, $0 < s \le r -|y|$ and any affine function $P : \R^{\vdim} \to \R^{\adim - \vdim}$} \,.
    \end{multline}
    In~particular, one can take $y = 0$ and $s = r$. From~\eqref{eq:D2u2d} and
    the Sobolev-Morrey embedding theorem~\cite[7.17]{Gilbarg2001} it follows
    that~$u|{\oball 0s}$ is of class~$\cnt{1,1/2}$ for any $0 < s < r$; hence,
    employing~\cite[5.2.15]{Federer1969} we see that $u \in \espace{\oball
      0r}{\R^{\adim - \vdim}}$. Choose $\varphi \in \dspace{\oball 0r}{\R}$ such
    that
    \begin{equation}
        \spt \varphi \subseteq \oball 0{r/2} \,,
        \quad
        \oball 0{r/4} \subseteq \varphi^{-1}\{1\} \,,
        \quad
        \norm{\uD \varphi}{\infty}{0,r} \le \frac{16}r \,,
        \quad
        \norm{\uD^2 \varphi}{\infty}{0,r} \le \frac{128}{r^2} \,.
    \end{equation}
    Define
    \begin{gather}
        a = {\textstyle \fint_{\oball 0{r/2}}} u \ud \LM^{\vdim} \,,
        \quad
        b = {\textstyle \fint_{\oball 0{r/2}}} \uD u \ud \LM^{\vdim} \,,
        \quad
        P(x) = a + b(x) \,,
        \quad
        v = \varphi \cdot (u - P) \,,
        \\
        \tilde a = {\textstyle \fint_{\oball 0{r}}} u \ud \LM^{\vdim} \,,
        \quad
        \tilde b = {\textstyle \fint_{\oball 0{r}}} \uD u \ud \LM^{\vdim} \,,
        \quad
        \tilde P(x) = \tilde a + \tilde b(x)
        \quad \text{for $x \in \R^{\vdim}$} \,.
    \end{gather}
    We~record that the Poincar{\'e} inequality~\cite[(7.45)]{Gilbarg2001},
    in~the form in which it is used below --- the mean being taken over the
    convex set~$\Omega$ itself --- carries the constant $\bigl(
    \unitmeasure{\vdim} / \LM^{\vdim}(\Omega) \bigr)^{1-1/\vdim} (\diam
    \Omega)^{\vdim}$, hence $2^{\vdim}\rho$ in case $\Omega = \oball 0\rho$.
    Since $v(x) = 0$ whenever $x \in \oball 0r$ and $|x| \ge r/2$ the
    Sobolev-Morrey inequality (e.g.~\cite[(7.42)]{Gilbarg2001}) yields
    \begin{equation}
        \label{eq:Dvx-Dvy}
        \| \uD v(x) - \uD v(y) \|
        \le \Delta_1 |x-y|^{1/2} \norm{\uD^2 v}{2\vdim}{0,r/2}
        \quad \text{for $x,y \in \oball 0{r/2}$}
        \,.
    \end{equation}
    Since $\fint_{\oball 0{r/2}} (u-P) \ud \LM^{\vdim} = 0$ and $\fint_{\oball
      0{r/2}} (\uD u - b) \ud \LM^{\vdim} = 0$, the Poincar{\'e} inequality
    (e.g.~\cite[(7.45)]{Gilbarg2001}), applied on $\oball 0{r/2}$, yields
    \begin{gather}
        \norm{u-P}{2\vdim}{0,r/2}
        \le 2^{\vdim-1} r \norm{\uD (u-P)}{2\vdim}{0,r/2} \,,
        \\
        \norm{\uD (u-P)}{2\vdim}{0,r/2}
        = \norm{\uD u - b}{2\vdim}{0,r/2}
        \le 2^{\vdim-1} r  \norm{\uD^2 u}{2\vdim}{0,r/2} \,.
    \end{gather}
    Thus,
    \begin{equation}
        \label{eq:D2vL2d}
        \norm{\uD^2 v}{2\vdim}{0,r/2}
        \le \tfrac{128}{r^2} \norm{u-P}{2\vdim}{0,r/2}
        + \tfrac{32}{r} \norm{\uD (u-P)}{2\vdim}{0,r/2}
        + \norm{\uD^2 u}{2\vdim}{0,r/2}
        \le \Delta_2 \norm{\uD^2 u}{2\vdim}{0,r/2} \,.
    \end{equation}
    Since $\fint_{\oball 0r}(u - \tilde P) \ud \LM^{\vdim} = 0$, the same
    inequality applied on $\oball 0r$, the bound $|\tilde b| \LM^{\vdim}(\oball
    0r) \le \norm{\uD u}{1}{0,r}$, H{\"o}lder's inequality,
    and~\eqref{eq:Du-g} applied with~$w = 0$ --- admissible because $\uD G(0) =
    \uD \Phi(0) = 0$, whence $L_G(0) = 0$ --- yield
    \begin{multline}
        \label{eq:u-PL1}
        r^{-2-\vdim + 1/2} \norm{u-\tilde P}{1}{0,r}
        \le 2^{\vdim} r^{-1-\vdim + 1/2} \norm{\uD u - \tilde b}{1}{0,r}
        \le 2^{\vdim+1} r^{-1/2 - \vdim} \norm{\uD u}{1}{0,r}
        \\
        \le 2^{\vdim+1} \unitmeasure{\vdim}^{1/2} r^{-1/2 - \vdim/2} \norm{\uD u}{2}{0,r}
        \le 2^{\vdim+1} \unitmeasure{\vdim}^{1/2} r^{-(\vdim+1)/2}
        \bigl( \norm{\uD (u-g)}{2}{0,r} + \norm{\uD g}{2}{0,r} \bigr)
        \\
        \le 2^{\vdim+3} \unitmeasure{\vdim}^{1/2} r^{-(\vdim+1)/2} \norm{\uD g}{2}{0,r}
        \,.
    \end{multline}
    Applying~\eqref{eq:D2u2d} with $y = 0$, $s = r$, and $\tilde P$ in place
    of~$P$ --- legitimate, that estimate holding for \emph{every} affine
    function --- and combining the outcome with~\eqref{eq:Dvx-Dvy}
    and~\eqref{eq:D2vL2d} we~obtain
    \begin{displaymath}
        \| \uD v(x) - \uD v(y) \|
        \le \Delta_3 \biggl( \frac{|x-y|}{r} \biggr)^{1/2} r^{-\vdim/2} \norm{\uD g}{2}{0,r}
        \quad \text{for $x,y \in \oball 0{r/2}$} \,.
    \end{displaymath}
    Note that $\uD v(x) = \uD u(x) - b$ whenever $x \in \cball 0{r/4}$, that
    $\uD v(z) = 0$ whenever $|z| = r/2$, and that, by H{\"o}lder's inequality
    and~\eqref{eq:Du-g} applied with $w = 0$,
    \begin{displaymath}
        |b| \le \LM^{\vdim}(\oball 0{r/2})^{-1/2} \norm{\uD u}{2}{0,r/2}
        \le 2^{\vdim/2+2} \unitmeasure{\vdim}^{-1/2} r^{-\vdim/2} \norm{\uD g}{2}{0,r} \,.
    \end{displaymath}
    Choosing $z$ with $|z| = r/2$ we~therefore get
    \begin{gather}
        \| \uD u(x) - \uD u(y) \|
        \le \Delta_3 \biggl( \frac{|x-y|}{r} \biggr)^{1/2} r^{-\vdim/2} \norm{\uD g}{2}{0,r}
        \quad \text{for $x,y \in \cball 0{r/4}$} \,,
        \\
        \|\uD u(x)\| \le \| \uD v(x) - \uD v(z) \| + |b|
        \le \Delta_4 r^{-\vdim/2} \norm{\uD g}{2}{0,r}
        \quad \text{for $x \in \cball 0{r/4}$} \,;
    \end{gather}
    hence,
    \begin{displaymath}
        r^{1/2} \hoelder{1/2}{\uD u|{\cball 0{r/4}}} + \norm{\uD u}{\infty}{0,r/4}
        \le 2 \Delta_4 r^{-\vdim/2} \norm{\uD g}{2}{0,r}
        \le 2 \Delta_4 \unitmeasure{\vdim}^{1/2} \Lip g \,.
    \end{displaymath} 
    This is~\eqref{eq:u-lo-est}, with $\Gamma_0 = 2\Delta_4$.  In~particular,
    since $2 \Delta_4 \unitmeasure{\vdim}^{1/2} \Lip g \le \min
    \{ \eta ,\, \eta_0 \} \le \eta$, we see that~\eqref{eq:LPhiu} holds;
    moreover
    \begin{displaymath}
        2 \Delta_4 \unitmeasure{\vdim}^{1/2} \Lip g \le \eta_0 \,,
    \end{displaymath}
    so the smallness hypothesis of~\ref{lem:schauder-est} is satisfied
    with $r/4$ in place of~$r$.  Given $k \in \natp$ we~obtain, applying
    \ref{lem:schauder-est} with $\max\{k,2\}$ in place of~$k$ --- whose
    conclusion is stated on $\cball 0{(r/4)/4} = \cball 0{2^{-4}r}$ ---
    and absorbing the factors $4^{i-1}$ and $4^{i-1/2}$ arising from the
    passage from the weights $(r/4)^{i-1}$, $(r/4)^{i-1/2}$ to $r^{i-1}$,
    $r^{i-1/2}$, the estimate~\eqref{eq:u-hi-est} with $\Gamma = 4^{k} \cdot 2
    \Delta_4 \Gamma_{\ref{lem:schauder-est}}(\adim,\vdim,1 - \varepsilon,\max\{k,2\},G)$.
\end{proof}

\section{Estimates for the bad set}

\begin{lemma}
    \label{lem:lip-approx}
    Assume
    \begin{gather}
        \adim,\vdim,Q \in \natp \,,
        \quad
        L \in (0,\infty) \,,
        \quad
        M \in [1,\infty) \,,
        \quad
        \delta \in (0,1/2) \,.
    \end{gather}
    Then there exists a~positive finite number $\varepsilon$ such that the
    following holds. Suppose
    \begin{enumerate}
    \item $\gamma \in (0,\varepsilon]$, $2 \le \vdim < \adim$, $r \in
        (0,\infty)$, $h \in (0,\infty]$, $h > 2 \delta r$,

    \item $T \in \grass{\adim}{\vdim}$,
        $p \in \orthproj{\adim}{\vdim}$,
        $\im p^* = T$,
        $q \in \orthproj{\adim}{\adim - \vdim}$,
        $\im q^* = T^{\perp}$,

    \item $U = \cylinder T0rh + \oball{0}{2r}$,
        $V \in \IVar{\vdim}(U)$,
        $\|\delta V\|$ is a~Radon measure,

    \item $(Q-1+\delta)\unitmeasure{\vdim} r^{\vdim} \le \|V\|(\cylinder T0rh)
        \le (Q+1-\delta)\unitmeasure{\vdim} r^{\vdim}$,
        
    \item $\|V\|( \cylinder T0r{h+\delta r} \without \cylinder T0r{h-2\delta r})
        \le (1 - \delta) \unitmeasure{\vdim} r^{\vdim}$,

    \item $\|V\|(U) \le M \unitmeasure{\vdim} r^{\vdim}$,
        
    \item $B$ is the set of all $z \in \cylinder T0rh$ satisfying
        $\density^{*\vdim}(\|V\|,z) > 0$ and
        \begin{gather}
            \text{either} \quad
            \measureball{\| \delta V \|}{\cball z{\rho}} > \gamma \measureball{\|V\|}{\cball z{\rho}}^{1 - 1/\vdim}
            \quad \text{for some $\rho \in (0,2r)$} 
            \\
            \text{or} \quad
            \textint{\cball z{\rho} \times \grass{\adim}{\vdim}}{} |\project S - \project T| \ud V(\xi,S)
            > \gamma \measureball{\|V\|}{\cball z{\rho}}
            \quad \text{for some $\rho \in (0,2r)$} \,,
        \end{gather}

    \item $A = \cylinder T0rh \without B$,
        $A(x) = A \cap \{ z : p(z) = x \}$ for $x \in \R^{\vdim}$,

    \item
        \label{i:la:X}
        $X$ is the set of all $x \in p \lIm A \rIm$ such that
        \begin{displaymath}
            \textsum{z \in A(x)}{} \density^{\vdim}(\|V\|,z) = Q
            \quad \text{and} \quad
            \density^{\vdim}(\|V\|,z) \in \nat \quad \text{for $z \in A(x)$} \,,
        \end{displaymath}

    \item
        \label{i:la:f}
        $f : X \to \qspace_Q(\R^{\adim - \vdim})$ is characterised by the requirement
        \begin{displaymath}
            \density^{\vdim}(\|V\|,z) = \density^0(\| f(x) \|, q(z))
            \quad \text{whenever $x \in X$ and $z \in A(x)$} \,,
        \end{displaymath}
        
    \item $g : \R^{\vdim} \to \R^{\adim-\vdim}$ is a Lipschitz extension
        of~$\boldsymbol{\eta}_Q \circ f : X \to \R^{\adim-\vdim}$ acquired by
        applying the Kirszbraun theorem~\cite[2.10.43]{Federer1969},
        
    \item $\measureball{\| \delta V \|}{\cball 0\rho}
        < 4^{-\vdim-1} \gamma \measureball{\|V\|}{\cball 0\rho}^{1-1/\vdim}$
        for $\rho \in (0,4r)$,
        
    \item $\int_{\cball 0\rho \times \grass{\adim}{\vdim}} |\project{S} -
        \project{T}| \ud V(\xi,S) < 4^{-\vdim-1} \gamma \measureball{\|V\|}{\cball
          0\rho}$ for $\rho \in (0,4r)$,

    \item $1 - \tfrac 12
        \le \frac{  \measureball{\|V\|}{\cball 0\rho} }{\unitmeasure{\vdim} \rho^\vdim \density^{\vdim}(\|V\|,0)}
        \le 1+ \tfrac 12$ 
        for $\rho \in (0,4r)$,

    \item $\Xi(s) = \measureball{\| \delta V \|}{\cylinder T0{3s}{3s}}^{\vdim/(\vdim-1)}
        + \textint{\cylinder T0{3s}{3s} \times \grass{\adim}{\vdim}}{} \| \project S - \project T \|^2 \ud V(\xi,S)$
        for $s \in (0,2r)$,

    \item
        \label{i:la:height}
        $\dist(z,T) \le \tfrac 12 |z|$ for $z \in \spt\|V\|$.
    \end{enumerate}

    Then $\Lip f \le L$ and the following hold.
    \begin{enumerate}[resume]
    \item
        \label{i:bs:good-scale}
        For $z \in \cball{0}{2r}$ and $\rho \in (2|z|,2r)$ we have
        \begin{gather}
            \measureball{\| \delta V \|}{\cball z\rho}
            < \gamma \measureball{\|V\|}{\cball z\rho}^{1-1/\vdim} 
            \\
            \text{and} \quad
            \textint{\cball z{\rho} \times \grass{\adim}{\vdim}}{} |\project S - \project T| \ud V(\xi,S)
            < \gamma \measureball{\|V\|}{\cball z{\rho}} \,.
        \end{gather}
        In particular, if $s \in (0,2r)$, then $B \cap \cball 0s$ is the set
        of~$z \in \cylinder T0rh \cap \cball 0s$ satisfying
        $\density^{*\vdim}(\|V\|,z) > 0$ and
        \begin{gather}
            \text{either} \quad
            \measureball{\| \delta V \|}{\cball z{\rho}} > \gamma \measureball{\|V\|}{\cball z{\rho}}^{1 - 1/\vdim}
            \quad \text{for some $\rho \in (0,2s]$} \,,
            \\
            \text{or} \quad
            \textint{\cball z{\rho} \times \grass{\adim}{\vdim}}{} |\project S - \project T| \ud V(\xi,S)
            > \gamma \measureball{\|V\|}{\cball z{\rho}}
            \quad \text{for some $\rho \in (0,2s]$} \,.
        \end{gather}

    \item
        \label{i:bs:BF-covering}
        For $s \in (0,2r)$ we have
        \begin{displaymath}
            \|V\|(B \cap \cball 0s) \le 4 \vdim \besicovitch{\adim} \gamma^{-2} \Xi(s) \,.
        \end{displaymath}
        
    \item  
        \label{i:bs:g-L2-norm}
        There exists $1 \le \Gamma < \infty$ determined by $Q$, $\vdim$,
        $\adim$, $L$, and $\gamma$ such that for $s \in (0,2r)$ there holds
        \begin{displaymath}
            \norm{\uD g}{2}{0,s}^2 \le \Gamma^2 \Xi(s) \,.
        \end{displaymath}
        
    \item
        \label{i:bs:geometric-tilt}
        Assume $s \in (0,r/2)$, $u : \R^{\vdim} \to \R^{\adim - \vdim}$, $\Lip u
        \le \gamma (8\vdim)^{-1/2}$, $w = p^* + q^* \circ u$, and $\Sigma =
        \im w$ is isometric to the~graph of~$u$. Then the following hold.
        \begin{enumerate}
        \item
            \label{i:bs:gt:angle}
            $\| \project{T} - \project{\Tan(\Sigma,w(p(x)))} \| < (8\vdim)^{-1/2} \gamma$
            and $| \project{T} - \project{\Tan(\Sigma,w(p(x)))} | < \tfrac 12 \gamma$
            for $x \in \R^{\vdim}$.
        \item
            \label{i:bs:gt:bad-set-local}
            $B \cap \cball 0{2s}$ is contained in the set of those $z \in
            \cylinder T0rh \cap \cball 0{2s}$ for which $\density^{*\vdim}(\|V\|,z)
            > 0$ and for some $\rho \in (0,4s]$ there holds
            \begin{gather}
                \text{either} \quad
                \measureball{\| \delta V \|}{\cball z{\rho}} > \gamma \measureball{\|V\|}{\cball z{\rho}}^{1 - 1/\vdim}
                \\
                \text{or} \quad
                \textint{\cball z{\rho} \times \grass{\adim}{\vdim}}{} |\project S - \project{\Tan(\Sigma,w(p(\xi)))}| \ud V(\xi,S)
                > \tfrac 12 \gamma \measureball{\|V\|}{\cball z{\rho}}
                \,.
            \end{gather}
        \item
            \label{i:bs:gt:bad-set-global}
            There exists $1 \le \Gamma < \infty$ determined by $Q$, $\vdim$,
            $\adim$, $L$, and $\gamma$ such that
            \begin{multline}
                \|V\|(B \cap \cylinder T0ss) \le
                \Gamma \bigl(
                \measureball{\| \delta V \|}{\cylinder T0{2^3s}{2^3s}}^{\vdim/(\vdim-1)}
                \\
                + \textint{\cylinder T0{2^3s}{2^3s} \times \grass{\adim}{\vdim}}{}
                \| \project S - \project{\Tan(\Sigma,w(p(\xi)))} \|^2
                \ud V(\xi,S)
                \bigr) \,.
            \end{multline}
        \end{enumerate}

    \end{enumerate}
\end{lemma}

\begin{proof}
    To find $\varepsilon$ we employ~\cite[5.7]{Menne2012a}. Clearly,
    by~\cite[5.7(4)]{Menne2012a}, we have $\Lip f \le L$.
    
    \textbf{Proof of~\ref{i:bs:good-scale}.} Assume $z \in \cball{0}{2r}$, $t =
    |z|$, and $\rho \in (2t,2r)$. Then
    \begin{displaymath}
        \cball{0}{\rho/2}
        \subseteq \cball{0}{\rho - t}
        \subseteq \cball{z}{\rho}
        \subseteq \cball{0}{\rho + t}
        \subseteq \cball{0}{2\rho} \,;
    \end{displaymath}
    the second and third inclusions because $|z| = t$, the first and fourth
    because $\rho > 2t$.  The ``in~particular'' clause follows: if $z \in \cball
    0s$ then every $\rho \in (2s,2r)$ lies in $(2|z|,2r)$, so that no such
    $\rho$ can witness membership of~$z$ in~$B$.
    hence,
    \begin{gather}
        \measureball{\| V \|}{\cball{0}{2\rho}}
        \le \tfrac{1+1/2}{1-1/2} 4^{\vdim}
            (1-\tfrac 12) \unitmeasure{\vdim} (\rho/2)^{\vdim} \density^{\vdim}(\|V\|,0)
        \le 4^{\vdim+1} \measureball{\| V \|}{\cball{0}{\rho/2}}
        \le 4^{\vdim+1} \measureball{\| V \|}{\cball{z}{\rho}} \,,
        \\
        \measureball{\| \delta V \|}{\cball{z}{\rho}}
        \le \measureball{\| \delta V \|}{\cball{0}{2\rho}}
        < 4^{-\vdim-1} \gamma \measureball{\| V \|}{\cball{0}{2\rho}}^{1 - 1/\vdim}
        \le \gamma \measureball{\| V \|}{\cball{z}{\rho}}^{1 - 1/\vdim} \,,
        \\
        \begin{aligned}
            \int_{\cball z\rho \times \grass{\adim}{\vdim}} |\project{S} - \project{T}| \ud V(\xi,S) 
            &\le \int_{\cball{0}{2\rho} \times \grass{\adim}{\vdim}} |\project{S} - \project{T}| \ud V(\xi,S)
            \\
            &< 4^{-\vdim-1} \gamma \measureball{\|V\|}{\cball{0}{2\rho}}
            \le \gamma \measureball{\|V\|}{\cball{z}{\rho}} \,.
        \end{aligned}
    \end{gather}

    \textbf{Proof of~\ref{i:bs:BF-covering}.} Let $s \in (0,2r)$ and set $\beta
    = \vdim/(\vdim-1)$. We proceed as in the proof of~\cite[7.7,
    pp.~35,36]{KolasnskiMenne}. Define sets $B_1$ and $B_2$ consisting of those
    $z \in B \cap \cball 0s$ satisfying
    \begin{gather*}
        \measureball{\| \delta V \|} { \cball z{\rho} }
        > \gamma \| V \| ( \cball z{\rho} )^{1/\beta}
        \text{ for some $\rho \in (0,2s)$} \,,
        \\
        \textint{\cball z{\rho} \times  \grass \adim \vdim}{} \| \project{S} - \project{T} \|^2 \ud V(z,S)
        > \tfrac 12 \gamma^2 \vdim^{-1} \measureball{\| V \|}{ \cball z{\rho} }
        \text{ for some $\rho \in (0,2s)$} 
    \end{gather*}
    respectively. To estimate $\| V \| ( B_1 )$ we employ the
    Besicovitch-Federer covering theorem (cf.~\cite[2.8.14]{Federer1969}) which
    provides pairwise disjoint families $F_1, \ldots, F_{\besicovitch{\adim}}$
    of closed balls such that
    \begin{gather*}
        {\textstyle B_1 \subset \bigcup \bigcup \{ F_i \with i = 1, \ldots, \besicovitch \adim \} }
          \subset \cball 0{3s} \,,
        \\
        \| V \| (C) < \gamma^{-\beta} \| \delta V \| ( C )^{\beta}
        \quad \text{for $i \in \{ 1, \ldots, \besicovitch{\adim} \}$ and $C \in F_i$} \,,
    \end{gather*}
    and we obtain
    \begin{multline}
        \| V \| ( B_1 )
        \le \gamma^{-\beta}
        \textsum{i=1}{\besicovitch \adim} \textsum{C \in F_i}{} \| \delta V \| ( C )^\beta 
        \\
         \le \gamma^{-\beta} \textsum{i=1}{ \besicovitch \adim } \big(
        \textsum{C \in F_i}{} \| \delta V \| ( C ) \big)^\beta 
        \le \gamma^{-\beta} \besicovitch{\adim} \measureball{\| \delta V \|}{\cball 0{3s}}^\beta \,.
    \end{multline}
    To estimate $\|V\|(B_2)$ we find another disjointed families $E_1, \ldots,
    E_{\besicovitch \adim}$ of closed balls such that
    \begin{gather*}
        {\textstyle B_2 \subset \bigcup \bigcup \{ E_i \with i = 1, \ldots, \besicovitch \adim \}}
        \subset \oball 0{3s} \,,
        \\
        \| V \| (C) < 2 \gamma^{-2} \vdim
        \textint{C \times \grass \adim \vdim}{} \| \project{S} - \project{T} \|^2 \ud V(z,S)
        \quad \text{for $i \in \{ 1, \ldots, \besicovitch{\adim} \}$ and $C \in E_i$} \,,
    \end{gather*}
    and consequently
    \begin{multline}
        \| V \| ( B_2 )
        \le 2 \gamma^{-2} \vdim \textsum{i = 1}{\besicovitch \adim}
        \textsum{C \in E_i}{} \textint{C \times \grass \adim \vdim}{}
        \| \project{S} - \project{T} \|^2 \ud V(z,S)
        \\
        \le 2 \gamma^{-2} \vdim \besicovitch \adim
        \textint{\cball 0{3s} \times \grass \adim \vdim}{}
        \| \project{S} - \project{T} \|^2 \ud V(z,S) \,.
    \end{multline}
    Since $\vdim \ge 2$ we have $1 < \beta \le 2$, whence $\gamma^{-\beta} \le
    \gamma^{-2}$ as $\gamma \le 1$, and $\gamma^{-\beta} + 2\vdim\gamma^{-2} \le
    4\vdim\gamma^{-2}$; the conclusion follows on
    noting $B \cap \cball 0s = B_1 \cup B_2$ due to~\ref{i:bs:good-scale}, that
    $\cball 0{3s} \subseteq \cylinder T0{3s}{3s}$, the
    H{\"o}lder inequality, and, recalling~\cite[8.9(1)(2)(3)]{Allard1972},
    \begin{equation}
        \label{eq:proj-norms}
        |\project{S} - \project{T}|^2
        = 2 \perpproject{S} \bullet \project{T}
        = 2|\perpproject{S} \circ \project{T}|^2
        \le 2 \vdim \| \perpproject{S} \circ \project{T} \|^2
        = 2 \vdim \| \project{S} - \project{T} \|^2
        \quad \text{for $S,T \in \grass{\adim}{\vdim}$} \,.
    \end{equation}

    \textbf{Proof of~\ref{i:bs:g-L2-norm}.} Employ~\cite[2.5]{Menne2010}
    to find a countable index set~$I$ and for each $i \in I$ an
    $\LM^{\vdim}$~measurable set $A_i \subseteq X$ and a~function $f_i : A_i \to
    \R^{\adim - \vdim}$ satisfying $\Lip f_i \le \Lip f$ and
    \begin{displaymath}
        \HM^0(I \cap \{ i  : f_i(x) = y \}) = \density^0(\| f(x) \|, y)
        \quad \text{for $(x,y) \in X \times \R^{\adim-\vdim}$} \,.
    \end{displaymath}
    Let $s \in (0,2r)$. If $x \in X$, then
    \begin{displaymath}
        \uD g(x) = \tfrac 1Q \textsum{i \in I(x)}{} \ap \uD f_i(x)
        \quad \text{where } I(x) = I \cap \{ i : x \in \dmn \ap \uD f_i\} \,.
    \end{displaymath}
    Employing~\cite[2.8]{Menne2010}, \cite[2.2]{Menne2012a},
    \cite[3.15(7d)]{Menne2010}, and~\cite[3.5(1)(b)]{Allard1972} we conclude
    that for $x \in X$ there holds
    \begin{displaymath}
        \| \uD g(x) \|^2
        \le \| \ap A f(x) \|^2
        \le Q (1 + (\Lip f)^2) \max \bigl\{ \| \project T - \project{\Tan^{\vdim}(\|V\|,\xi)} \|^2 : \xi \in p^{-1}\{x\} \bigr\} \,.
    \end{displaymath}
    Next, proceeding as in the last paragraph of the proof
    of~\cite[7.7]{KolasnskiMenne} --- whose conclusion~(4) is the present
    assertion with $\|V\| \restrict H$ in place of $\|V\|$ --- we get
    \begin{displaymath}
        \norm{\uD g}{2}{X \cap \cball 0s}^2
        \le \textint{A \cap p^{-1}\lIm X \cap \cball 0s \rIm \times \grass \adim \vdim}{} \| \project{S} - \project{T} \|^2 \ud V(z,S) \,.
    \end{displaymath}
    The region of integration lies in $\cylinder T0{3s}{3s}$, so that the right
    hand side is at most~$\Xi(s)$.  Indeed, $V$ is carried by $\spt\|V\| \times
    \grass{\adim}{\vdim}$, and if $z \in \spt\|V\|$ satisfies $|p(z)| \le s$
    then, writing $\sigma = |q(z)| = \dist(z,T)$, hypothesis~\ref{i:la:height}
    gives $\sigma \le \tfrac12 |z| = \tfrac12 (|p(z)|^2 + \sigma^2)^{1/2}$;
    whence, $\sigma^2 \le \tfrac14 (s^2+\sigma^2)$ and $\sigma \le 3^{-1/2}s <
    s$.  Consequently $\spt\|V\| \cap p^{-1}\lIm \cball 0s \rIm \subseteq
    \cylinder T0ss \subseteq \cylinder T0{3s}{3s}$.  We~note that the height
    bound $q\lIm A \cap \spt\|V\| \rIm \subseteq \cball 0{h-\delta r}$
    of~\cite[5.7(2)]{Menne2012a} would not serve here, being of the order
    of~$h$ rather than of~$s$.
    Noting $\Lip g \le \Lip f \le L$ and setting $\Delta =
    4 \vdim \Gamma_{\text{\cite[5.7(7)]{Menne2012a}}}(Q,\vdim) L^2 \besicovitch{\adim}
    \gamma^{-2}$ we obtain, by~\cite[5.7(7)]{Menne2012a}
    and~\ref{i:bs:BF-covering}, the estimate
    \begin{displaymath}
        \norm{\uD g}{2}{\cball 0s \without X}^2 
        \le \Delta  \Xi(s) \,.
    \end{displaymath}
    Combining the last two estimates we reach the conclusion.

    \textbf{Proof of~\ref{i:bs:geometric-tilt}.}
    Item~\ref{i:bs:gt:angle} is obtained by employing~\cite[4.2, 4.3]{KolRMI} to
    see that for $x \in \R^{\vdim}$
    \begin{displaymath}
        \| \project{T} - \project{\Tan(\Sigma, w(p(x)))} \|^2
        = \frac{\|\uD u(x)\|^2}{1 + \|\uD u(x)\|^2}
        \le \frac{\gamma^2}{8\vdim + \gamma^2}
        < \frac{\gamma^2}{8 \vdim} \,;
    \end{displaymath}
    the second assertion then follows from~\eqref{eq:proj-norms}, which gives
    $| \project{T} - \project{\Tan(\Sigma,w(p(x)))} | \le (2\vdim)^{1/2} \|
    \project{T} - \project{\Tan(\Sigma,w(p(x)))} \| < \tfrac 12 \gamma$.
    Now clause~\ref{i:bs:gt:bad-set-local} follows from~\ref{i:bs:gt:angle} and
    from~\ref{i:bs:good-scale}, applied with $2s$ in place of~$s$ --- admissible
    because $2s < r < 2r$ --- by~applying the triangle inequality, since
    \begin{multline}
        \textint{\cball z{\rho} \times \grass{\adim}{\vdim}}{}
        |\project S - \project{\Tan(\Sigma,w(p(\xi)))}|
        \ud V(\xi,S)
        \\
        \ge \textint{\cball z{\rho} \times \grass{\adim}{\vdim}}{}
        |\project S - \project{T}| \ud V(\xi,S)
        - \frac{\gamma}{2} \measureball{\|V\|}{\cball z{\rho}}
    \end{multline}
    whenever $z \in \R^{\adim}$ and $0 < \rho < \infty$.  Finally,
    \ref{i:bs:gt:bad-set-global} is obtained by repeating the covering argument
    of~\ref{i:bs:BF-covering} with $2s$, $\Tan(\Sigma,w(p(\xi)))$, and
    $\tfrac12\gamma$ in place of $s$, $T$, and~$\gamma$: the sets $B_1$
    and~$B_2$ are then formed with $\rho \in (0,4s]$ and $z \in B \cap \cball
    0{2s}$, their coverings lie in $\cball 0{6s} \subseteq \cylinder
    T0{2^3s}{2^3s}$, and $\cylinder T0ss \subseteq \cball 0{2s}$.
    %
\end{proof}

\section{The main theorem}

\begin{lemma}
    \label{lem:Dk-tilt}
    Assume
    \begin{gather}
        Q \in \natp \,,
        \quad
        A \subseteq \R^{\vdim} \text{ is $\LM^{\vdim}$-measurable} \,,
        \quad
        f : A \to \qspace_Q(\R^{\adim-\vdim}) \,,
        \quad
        \Lip f \le \tfrac 12 \,,
        \\
        v : A \to \R^{\adim-\vdim} \,,
        \quad
        \Lip v \le \tfrac 12 \,,
        \quad
        w : A \to \R^{\adim} \,,
        \\
        p \in \orthproj{\adim}{\vdim} \,,
        \quad
        q \in \orthproj{\adim}{\adim-\vdim} \,,
        \quad
        p \circ q^* = 0 \,,
        \\
        w = p^* + q^* \circ v \,,
        \quad
        M = \im w \,,
        \quad
        N = \bigl\{ p^*(x) + q^*(y) : x \in A ,\, y \in \spt f(x) \bigr\} \,,
        \\
        k(x) = (\trans{-v(x)})_{\#}f(x) \quad \text{for $x \in A$}
        \\
        \psi : N \to \R \,,
        \quad
        \psi(z) = (2Q)^{1/2} \| \project{\Tan( M, w \circ p(z) )} - \project{\Tan(N,z)} \|
        \quad \text{for $z \in N$} \,.
    \end{gather}
    Then
    \begin{displaymath}
        \LM^{\vdim}\bigl(
        A \cap \{ x : \| \ap \Aff k(x) \| > s \}
        \without p \lIm N \cap p^{-1} \lIm A \rIm \cap \{ z : \psi(z) > s \} \rIm
        \bigr) = 0
        \quad \text{for $s > 0$} \,.
    \end{displaymath}
\end{lemma}

\begin{proof}
    Employ~\cite[2.3]{Menne2012a} to find a~countable family of Lipschitz
    functions $\{ f_i : i \in I \}$ with the properties described therein. For
    $a \in \R^{\vdim}$ let $I(a) = I \cap \{ i : a \in \dmn \ap \uD f_i \}$ and
    note that $\HM^0(I(a)) = Q$. For $i \in I$ define $k_i = f_i - v$. Let $a
    \in \R^{\vdim}$ be such that $f$ is approximately strongly affinely
    approximable at~$a$, and $v$ is differentiable at~$a$. Observe that
    \begin{displaymath}
        \ap \Aff k(a) u
        = \sum_{i \in I(a)} \boldsymbol{[} k_i(a) + \langle u ,\, \ap \uD k_i(a) \rangle \boldsymbol{]}
        \quad \text{for $u \in \R^{\vdim}$} \,.
    \end{displaymath}
    Moreover, using the definition~\cite[1.1(9)]{Almgren2000e} of $\| \ap \Aff
    k(a) \|$ together with~\cite[8.9(5)]{Allard1972} we~get
    \begin{multline}
        \| \ap \Aff k(a) \|
        \le \ap \limsup_{u \to 0} \frac{\mathscr{G}(\ap \Aff k(a) u, \ap \Aff k(a)0)}{|u|}
        \le \biggl( \sum_{i \in I(a)} \| \ap \uD k_i(a) \|^2 \biggr)^{1/2}
        \\
        \le \biggl( \frac{1 + (\Lip f)^2}{1 - (\Lip v)^2}  \sum_{i \in I(a)} \| \project{ \Tan(N, (p^* + q^* \circ f_i)(a))} - \project{\Tan(M,w(a))} \|^2 \biggr)^{1/2} \,.
    \end{multline}
    Since $\Lip f \le \frac 12$ and $\Lip v \le \frac 12$ we have $\frac{1 +
      (\Lip f)^2}{1 - (\Lip v)^2} \le \frac 53 < 2$ and the conclusion follows
    by recalling that $\HM^0(I(a)) = Q$ and that $f$ is approximately strongly
    affinely approximable at $\LM^{\vdim}$ almost all $a \in \R^{\vdim}$;
    cf.~\cite[2.3]{Menne2012a}.
\end{proof}

\begin{corollary}
    \label{cor:Dg-Dh-tilt}
    Let $g : \R^{\vdim} \to \R^{\adim-\vdim}$ and $h : \R^{\vdim} \to \R$ be
    defined by
    \begin{gather}
        g = \boldsymbol{\eta}_Q \circ f - v 
        \quad \text{and} \quad
        h(x) = \mathscr{G}(f(x), \boldsymbol{[} v(x) \boldsymbol{]})
        \quad
        \text{for $x \in \R^{\vdim}$} \,.
    \end{gather}
    Then
    \begin{displaymath}
        \LM^{\vdim}\bigl(
        A \cap \{ x : \max \{ \| \uD h(x) \| ,\, \| \uD g(x) \| \} > s \}
        \without p \lIm N \cap p^{-1} \lIm A \rIm \cap \{ z : \psi(z) > s \} \rIm
        \bigr) = 0
        \quad \text{for $s > 0$} \,.
    \end{displaymath}
    In~particular, if $V = \var{\vdim}(N)$ and $1 \le r < \infty$, then
    \begin{displaymath}
        \int_{A}\max \{ \| \uD g \| ,\, \| \uD h \| \}^r \ud \LM^{\vdim}
        \le (2Q)^{r/2} \int_{p^{-1} \lIm A \rIm} \| \project{\Tan( M, w \circ p(z) )} - \project{S} \|^r \ud V(z,S) \,.
    \end{displaymath}
\end{corollary}

\begin{proof}
    Observe that
    \begin{displaymath}
        \Lip \boldsymbol{\eta}_Q \le Q^{-1/2} 
        \quad \text{and} \quad
        g = \boldsymbol{\eta}_Q \circ k \,.
    \end{displaymath}
    Moreover, we readily verify 
    that
    \begin{displaymath}
        \| \uD g(a) \| \le Q^{-1/2} \| \ap \Aff k(a) \| 
        \quad \text{and} \quad
        \| \uD h(a) \| \le \| \ap \Aff k(a) \| \,.
    \end{displaymath}
    The conclusion follows now from~\ref{lem:Dk-tilt} combined
    with~\cite[3.5(1)]{Allard1972} and~\cite[2.10.35]{Federer1969} exactly as in
    the proof of~\cite[7.9]{KolasnskiMenne}.
\end{proof}

\begin{lemma}
    \label{lem:generic-props}
    Suppose
    \begin{gather}
        U \subseteq \R^{\adim} \text{ is open} \,,
        \quad
        V \in \IVar{\vdim}(U) \,,
        \quad
        \delta V = 0 \,,
        \quad
        \varepsilon, \delta \in \R \cap \{ t : 0 < 2t < 1 \} \,,
    \end{gather}
    and let $A$ consist of all $a \in U$ such that
    \begin{gather}
        \label{eq:gp:dens-tan}
        \density^{\vdim}(\|V\|,a) \in \natp \,,
        \quad
        \Tan^{\vdim}(\|V\|,a) \in \grass{\adim}{\vdim} \,,
        \\
        \label{eq:gp:dens-cnt}
        \lim_{s \downarrow 0}
        \frac{\|V\| \bigl( \cball as \cap \{ z : | \density^{\vdim}(\|V\|,z) - \density^{\vdim}(\|V\|,a)| > \frac 12 \}\bigr)}
        {\measureball{\|V\|}{\cball as}} = 0 \,,
        \\
        \label{eq:gp:tan-cnt}
        \text{and} \quad
        \lim_{s \downarrow 0} \bigl( s^{-\vdim} \textint{\cylinder Tasr \times \grass{\adim}{\vdim}}{}
        \| \project{S} - \project{\Tan^{\vdim}(\|V\|,a)} \|^2 \ud V(x,S) \bigr)^{1/2} = 0 \,.
    \end{gather}
    Then
    \begin{displaymath}
        \|V\|(U \without A) = 0
    \end{displaymath}
    and for $a \in A$ there exists $0 < r <\infty$ such that setting $Q =
    \density^{\vdim}(\|V\|,a)$ and $T = \Tan^{\vdim}(\|V\|,a)$ there holds

    \begin{enumerate}
    \item $\cylinder{T}{a}{r}{r} + \oball{0}{2r} \subseteq \oball{a}{4r} \subseteq U$,

    \item $\density^{\vdim}(\|V\|,a) = Q$,
        
    \item $Q \unitmeasure{\vdim} r^{\vdim}
        \le \measureball{\|V\|}{\cylinder{T}{a}{r}{r}}
        \le (Q + 1 - \delta) \unitmeasure{\vdim} r^{\vdim}$,
        
    \item $\|V\|\bigl( \cylinder{T}{a}{r}{(1+\delta)r} \without
        \cylinder{T}{a}{r}{(1-2\delta)r} \bigr) = 0$,
        
    \item
        \label{i:gp:dens-ratio}
        $Q \unitmeasure{\vdim} s^{\vdim} \le \measureball{\|V\|}{\cylinder Tass}
        \le 2Q \unitmeasure{\vdim} (4s)^{\vdim}$ for $0 < s < r$,
        
    \item
        \label{i:gp:density=Q}
        $\|V\|\bigl( \cylinder Tasr \without \{ z : \density^{\vdim}(\|V\|,z) = Q \} \bigr)
        < \varepsilon s^{\vdim}$ for $0 < s < r$,
        
    \item
        \label{i:gp:tilt-uniform}
        $\textint{\oball{z}{2r} \times \grass{\adim}{\vdim}}{}
        | \project{S} - \project{T} | \ud V(x,S) < \varepsilon
        \measureball{\|V\|}{\oball{z}{2r}}$ for $\|V\|$ almost all $z \in
        \cylinder Tarr$,
        
    \item
        \label{i:gp:L2-tilt-small}
        $\bigl(s^{-\vdim} \textint{\cylinder Tasr \times \grass{\adim}{\vdim}}{}
        \| \project{S} - \project{T} \|^2 \ud V(x,S) \bigr)^{1/2} < \varepsilon$
        \ for $0 < s < r$,
        
    \item $\sup \bigl\{ \dist(z,a+T) : z \in \spt \|V\| \cap \cball as \bigr\} <
        \varepsilon s$ \ for $0 < s < r$.
    \end{enumerate}
\end{lemma}

\begin{proof}
    Let $C = \bigl\{ (x,\cball xr) : x \in \R^{\adim} ,\, r > 0 \bigr\}$. Then
    $C$ is a $\|V\|$~Vitali relation, cf.~\cite[2.8.16, 2.8.18]{Federer1969}.
    Therefore, $A$ is the set of all $a \in U$ such that~\eqref{eq:gp:dens-tan}
    holds, $\density^{\vdim}(\|V\|,\cdot)$ is $(\|V\|,C)$~approximately
    continuous at~$a$, and $a$ is an~$\Lp{2}(\|V\|)$~Lebesgue point of the map
    $\project{\Tan^{\vdim}(\|V\|,\cdot)}$.

    Since $V \in \IVar{\vdim}(U)$ and $\density^{\vdim}(\|V\|,\cdot)$ is
    a~$\|V\|$~measurable function (actually, it is even upper semicontinuous;
    cf.~\cite[5.1(4), 8.6]{Allard1972}) it follows that it is
    $(\|V\|,C)$~approximately continuous (cf.~\cite[2.9.13]{Federer1969}) so
    indeed $\|V\|(U \without A) = 0$.

    Let $a \in A$, $Q = \density^{\vdim}(\|V\|,a)$, and $T =
    \Tan^{\vdim}(\|V\|,a)$. Since $\density^{\vdim}(\|V\|,z) \in \natp$ for
    $\|V\|$ almost all $z \in U$ (cf.~\cite[3.5(1)(a)(c)]{Allard1972}) we derive
    \begin{displaymath}
        \lim_{s \downarrow 0} \frac{\|V\|(\cball as \without \{ z : \density^{\vdim}(\|V\|,z) = Q\})}{\unitmeasure{\vdim} s^{\vdim}} = 0 \,.
    \end{displaymath}
    Recall the monotonicity formula~\cite[5.1(2)]{Allard1972} (see
    also~\cite[Corollary~4.5]{Menne2016a} for a streamlined proof): the function
    \begin{displaymath}
        t \mapsto \unitmeasure{\vdim}^{-1} t^{-\vdim} \measureball{\|V\|}{\cball at}
    \end{displaymath}
    is nondecreasing on $\{ t : \cball at \subseteq U \}$ with limit
    $\density^{\vdim}(\|V\|,a) = Q$ as $t \downarrow 0$; in particular
    \begin{equation}
        \label{eq:gp:mono}
        \measureball{\|V\|}{\cball at} \ge Q \unitmeasure{\vdim} t^{\vdim}
        \quad \text{whenever $\cball at \subseteq U$} \,,
    \end{equation}
    and, given $0 < \sigma < \infty$, there is $t_0 > 0$ with
    $\measureball{\|V\|}{\cball at} \le (Q+\sigma)\unitmeasure{\vdim}t^{\vdim}$
    for $0 < t \le t_0$. Moreover, $\density^{\vdim}(\|V\|,z) \ge 1$ for $z \in
    \spt\|V\|$ by~\eqref{eq:gp:mono} applied at~$z$, since
    $\density^{\vdim}(\|V\|,z) \ge 1$ for $\|V\|$~almost all~$z$
    and~$\density^{\vdim}(\|V\|,\cdot)$ is upper
    semicontinuous;~cf.~\cite[3.5(1)(a)(c), 5.1(4), 8.6]{Allard1972}. Finally, since
    $\density^{\vdim}(\|V\|,a) = Q$ and $\Tan^{\vdim}(\|V\|,a) = T$, the
    rescaled measures $\|V\|_s$, defined by $\|V\|_s(E) = s^{-\vdim}
    \|V\|(a + \scale s \lIm E \rIm)$, converge weakly to $Q \HM^{\vdim}
    \restrict T$ as $s \downarrow 0$; cf.~\cite[3.2.19]{Federer1969}
    and~\cite[3.5]{Allard1972}.

    We~claim that
    \begin{equation}
        \label{eq:gp:height}
        \lim_{s \downarrow 0} s^{-1} \sup \bigl\{ \dist(z,a+T) : z \in \spt\|V\| \cap \cball as \bigr\} = 0 \,.
    \end{equation}
    Otherwise there are $\vartheta > 0$ and $s_i \downarrow 0$ together with
    $z_i \in \spt\|V\| \cap \cball a{s_i}$ satisfying $\dist(z_i,a+T) \ge
    \vartheta s_i$. Then $\cball{z_i}{\vartheta s_i/2}$ does not meet $a+T$,
    whereas $\measureball{\|V\|}{\cball{z_i}{\vartheta s_i/2}} \ge
    \unitmeasure{\vdim} (\vartheta s_i/2)^{\vdim}$ by~\eqref{eq:gp:mono}
    and~$\density^{\vdim}(\|V\|,z_i) \ge 1$. Passing to a~subsequence so that
    $(z_i-a)/s_i$ converges to some~$\zeta$ with $|\zeta| \le 1$ and
    $\dist(\zeta,T) \ge \vartheta$, the weak convergence $\|V\|_{s_i} \to Q
    \HM^{\vdim} \restrict T$ yields
    \begin{displaymath}
        Q \HM^{\vdim} \bigl( T \cap \cball{\zeta}{\vartheta/2} \bigr)
        \ge \limsup_{i \to \infty} \|V\|_{s_i} \bigl( \cball{(z_i-a)/s_i}{\vartheta s_i / (2s_i)} \bigr)
        \ge \unitmeasure{\vdim}(\vartheta/2)^{\vdim} > 0 \,,
    \end{displaymath}
    contradicting $T \cap \cball{\zeta}{\vartheta/2} = \varnothing$.

    We~now choose~$r$. Using~\eqref{eq:gp:dens-cnt}
    and~\eqref{eq:gp:tan-cnt}, the convergence $\|V\|_s \to Q \HM^{\vdim}
    \restrict T$, \eqref{eq:gp:height}, the monotonicity formula, and the
    equality
    \begin{displaymath}
        \lim_{s \downarrow 0} \frac{\|V\|(\cball as \without \{ z : \density^{\vdim}(\|V\|,z) = Q\})}{\unitmeasure{\vdim} s^{\vdim}} = 0
    \end{displaymath}
    established above, we~select $0 < r < \infty$ so small that $\cball a{4r}
    \subseteq U$ and, for $0 < s \le 4r$,
    \begin{gather}
        \label{eq:gp:choice}
        \measureball{\|V\|}{\cball as} \le (Q + \tfrac 12) \unitmeasure{\vdim} s^{\vdim} \,,
        \quad
        \|V\|\bigl(\cylinder Tass\bigr) \le (Q + 1 - \delta) \unitmeasure{\vdim} s^{\vdim} \,,
        \\
        \sup \bigl\{ \dist(z,a+T) : z \in \spt\|V\| \cap \cball as \bigr\}
        \le \tfrac 14 \min \{ \varepsilon ,\, 1 - 2\delta \} s \,,
        \\
        \|V\|\bigl( \cball a{2s} \without \{ z : \density^{\vdim}(\|V\|,z) = Q\} \bigr) \le \varepsilon s^{\vdim} \,,
        \quad
        \textint{\cball a{4r} \times \grass{\adim}{\vdim}}{} \| \project{S} - \project{T} \|^2 \ud V(x,S) \le \tfrac{\varepsilon^2}{2\vdim} \unitmeasure{\vdim} (2r)^{\vdim} \,,
    \end{gather}
    and so that the left hand side of~\eqref{eq:gp:tan-cnt} is less
    than~$\varepsilon$ for $0 < s < r$.

    We verify the conclusions in turn. Item~(1) holds because every $z \in
    \cylinder Tarr$ satisfies $|z-a| \le \sqrt 2 r$; whence, $\cylinder Tarr +
    \oball 0{2r} \subseteq \oball a{(2+\sqrt 2)r} \subseteq \oball a{4r}
    \subseteq U$; item~(2) is the definition of~$Q$. For items~(3) and~(5)
    we~note $\cball as \subseteq \cylinder Tass$, so
    that~\eqref{eq:gp:mono} gives the lower bounds $Q \unitmeasure{\vdim}
    s^{\vdim} \le \|V\|(\cylinder Tass)$, while the upper bound in~(3) is part
    of~\eqref{eq:gp:choice} and that in~(5) follows from $\cylinder Tass
    \subseteq \cball a{4s}$ together with $\measureball{\|V\|}{\cball a{4s}} \le
    (Q+\tfrac 12)\unitmeasure{\vdim}(4s)^{\vdim} \le 2Q
    \unitmeasure{\vdim}(4s)^{\vdim}$, as $Q \ge 1$.

    Item~(4): if $z \in \spt\|V\| \cap \cylinder Tar{(1+\delta)r}$ then $|z-a|
    \le \sqrt{1 + (1+\delta)^2}\, r \le 2r$, so $|q(z-a)| = \dist(z,a+T) \le
    \tfrac 14 (1-2\delta) \cdot 2r < (1-2\delta) r$ and therefore $z \in
    \cylinder Tar{(1-2\delta)r}$; hence, the set in question misses $\spt\|V\|$
    and is $\|V\|$~null.

    Item~(6): by the third line of~\eqref{eq:gp:choice}, applied with $s$ in
    place of~$s$, every $z \in \spt\|V\| \cap \cylinder Tasr$ satisfies
    $\dist(z,a+T) \le \tfrac 14 \varepsilon \cdot 2s \le s$; whence, $\cylinder
    Tasr \cap \spt\|V\| \subseteq \cball a{2s}$ and the assertion follows from
    the fourth line of~\eqref{eq:gp:choice}.

    Item~(7): let $z \in \cylinder Tarr$ with $\density^{\vdim}(\|V\|,z) \ge 1$;
    since $\oball z{2r} \subseteq \oball a{4r}$ and, by~\eqref{eq:gp:mono},
    $\measureball{\|V\|}{\oball z{2r}} \ge \unitmeasure{\vdim}(2r)^{\vdim}$,
    H{\"o}lder's inequality and~\eqref{eq:proj-norms} give
    \begin{multline*}
        \textint{\oball z{2r} \times \grass{\adim}{\vdim}}{} | \project{S} - \project{T} | \ud V(x,S)
        \le \bigl( 2\vdim \textint{\oball a{4r} \times \grass{\adim}{\vdim}}{} \| \project{S} - \project{T} \|^2 \ud V(x,S) \bigr)^{1/2}
        \measureball{\|V\|}{\oball z{2r}}^{1/2}
        \\
        \le \varepsilon \unitmeasure{\vdim}^{1/2} (2r)^{\vdim/2} \measureball{\|V\|}{\oball z{2r}}^{1/2}
        \le \varepsilon \measureball{\|V\|}{\oball z{2r}} \,,
    \end{multline*}
    which proves~(7), as $\density^{\vdim}(\|V\|,z) \ge 1$ for $\|V\|$~almost
    all~$z$. Item~(8) is the choice of~$r$ made in the last sentence preceding
    the verification, and item~(9) follows from the third line
    of~\eqref{eq:gp:choice}, since $\tfrac 14 \varepsilon < \varepsilon$.
\end{proof}

%

\begin{definition}
    \label{def:height-and-tilt}
    Suppose
    \begin{gather}
        T \in \grass{\adim}{\vdim} \,,
        \quad
        p \in \orthproj{\adim}{\vdim} \,,
        \quad
        q \in \orthproj{\adim}{\adim - \vdim} \,,
        \quad
        \im p^* = T = \ker q \,,
        \\
        A \subseteq \R^{\vdim} \text{ is $\LM^{\vdim}$~measurable}\,,
        \quad
        v : A \to \R^{\adim-\vdim} \,,
        \quad
        w = p^* + q^* \circ v \,,
        \quad
        M = \im w \,,
        \\
        \Lip v < \infty \,,
        \quad
        a \in \R^{\adim} \,,
        \quad
        0 < s < \infty \,,
        \quad
        U \subseteq \R^{\adim} \text{ is open} \,,
        \quad
        V \in \Var{\vdim}(U) \,.
    \end{gather}
    We define
    \begin{gather}
        C(v,p,a,s) = U \cap \cylinder Tass \cap p^{-1}\lIm A \rIm \,,
        \quad
        D(M,T,a,s) = U \cap \cylinder Tass \cap \Unp(M) \,,
        \\
        \phi(v,p,a,s) = \Bigl( s^{-\vdim+2} \textint{C(v,p,a,s)}{} |q(z) - v(p(z))|^2 \ud \|V\|(z) \Bigr)^{1/2} \,,
        \\
        \psi(v,p,a,s) = \Bigl( s^{-\vdim} \textint{C(v,p,a,s) \times \grass{\adim}{\vdim}}{}
        \| \project{\Tan( M, w \circ p(z) )} - \project{S} \|^2 \ud V(z,S) \Bigr)^{1/2} \,,
        \\
        \tau(T,a,s) = \Bigl( s^{-\vdim} \textint{U \cap \cylinder Tass \times \grass{\adim}{\vdim}}{}
        \| \project S - \project T \|^2 \ud V(z,S) \Bigr)^{1/2} \,,
        \\
        \beta(M,T,a,s) = \Bigl( s^{-\vdim-2} \textint{D(M,T,a,s)}{} \dist(z,M)^2 \ud \|V\|(z) \Bigr)^{1/2} \,,
        \\
        \eta(M,T,a,s) = \Bigl( s^{-\vdim} \textint{D(M,T,a,s) \times \grass{\adim}{\vdim}}{}
        \| \project{T} - \project{\Tan(M,\npp{M}(x))} \|^2 \ud V(x,T) \Bigr)^{1/2} \,.
    \end{gather}
    Consistently with the notation for~$\phi$ and~$\psi$, for any $\tilde v : A
    \to \R^{\adim-\vdim}$ with $\Lip \tilde v < \infty$ we also write
    $\beta(\tilde v,p,a,s) = \beta(\im(p^* + q^* \circ \tilde v),T,a,s)$.
\end{definition}

\begin{theorem}
    \label{thm:decay}
    Suppose
    \begin{gather}
        2 \le \vdim < \adim \,,
        \quad
        U \subseteq \R^{\adim} \text{ is open} \,,
        \quad
        V \in \IVar{\vdim}(U) \,,
        \quad
        \delta V = 0 \,,
        \quad
        0 < \kappa < 1 \,.
    \end{gather}
    Then for $\|V\|$ almost all $a \in U$ there exist a number $0 < r \le 1$ and
    a~family $\{ M_s : 0 < s < r \}$ of submanifolds of~$U$ of
    class~$\cnt{\infty}$ such that, recalling~\ref{def:height-and-tilt} and
    setting
    \begin{displaymath}
        \zeta(s) = \sup \bigl\{ \beta(M_{\rho},\Tan^{\vdim}(\|V\|,a),a,\rho) : 0 < \rho \le s \}
        \quad \text{for $0 < s < r$} \,,
    \end{displaymath}
    there holds
    \begin{displaymath}
        \zeta(s) \le \kappa \zeta(2^{20}s)
        \quad \text{whenever $0 < 2^{20}s < r$} \,.
    \end{displaymath}
\end{theorem}
     
\begin{proof}
    Let $A$ be as in~\ref{lem:generic-props}. Fix $a \in A$. Using translations
    we may and shall assume $a = 0$. Set
    \begin{displaymath}
        Q = \density^{\vdim}(\|V\|,a) \,,
        \quad
        T = \Tan^{\vdim}(\|V\|,a) \,,
        \quad
        \delta = 2^{-4} \,.
    \end{displaymath}
    Let $\Delta_0$ be the constant appearing in the Sobolev embedding
    theorem~\cite[7.10]{Gilbarg2001} for functions from the
    space~$\VSobz{1,2}(\R^{\adim} \cap \oball 01,\R^{\adim-\vdim})$.
    Employ~\ref{lem:approx-function}, \ref{lem:lip-approx},
    \ref{lem:L1-T-1-est}, and~\ref{lem:generic-props} to define
    \begin{gather}
        G = G_{\ref{lem:approx-function}}(\adim,\vdim) \,,
        \quad
        \Lambda = \min \bigl\{ 2^{-8}(6 + 2^5\Delta_0)^{-1} ,\, \kappa_{\ref{lem:approx-function}}(\adim,\vdim) ,\, \eta_{0,\ref{lem:schauder-est}}(\adim,\vdim,\tfrac 12,G) ,\, Q^{-1} \bigr\} \,,
        \\
        \gamma = \min\bigl\{ 2^{-8} ,\, 2^{-8} \unitmeasure{\vdim} ,\,
        \varepsilon_{\ref{lem:L1-T-1-est}}(\adim,\vdim) ,\,
        \varepsilon_{\ref{lem:approx-function}}(\adim,\vdim) ,\,
        \varepsilon_{\ref{lem:lip-approx}}(\adim,\vdim,Q,\Lambda,2Q \cdot 4^{\vdim},\delta)
        \bigr\} \,,
        \\
        \Delta_1 = \max \bigl\{ 1 ,\, Q \Gamma_{\text{\cite[5.7(8)]{Menne2012a}}}(\vdim) \bigr\} \,,
        \quad
        \Delta_2 = 2^{14\vdim}Q \Lip \uD^2G \,,
        \\
        \Delta_3 = \max \bigl\{ \Gamma_{\text{\cite[6.3]{Menne2012a}}}(\vdim, 2\vdim/(2+\vdim)) ,\, 2^{5\vdim+2}\bigr\} \,,
        \quad
        \Delta_4 = \Delta_1 \Delta_3 \max\bigl\{ \sqrt{2} (1 + \Lambda) ,\, 4 Q \bigr\} \,,
        \\
        \Delta_5 = 2^{9\vdim} \max\bigl\{ 1 ,\,
        \Gamma_{\text{\cite[5.7(7)]{Menne2012a}}}(Q,\vdim)
        \Gamma_{\ref{lem:lip-approx}\ref{i:bs:geometric-tilt}\ref{i:bs:gt:bad-set-global}}
        (Q,\vdim,\adim,\Lambda,\gamma) \bigr\} \,,
        \\
        \Delta_6 = 2^{7\vdim+5} Q \,,
        \quad
        \Delta_7 = \max \bigl\{ \Delta_4 (\sqrt{\Delta_5}+1) ,\, 2^4
        \Gamma_{\ref{lem:v-u-le-LFu-LFv}}(\adim,\vdim,\Lip \uD^2G) \Delta_4 \Delta_5 (\Delta_6 + \Delta_2 + 1) \bigr\}  \,,
        \\
        \Delta_8 = 2^{\vdim} \bigl( \Gamma_{0,\ref{lem:approx-function}}(\adim,\vdim)
        + \Gamma_{\ref{lem:approx-function}}(\adim,\vdim,2) \bigr)
        \Gamma_{\ref{lem:lip-approx}\ref{i:bs:g-L2-norm}}(Q,\vdim,\adim,\Lambda,\gamma) \,,
        \quad
        \Delta_9 = 4\Gamma_{\text{\cite[5.7(7)]{Menne2012a}}}(Q,\vdim) \vdim \besicovitch{\adim} \,,
        \quad
        \Delta_{10} = \max \bigl\{ (2\vdim/\unitmeasure{\vdim})^{1/2} ,\, 2^{\vdim/2} \bigr\} \,,
        \\
        \lambda = \min \bigl\{ \kappa^{\vdim} ( 2 (10 + 2\adim) \Delta_7 )^{-\vdim} ,\, \tfrac 12 \unitmeasure{\vdim} \bigr\}\,,
        \\
        \varepsilon = \min \bigl\{
        \gamma^{2} \Delta_9^{-1} ,\,
        \Lambda \Delta_8^{-1} ,\,
        \gamma (8\vdim)^{-1/2} \Delta_8^{-1} ,\,
        4^{-\vdim-1} \gamma \Delta_{10}^{-1} ,\,
        2^{-4} \gamma ,\,
        \sqrt{\lambda} \kappa  (2(10 + 2\adim)\Delta_7)^{-1} ( 1 + 2^{2\vdim+1} Q \unitmeasure{\vdim} \Delta_8 )^{-1}
        \bigr\} ,\,
        \\
        r_0 = r_{\ref{lem:generic-props}}(U,V,\varepsilon,\delta,a) \,,
        \quad
        r = \min \bigl\{ 1 ,\, 2^{-20} r_0 \bigr\} 
        \,.
    \end{gather}
    In~particular, $G$ is an integrand which coincides with the nonparametric
    area integrand~\ref{def:area-integrand} on $\cball{0}{\Lambda} \subseteq
    \Hom(\R^{\vdim},\R^{\adim-\vdim})$ and $\| \uD^2 G(\sigma) - \Upsilon \| \le
    \varepsilon_{\ref{lem:approx-function}}(\adim,\vdim)$ for $\sigma \in
    \Hom(\R^{\vdim},\R^{\adim-\vdim})$, where $\Upsilon$ corresponds to the
    Laplace operator as in~\ref{rem:Laplacian}; the passage
    from~\ref{lem:approx-function}, which bounds the quadratic form $\tau
    \mapsto \langle \tau\odot\tau, \uD^2G(\sigma)-\Upsilon\rangle$, to the
    norm $\|\uD^2G(\sigma)-\Upsilon\|$ is legitimate because a~symmetric
    bilinear form with values in~$\R$ attains its norm on the diagonal.  Since
    $\varepsilon_{\ref{lem:approx-function}}(\adim,\vdim) <
    \varepsilon_{\ref{lem:L1-T-1-est}}(\adim,\vdim)$
    by~\ref{lem:approx-function}, the integrand~$G$ is admissible
    in~\ref{lem:v-u-le-LFu-LFv}, whose constant we~abbreviate
    $\Gamma_{\ref{lem:v-u-le-LFu-LFv}}(\adim,\vdim,\Lip \uD^2G)$, the third
    argument being $\Lambda\Theta$ with $\Lambda = \Lip \uD^2G$ and $\Theta =
    1$; since moreover
    $\varepsilon_{\ref{lem:approx-function}}(\adim,\vdim) < 1/2$, we~have
    \begin{displaymath}
        \tfrac 12 |\tau|^2 \le \langle \tau \odot \tau ,\, \uD^2 G(\sigma) \rangle \le 2|\tau|^2
        \quad \text{for $\sigma,\tau \in \Hom(\R^{\vdim},\R^{\adim-\vdim})$} \,,
    \end{displaymath}
    so that $G$ satisfies the ellipticity hypothesis of~\ref{lem:schauder-est}
    with $\lambda = 1/2$; whence, $\eta_{0,\ref{lem:schauder-est}}(\adim,\vdim,\tfrac
    12,G)$ is defined and, $G$ being determined by $\adim$ and~$\vdim$, depends
    only on $\adim$ and~$\vdim$. Choose $p \in
    \orthproj{\adim}{\vdim}$ and $q \in \orthproj{\adim}{\adim - \vdim}$ such
    that $\im p^* = T = \ker q$.

    We~record how the hypotheses of~\ref{lem:lip-approx} numbered (l), (m),
    and~(n) are met, the remaining ones being immediate
    from~\ref{lem:generic-props}.  Hypothesis~(l) holds because $\delta V = 0$
    and $\measureball{\|V\|}{\cball a\rho} \ge Q \unitmeasure{\vdim}
    \rho^{\vdim} > 0$ by the monotonicity formula, which also gives~(n)
    together with the first line of~\eqref{eq:gp:choice} and $Q \ge 1$.  For~(m)
    let $0 < \rho < 4r_0$.  In~case $\rho < r_0$
    then~\ref{lem:generic-props}\ref{i:gp:L2-tilt-small} and $\cball a\rho
    \subseteq \cylinder Ta{\rho}{r_0}$ give $\textint{\cball a\rho \times
      \grass{\adim}{\vdim}}{} \| \project S - \project T \|^2 \ud V <
    \varepsilon^2 \rho^{\vdim}$, whereas in~case $\rho \ge r_0$ the fourth line
    of~\eqref{eq:gp:choice} gives the same with $2^{\vdim}\varepsilon^2
    \unitmeasure{\vdim} \rho^{\vdim} / (2\vdim)$ on the right, since $(2r_0)^{\vdim}
    \le (2\rho)^{\vdim}$; hence, by
    H{\"o}lder's inequality, \eqref{eq:proj-norms}, and $\measureball{\|V\|}{\cball
      a\rho} \ge \unitmeasure{\vdim}\rho^{\vdim}$,
    \begin{displaymath}
        \textint{\cball a\rho \times \grass{\adim}{\vdim}}{} | \project S - \project T | \ud V
        \le \Delta_{10} \varepsilon \measureball{\|V\|}{\cball a\rho}
        \le 4^{-\vdim-1} \gamma \measureball{\|V\|}{\cball a\rho} \,,
    \end{displaymath}
    the two cases contributing the two entries of~$\Delta_{10}$ and the last
    step being the definition of~$\varepsilon$.  Hypothesis~(c)
    of~\ref{lem:lip-approx} prescribes the ambient set, so we~apply that lemma
    to the restriction of~$V$ to $U' = \cylinder Ta{r_0}{r_0} + \oball 0{2r_0}$,
    which by~\ref{lem:generic-props}(1) is contained in $\oball a{4r_0}$; the
    conclusions concern only cylinders about~$a$ of radius less than~$r_0$ and
    are unaffected.  Hypothesis~\ref{lem:lip-approx}\ref{i:la:height} then
    holds: every $z \in \spt\|V\| \cap U'$ satisfies $|z-a| \le 4r_0$, so that
    the third line of~\eqref{eq:gp:choice}, applied with $s = |z-a|$, gives
    $\dist(z,a+T) \le \tfrac 14 \min\{\varepsilon,1-2\delta\}|z-a| \le \tfrac 12
    |z-a|$.

    Now~\ref{lem:generic-props} ensures that all
    the assumptions of~\ref{lem:lip-approx} with
    \begin{align}
        \adim,\, \vdim,\, Q,\, \Lambda,\, 2Q \cdot 4^\vdim,\, &\delta,\, \gamma,\, r_0,\, r_0,\, T,\, p,\, q,\, U,\, V
        \\
        \text{in place of} \quad
        \adim,\, \vdim,\, Q,\, L,\, M,\, &\delta,\, \gamma,\, r,\, h,\, T,\, p,\, q,\, U,\, V
    \end{align}
    are satisfied and we let $B$, $X$, $f$, and~$g$ be defined as
    in~\ref{lem:lip-approx}. In~particular, $f : X \to
    \qspace_{Q}(\R^{\adim-\vdim})$ is characterised by
    \begin{displaymath}
        \density^0(\|f(x)\|,q(z)) = \density^{\vdim}(\|V\|,z)
        \quad \text{whenever $x \in X$, $p(z) = x$, and $z \in \cylinder Ta{r_0}{r_0} \without B$} \,,
    \end{displaymath}
    $\Lip f \le \Lambda$, and $g : \R^{\vdim} \to \R^{\adim-\vdim}$ is the
    Lipschitz extension of~$\boldsymbol{\eta}_Q \circ f$ with $\Lip g \le
    \Lambda Q^{-1/2}$. Employ~\cite[2.3]{Menne2012a} to find a~countable index
    set $I$ and a family of Lipschitz functions $\{ f_i : i \in I \}$ with the
    properties described therein. For $i \in I$ and $x \in \R^{\vdim}$ define
    \begin{equation}
        \label{eq:Ix-Di}
        D_i = \dmn \ap \uD f_i \,,
        \quad
        D = \tbcup \{ D_i : i \in I \} \,,
        \quad \text{and} \quad
        I(x) = I \cap \{ i : x \in D_i \} \,.
    \end{equation}
    Note that
    \begin{equation}
        \label{eq:fi-Di-props}
        \HM^0(I(x)) = Q \quad \text{for $x \in D$} \,,
        \quad
        \LM^{\vdim}(X \without D) = 0 \,,
        \quad
        \Lip f_i \le \Lip f \quad \text{for $i \in I$} \,.
    \end{equation}
    For $0 < s < r_0$ define $v_s \in \espace{\oball 0{2^5s}}{\R^{\adim-\vdim}}$
    to be the map given by~\ref{lem:approx-function}, i.e., such that
    \begin{displaymath}
        L_G(v_s) = 0 
        \quad \text{and} \quad
        v_s - g \in \VSobz{1,2}(\oball 0{2^5s},\R^{\adim - \vdim}) 
    \end{displaymath}
    and set
    \begin{displaymath}
        w_s = p^* + q^* \circ v_s \,,
        \quad
        M_s = w_s\lIm \oball{0}{2s} \rIm  \,.
    \end{displaymath}
    Using~\ref{def:height-and-tilt} for $0 < s < r_0$ and $x \in D \cap \oball
    0{2^5s}$ we define
    \begin{gather}
        \phi(s) = \phi(v_s,p,a,s) \,,
        \quad
        \psi(s) = \psi(v_{s},p,a,\tfrac 12 s) \,,
        \quad
        \tau(s) = \tau(T,a,s) \,,
        \quad
        \beta(s) = \beta(v_s,p,a,s) \,,
        \\
        \quad
        \xi(s) = \sup \bigl\{ \psi(\rho) : 0 < \rho \le s \bigr\} \,,
        \quad \text{and} \quad
        k_s(x) = \mathscr{G}(f(x),Q\boldsymbol{[} v_s(x) \boldsymbol{]}) \,.
    \end{gather}
    Note the ``$\frac 12$'' in the definition of $\psi$. We shall control all
    terms first with $\psi$ and then use~\ref{rem:tilt-height} to pass
    to~$\beta$. However, when using~\ref{rem:tilt-height} we need to enlarge the
    set over which we integrate and, at that point, we will not be able to
    easily change~$M_s$ to~$M_{2s}$; hence, we need to do all the estimates with
    respect to~$M_{2s}$ from the beginning.

    Combining~\ref{lem:approx-function}\eqref{eq:u-lo-est}
    and~\ref{lem:approx-function}\eqref{eq:u-hi-est}, the former on
    $\cball 0{2^{3}s} \supseteq \cball 0{4s}$ and the latter, with $k = 2$, on
    $\cball 0{2s}$, with~\ref{lem:lip-approx}\ref{i:bs:g-L2-norm}, and
    \ref{lem:generic-props}\ref{i:gp:L2-tilt-small}, we see that
    \begin{multline}
        \label{eq:Lip-vs-est}
        \norm{\uD v_s}{\infty}{0,4s} + 2^5s \norm{\uD^2 v_s}{\infty}{0, 2s}
        \le \bigl( \Gamma_{0,\ref{lem:approx-function}}(\adim,\vdim)
        + \Gamma_{\ref{lem:approx-function}}(\adim,\vdim,2) \bigr)
        (2^5s)^{-\vdim/2} \norm{\uD g}{2}{0,2^5s}
        \\
        \le \Delta_8 \tau(2^7s)
        \le \Delta_8 \varepsilon \le \Lambda \le 2^{-8}
        \quad \text{for $0 < s < 2^{-7}r_0$}
        \,.
    \end{multline}
    Moreover, using~\cite[4.2]{KolRMI}, the triangle inequality,
    and~\ref{lem:generic-props}\ref{i:gp:dens-ratio}, we deduce
    \begin{equation}
        \label{eq:psi-small}
        \psi(s)
        \le \tau(s) + 2^{2\vdim+1}Q\unitmeasure{\vdim} \Lip v_s|\oball 0s
        \le \varepsilon (1 + 2^{2\vdim+1}Q\unitmeasure{\vdim} \Delta_8)
        \quad \text{for $0 < s < 2^{-7}r_0$} \,.
    \end{equation}
    From~\ref{lem:generic-props}\ref{i:gp:tilt-uniform} we see that the set
    termed ``$H$'' in~\cite[5.7(8)]{Menne2012a} coincides with $\cylinder
    Ta{r_0}{r_0}$ up~to a~set of $\|V\|$~measure zero. For $0 < s < 2^{-5}r_0$
    and $x \in D \cap \oball{0}{2^5s}$ we have $\Lip v_s|\cball{0}{4s} \le
    2^{-8}$ and $\sup \bigl\{ |y-v_s(x)| : y \in \spt f(x) \bigr\} \le k_s(x)$;
    thus, we may employ~\cite[5.7(8)]{Menne2012a} with~$M_s \cap \cylinder T0ss$
    in place of~$P$ to~get
    \begin{equation}
        \label{eq:hs-fst}
        \phi(s)
        \le \Delta_1 
        \bigl( s^{-(\vdim+2)} \textint{X \cap \cball 0s}{} |k_s|^2 \ud \LM^{\vdim} \bigr)^{1/2}
        + \Delta_1 s^{-(\vdim+2)/2} \LM^{\vdim} ( \cball 0s \without X )^{1/2 + 1/\vdim}
        \,.
    \end{equation}
    Define
    \begin{displaymath}
        Y = X \cap \bigl\{ x : f(x) = Q \boldsymbol{[}
        g(x) \boldsymbol{]} \bigr\}
        \subseteq p \bigl\lIm \cylinder T0{r_0}{r_0} \cap \{ z : \density^{\vdim}(\|V\|,z) = Q \} \bigr\rIm
    \end{displaymath}
    and note that
    \begin{equation}
        \label{eq:notQset-lambda}
        \sup \bigl\{ \rho^{-\vdim} \LM^{\vdim} ( \cball 0\rho \without Y ) : 0 < \rho < r_0 \}
        \le \varepsilon \le \lambda \le \tfrac 12 \unitmeasure{\vdim} 
    \end{equation}
    due to~\ref{lem:generic-props}\ref{i:gp:density=Q} because if $z \in
    \cylinder T0{r_0}{r_0}$, $\density^{\vdim}(\|V\|,z) = Q$, and $p(z) \notin
    Y$, then $p(z) \notin X$ and $p(z)$ lies in the set termed ``$N$''
    in~\cite[5.7]{Menne2012a} which is an~$\LM^{\vdim}$~null set
    by~\cite[5.7(1)]{Menne2012a}. Since $k_s|Y = \sqrt{Q} (g - v_s)|Y$,
    employing the Sobolev-Poincar{\'e} inequality in the
    form~\cite[6.3]{Menne2012a} with $2$, $2\vdim/(2+\vdim)$, $1$ in place of
    $q$, $\zeta$, $\xi$, we get for $0 < s < 2^{-1}r_0$
    \begin{displaymath}
        s^{-(\vdim+2)/2} \norm{k_{2s}}{2}{0,s}
        \le \Delta_3 \lambda^{1/\vdim} s^{-\vdim/2} \norm{\uD k_{2s}}{2}{0,s} 
        + \Delta_3 \lambda^{-1/2} s^{-\vdim-1} \textint{Y \cap \cball 0s}{} | g - v_{2s} | \ud \LM^{\vdim}
    \end{displaymath}
    and
    \begin{multline}
        s^{-(\vdim+2)/2} \norm{v_s - v_{2s}}{2}{0,s}
        \le \Delta_3 \lambda^{1/\vdim} s^{-\vdim/2} \norm{\uD(v_s - v_{2s})}{2}{0,s}
        + \Delta_3 \lambda^{-1/2} s^{-\vdim-1} \norm{v_s - v_{2s}}{1}{0,s} 
        \\
        \le \Delta_3 \lambda^{1/\vdim} s^{-\vdim/2} \bigl( \norm{\uD(v_s - g)}{2}{0,s} + \norm{\uD(v_{2s} - g)}{2}{0,s} \bigr)
        \\
        + \Delta_3 \lambda^{-1/2} s^{-\vdim-1} \bigl( \norm{v_s - g}{1}{0,s} + \norm{v_{2s} - g}{1}{0,s} \bigr) \,;
    \end{multline}
    hence, since 
    \begin{displaymath}
        | k_s(x) - k_{2s}(x) |
        \le \mathscr{G}(Q \boldsymbol{[} v_s(x)\boldsymbol{]}, Q \boldsymbol{[} v_{2s}(x)\boldsymbol{]})
        = Q^{1/2} |v_s(x) - v_{2s}(x)|
        \quad \text{for $x \in D \cap \oball 0s$}  \,,
    \end{displaymath}
    employing~\ref{lem:approx-function}\eqref{eq:Du-g} to replace~$\norm{\uD(v_s
      - g)}{2}{0,s}$ with~$3\norm{\uD(v_{2s} - g)}{2}{0,2^5s}$ we obtain
    \begin{multline}
        \label{eq:ks-poincare}
        s^{-(\vdim+2)/2} \norm{k_s}{2}{0,s}
        \le s^{-(\vdim+2)/2}
        \bigl(
        \norm{k_{2s}}{2}{0,s}
        + \norm{k_s - k_{2s}}{2}{0,s}
        \bigr)
        \\
        \le \Delta_3 \lambda^{1/\vdim} s^{-\vdim/2}
        \bigl(
        \norm{\uD k_{2s}}{2}{0,s}
        + 4 Q^{1/2} \norm{\uD(g - v_{2s})}{2}{0,2^5s}
        \bigr)
        \\
        + Q^{1/2} \Delta_3 \lambda^{-1/2} s^{-\vdim-1}
        \bigl(
        \norm{g-v_s}{1}{0,s}
        + 2 \norm{g-v_{2s}}{1}{0,s}
        \bigr) \,.
    \end{multline}
    Observing that $\Lip k_{2s} \le \Lambda (1 + \sqrt{Q})$ and using $\Lambda \le
    Q^{-1}$ an application of~\ref{cor:Dg-Dh-tilt} yields
    \begin{multline}
        \label{eq:ks-est}
        s^{-\vdim/2} \norm{\uD k_{2s}}{2}{0,s}
        \le \Bigl( s^{-\vdim} \textint{\cball 0s \cap X}{} \| \uD k_{2s} \|^2 \ud \LM^{\vdim} \Bigr)^{1/2}
        +  \Lip k_{2s} s^{-\vdim/2} \LM^{\vdim}(\cball 0s \without X)^{1/2}
        \\
        \le 2Q \psi(2s) + 2 s^{-\vdim/2} \LM^{\vdim}(\cball 0s \without X)^{1/2}
        \quad \text{for $0 < s < 2^{-7}r_0$} \,.
    \end{multline}
    Combining~\eqref{eq:hs-fst} and~\eqref{eq:ks-poincare}
    with~\eqref{eq:ks-est} and noting that $s^{-1} \LM^{\vdim}(\cball 0s
    \without X)^{1/\vdim} \le \lambda^{1/\vdim}$, by~\eqref{eq:notQset-lambda},
    we get
    \begin{multline}
        \label{eq:phis-interm}
        \phi(s)
        \le \Delta_4 \lambda^{1/\vdim} \bigl(
        s^{-\vdim/2} \LM^{\vdim}(\cball 0s \without X)^{1/2}
        + \psi(2s)
        + s^{-\vdim/2} \norm{\uD(g-v_{2s})}{2}{0,2^5s}
        \bigr)
        \\
        + \Delta_4 \lambda^{-1/2} s^{-\vdim-1} \bigl(
         \norm{g - v_s}{1}{0,s}
        + \norm{g - v_{2s}}{1}{0,s}
        \bigr)
        \quad \text{for $0 < s < 2^{-7}r_0$} \,.
    \end{multline}
    Since $g-v_{\rho} \in \VSobz{1,2}(\oball 0{2^5\rho},\R^{\adim - \vdim})$
    whenever $0 < \rho < r_0$ we employ~\ref{lem:v-u-le-LFu-LFv} twice: first
    with $g$, $v_s$, $v_{2^6s}$ and then with $g$, $v_{2s}$, $v_{2^7s}$ in place
    of~$u$, $v$, $w$; here $\Theta = 1$ is admissible.  Indeed the radius~$r$
    of~\ref{lem:v-u-le-LFu-LFv} is the one determined by the boundary
    condition, namely $2^5s$ in the first application and $2^6s$ in the second,
    whereas the comparison functions carry the scales $2^6s$ and $2^7s$; since
    $\Lip \uD v_{\varrho}|\cball 0{2\varrho} \le \Lambda 2^{-5}\varrho^{-1}$
    by~\eqref{eq:Lip-vs-est}, we~get $2^5 s \Lip \uD v_{2^6s} \le \Lambda
    2^{-6} \le 1$ and, likewise, $2^6 s \Lip \uD v_{2^7s} \le \Lambda 2^{-6}
    \le 1$. Next, we apply the triangle inequality twice
    and~\ref{lem:approx-function}\eqref{eq:Du-g} five times to obtain
    \begin{multline}
        \label{eq:L1g-vs-est}
        s^{-\vdim-1} \bigl(
         \norm{g - v_s}{1}{0,s}
        + \norm{g - v_{2s}}{1}{0,s}
        \bigr)
        \le 2^7 \Gamma_{\ref{lem:v-u-le-LFu-LFv}}(\adim,\vdim,\Lip \uD^2G) s^{-\vdim} \norm{L_G(g)}{-1,1}{0,2^6s}
        \\
        + 2^6 \Gamma_{\ref{lem:v-u-le-LFu-LFv}}(\adim,\vdim,\Lip \uD^2G) s^{-\vdim} \Lip \uD^2G \bigl(
        \norm{ \uD(g-v_{2^6s}) }{2}{0,2^5s}^2
        + \norm{ \uD(v_s-v_{2^6s}) }{2}{0,2^5s}^2
        \\
        + \norm{ \uD(g-v_{2^7s}) }{2}{0,2^6s}^2
        + \norm{ \uD(v_{2s}-v_{2^7s}) }{2}{0,2^6s}^2
        \bigr)
        \\
        \le 2^7 \Gamma_{\ref{lem:v-u-le-LFu-LFv}}(\adim,\vdim,\Lip \uD^2G) s^{-\vdim} \bigl(
        \norm{L_G(g)}{-1,1}{0,2^6s}
        + 46 \norm{ \uD(g-v_{2^7s}) }{2}{0,2^6s}^2
        \bigr)
        \quad \text{for $0 < s < 2^{-7}r_0$}
        \,.
    \end{multline}
    Since $\Lip v_{2^7s}|\cball{0}{2^{9}s} \le \Lambda$ by~\eqref{eq:Lip-vs-est}
    we may employ~\ref{cor:Dg-Dh-tilt}. Recalling $0 \in X$, $f(0) = Q
    \boldsymbol{[} 0 \boldsymbol{]}$, and $\Lip f \le \Lambda < 1$, we see also
    that the $Q$-graph of~$f$ over $X \cap \cball 0\rho$ lies in~$\cylinder
    T0\rho\rho$, i.e.,
    \begin{displaymath}
        \bigl\{ p^*(x) + q^*(y) : x \in X \cap \cball 0{\rho} ,\, y \in \spt f(x) \bigr\}
        \subseteq \cylinder T0\rho\rho
        \quad \text{for $0 < \rho < r_0$} \,;
    \end{displaymath}
    hence, for $0 < s < 2^{-7}r_0$
    \begin{multline}
        \label{eq:Dg-vs-squared}
        s^{-\vdim} \norm{ \uD(g-v_{2^7s}) }{2}{0,2^6s}^2
        \\
        \le 2Q s^{-\vdim} \int_{p^{-1}\lIm X \rIm \cap \cylinder T0{2^6s}{2^6s}}
        \| \project{\Tan( M_{2^7s}, w_{2^7s} \circ p(z) )} - \project{S} \|^2 \ud V(z,S)
        \\
        + s^{-\vdim} \int_{\cball 0{2^6s} \without X} \|\uD(g-v_{2^7s})\|^2 \ud \LM^{\vdim} 
        \le \Delta_6 \bigl( \psi(2^7s)^2 + s^{-\vdim} \LM^{\vdim}(\cball 0{2^6s} \without X) \bigr) \,,
    \end{multline}
    where we have used that $\Delta_6 \ge \max \{ 2^{7\vdim+1}Q ,\, \bigl( \Lip
    g + \Lip v_{2^7s}|\cball{0}{2^6s} \bigr)^2 \}$ because $\Lip g \le 1$ and
    $\Lip v_{2^7s}|\cball{0}{2^6s} \le 1$.

    Now, we need to estimate $s^{-\vdim}\norm{L_G(g)}{-1,1}{0,2^6s}$. We~shall
    write ``$\ap \uD$'' for the approximate derivative as
    in~\cite[3.1.2]{Federer1969}. Recalling~\ref{rem:LF-domain-problem}, for
    brevity of the notation, we~define
    \begin{displaymath}
        g_{s} = g|\oball 0{2^6s} 
        \quad \text{and} \quad
        u_{s} = v_{2^{10}s}|\oball{0}{2^{6}s}
        \quad \text{for $0 < s < 2^{-10}r_0$} \,,
    \end{displaymath}
    and we let $f_{i,s} : \oball 0{2^6s} \to \R^{\adim-\vdim}$ to be a~Lipschitz
    extension of~$f_i|\oball 0{2^6s} \cap \dmn f_i$. Recall the definitions of
    ``$L_F$'' and ``$T_A$'' from~\ref{def:ell-system}. For $0 < s < 2^{-10}r_0$
    and $\theta \in \dspace{\oball 0{2^6s}}{\R^{\adim-\vdim}}$ we have
    \begin{displaymath}
        L_G(g_s)\theta = (L_G(g_s) - L_G(u_{s}))\theta
        = - \textint{}{} \bigl\langle \uD \theta(x) , \uD G( \uD g_s(x) ) - \uD G( \uD u_{s}(x) ) \bigr\rangle \ud \LM^{\vdim}(x) \,.
    \end{displaymath}
    Whenever $0 < s < 2^{-10}r_0$, $x \in \oball 0{2^6s}$, $i \in I$, $y \in
    \oball 0{2^6s} \cap D_i$, and $z \in \oball 0{2^6s} \without D_i$ set
    \begin{gather}
        R(x) = \textint 01 \uD^2 G(t \uD g_s(x) + (1-t) \uD u_{s}(x)) \ud \LM^1(t) \,,
        \quad 
        H(x) = \uD^2G( \uD u_{s}(x)) \,,
        \\
        F_i(y) = \textint 01 \uD^2 G(t \ap \uD f_i(y) + (1-t) \uD u_{s}(y)) \ud \LM^1(t) \,,
        \quad
        F_i(z) = 0  \,.
    \end{gather}
    Observe that for $0 < s < 2^{-10}r_0$, $\theta \in \dspace{\oball
      0{2^6s}}{\R^{\adim-\vdim}}$, $x \in \oball 0{2^6s}$, $i \in I$, and $y \in
    \oball 0{2^6s} \cap D_i$ there holds
    \begin{gather}
        \bigl\langle \uD \theta(x) ,\, \uD G( \uD g_s(x) ) - \uD G( \uD u_{s}(x) ) \bigr\rangle
        = \bigl\langle \uD \theta(x) \odot \uD (g_s-u_{s})(x) ,\, R(x) \bigr\rangle 
        \,,
        \\
        \bigl\langle \uD \theta(y) ,\, \uD G( \ap \uD f_i(y) ) - \uD G( \uD u_{s}(y) ) \bigr\rangle
        = \bigl\langle \uD \theta(y) \odot \ap \uD (f_i-u_{s})(y) ,\, F_i(y) \bigr\rangle 
        \,;
    \end{gather}
    hence,
    \begin{multline}
        \label{eq:LGg}
        L_G(g_s)
        = L_G(g_s) - L_G(u_{s})
        = T_R(g_s-u_{s})
        = T_{R-H}(g_s-u_{s}) + T_{H}(g_s-u_{s})
        \\
        = T_{R-H}(g_s-u_{s}) + T_{H;\oball 0{2^6s} \without X}(g_s-u_{s}) + T_{H;X}(g_s-u_{s}) 
    \end{multline}
    and
    \begin{multline}
        \label{eq:THXg-v}
        Q T_{H;X}(g_s-u_{s})
        = \textsum{i \in I}{} T_{H;D_i}(f_{i,s}-u_{s})
        = \textsum{i \in I}{} \bigl( T_{H-F_i,D_i}(f_{i,s}-u_{s}) + T_{F_i;D_i}(f_{i,s}-u_{s}) \bigr)
        \\
        = \textsum{i \in I}{} \bigl( T_{H-F_i,D_i}(f_{i,s}-u_{s}) + L_{G;D_i}(f_{i,s}) - L_{G;D_i}(u_{s}) \bigr)
        \\
        = \textsum{i \in I}{} \bigl( T_{H-F_i,D_i}(f_{i,s}-u_{s}) + L_{\Phi;D_i}(f_{i,s}) \bigr) + Q L_{G;\oball 0{2^6s} \without X}(u_{s}) \,,
    \end{multline}
    where $\Phi$ is the nonparametric area integrand defined
    in~\ref{def:area-integrand} and we have used $L_G(u_{s}) = 0$ and $\Lip f_i
    \le \Lambda$ which implies $\uD G(\ap \uD f_i(x)) = \uD \Phi(\ap \uD
    f_i(x))$ for $i \in I$ and $x \in D_i$. Since
    \begin{multline}
        |(R-H)(x)| = \bigl| \textint 01 \uD^2 G(t \uD g_s(x) + (1-t) \uD u_{s}(x)) - \uD^2G(\uD u_{s}(x)) \ud \LM^1(t) \bigr|
        \\
        \le \tfrac 12 \Lip \uD^2 G | \uD (g_s-u_{s})(x) |
        \quad \text{for $x \in \oball 0{2^6s}$} 
    \end{multline}
    and, in the same way,
    \begin{displaymath}
        |(H-F_i)(x)|
        \le \tfrac 12 \Lip \uD^2 G | \ap \uD (f_{i,s}-u_{s})(x) |
        \quad \text{for $i \in I$ and $x \in D_i \cap \oball 0{2^6s}$} \,,
    \end{displaymath}
    we obtain, setting $l(x) = (\trans{-u_s(x)})_{\#}f(x)$ for $x \in X$ and
    using~\ref{lem:Dk-tilt} and~\ref{cor:Dg-Dh-tilt},
    \begin{gather}
        \label{eq:LR-H-gvs}
        \norm{T_{R-H}(g_s-u_{s})}{-1,1}{0,2^6s}
        \le \tfrac 12 \Lip \uD^2G \norm{\uD(g_s-u_{s})}{2}{0,2^6s}^2
        \le \Delta_2 s^{\vdim} \psi(2^{10}s)^2 \,,
        \\
        \label{eq:LH-Fi-fvs}
        Q^{-1} \textsum{i \in I}{} \norm{T_{H-F_i;D_i}(f_{i,s}-u_{s})}{-1,1}{0,2^6s}
        \le \tfrac 12 \Lip \uD^2G \norm{\ap \Aff l}{2}{0,2^6s}^2
        \le \Delta_2 s^{\vdim} \psi(2^{10}s)^2 \,.
    \end{gather}
    Employing~\cite[5.7(6)(7)]{Menne2012a}
    and~\ref{lem:lip-approx}\ref{i:bs:geometric-tilt}\ref{i:bs:gt:bad-set-global}
    we also get for $0 < 2^{17}s < r_0$
    \begin{equation}
        \label{eq:bad-set-psi}
        \LM^{\vdim}(\oball 0{2^6s} \without X)
        \le \Gamma_{\text{\cite[5.7(7)]{Menne2012a}}}(Q,\vdim) \|V\|(B \cap \cylinder{T}{a}{2^6s}{2^6s})
        \le \Delta_5 s^{\vdim} \psi(2^{10}s)^2 \,,
    \end{equation}
    the region occurring
    in~\ref{lem:lip-approx}\ref{i:bs:geometric-tilt}\ref{i:bs:gt:bad-set-global}
    at the scale $2^6s$ being $\cylinder Ta{2^9s}{2^9s}$, which is the cylinder
    of $\psi(2^{10}s)$;
    where we have used that $\Lip v_{2^{10}s} \le \Delta_8 \varepsilon \le
    \gamma(8\vdim)^{-1/2}$ by~\eqref{eq:Lip-vs-est} and the definition
    of~$\varepsilon$, which is the smallness required
    in~\ref{lem:lip-approx}\ref{i:bs:geometric-tilt}.  Now, we need to take care of the term
    $\textsum{i \in I}{} L_{\Phi;D_i}(f_{i,s})$. Let
    \begin{gather}
        N_i = \bigl\{ p^*(x) + q^*(f_i(x)) : x \in D_i \bigr\} \,,
        \quad
        N = \tbcup \{ N_i : i \in I \} \,,
        \\
        \text{and} \quad
        W = \textsum{i \in I}{} \var{\vdim}(N_i) \in \IVar{\vdim}(\R^{\adim})\,.
    \end{gather}
    Then, directly from the
    definitions~\ref{lem:lip-approx}\ref{i:la:X}\ref{i:la:f} of~$X$ and~$f$,
    \begin{displaymath}
        W = V \restrict ( N \times \grass{\adim}{\vdim} )
        \quad \text{and} \quad
        \| V-W \| \bigl( \cylinder T0{r_0}{r_0} \without B \bigr) = 0 \,.
    \end{displaymath}
    For $0 < s < 2^{-10}r_0$ and $\theta \in
    \dspace{\oball{0}{2^6s}}{\R^{\adim-\vdim}}$ we recall~\cite[2.6]{Menne2012a}
    and $\delta V = 0$ to write
    \begin{multline}
        \textsum{i \in I}{} L_{\Phi;D_i}(f_{i,s})\theta
        = \delta W(q^* \circ \theta \circ p)
        \\
        = \delta V(q^* \circ \theta \circ p) - \delta(V-W)(q^* \circ \theta \circ p)
        = \delta(V-W)(q^* \circ \theta \circ p) 
    \end{multline}
    and conclude using
    again~\ref{lem:lip-approx}\ref{i:bs:geometric-tilt}\ref{i:bs:gt:bad-set-global}
    \begin{equation}
        \label{eq:LPhi-fi}
        \norm{\textsum{i \in I}{} L_{\Phi;D_i}(f_{i,s})}{-1,1}{0,2^6s}
        \le \|V\|(B \cap \cylinder T0{2^6s}{2^6s})
        \le \Delta_5 s^{\vdim} \psi(2^{10}s)^2 \,.
    \end{equation}
    Combining~\eqref{eq:LPhi-fi}, \eqref{eq:bad-set-psi}, \eqref{eq:LH-Fi-fvs},
    \eqref{eq:LR-H-gvs}, \eqref{eq:THXg-v}, \eqref{eq:LGg},
    \eqref{eq:Dg-vs-squared}, \eqref{eq:L1g-vs-est} with~\eqref{eq:phis-interm}
    we obtain
    \begin{equation}
        \phi(s)
        \le \xi(2^{10}s) \Delta_7
        \bigl( \lambda^{1/\vdim} + \lambda^{-1/2} \xi(2^{10}s) \bigr)
        \quad \text{for $0 < 2^{17}s < r_0$}
        \,,
    \end{equation}
    where we have used that $\Delta_5 \ge 1$, $\Lip u_s \le 1$, and $\Lip g \le
    1$. The definitions of $\lambda$ and $\varepsilon$ and the
    estimate~\eqref{eq:psi-small} imply now that
    \begin{equation}
        \label{eq:final-phi}
        \phi(s)
        \le (10 + 2\adim)^{-1} \kappa \xi(2^{10}s)
        \quad \text{for $0 < s < 2^{-17}r_0$}
    \end{equation}

    The next step is to estimate~$\xi(s)$ by~$\phi(2s)$ for $0 < 2s < r_0$.
    To~this end we want employ~\ref{rem:tilt-height} with~$M_s$ in place of~$M$
    but to ensure that $\cylinder T0{2s}{2s}$ is contained in~$\Unp(M)$
    we~possibly need to translate~$M_s$. For $0 < s < r_0$ define $z_s \in
    \R^{\adim-\vdim}$ so that
    \begin{displaymath}
        \beta(v_s + z_s,p,0,s)
        = \inf \bigl\{ \beta(v_s + z,p,0,s) : z \in \R^{\adim-\vdim} \bigr\}
        \le \beta(s) 
    \end{displaymath}
    and note that
    \begin{displaymath}
        \psi(s) = \psi(v_s+z_s,p,0,s/2)
        \,.
    \end{displaymath}
    Employing~\eqref{eq:Lip-vs-est} together
    with~\ref{lem:convex-reach}\ref{i:cr:reach} (applied with $\cball 0{2s}$,
    $v_s$, $2^{-13}s^{-1}$, $\Lambda$ in place of $C$, $v$, $\Delta$, $\gamma$,
    so that $\reach(M_s) \ge 2^{13}s - \Lambda \diam \cball 0{2s} \ge (2^{13}-1)s$
    as $\Lambda \le 2^{-8}$) we~get for $0 < s < 2^{-7}r_0$
    \begin{gather}
        p\lIm M_s \rIm = \oball 0{2s} \,,
        \quad
        \Lip v_s|\cball 0{2s} \le \Lambda \,,
        \\
        \Lip \uD v_s|\cball 0{2s} \le \frac{\Lambda}{2^{5}s} \le 2^{-13} s^{-1} \,,
        \quad
        \reach(M_s) \ge (2^{13} - 1) s \,;
    \end{gather}
    hence,
    \begin{gather}
        M_s \subseteq \cylinder T{w_s(0)}{2s}{2\Lambda s} \,,
        \\
        \trans{q^{*}(z_s)}\lIm M_s \rIm
        \subseteq \cylinder T0{2s}{6\Lambda s} 
        \subseteq \cylinder T0{2s}{2s} 
        \subseteq \Unp(\trans{q^{*}(z_s)}\lIm M_s \rIm) \,.
    \end{gather}
    Recall~\ref{def:height-and-tilt} and define
    \begin{displaymath}
        \eta(s) = \eta(\trans{q^{*}(z_s)}\lIm M_s \rIm,T,0,s/2)
        \quad \text{for $0 < s < 2^{-7}r_0$} \,.
    \end{displaymath}
    Note that $\mathbf{h}(M_s,x) = 0$ for $x \in M_s$
    by~\ref{lem:approx-function} and~\eqref{eq:Lip-vs-est}.
    Since \ref{lem:heights} and~\ref{cor:tilt-tilt} concern an~entire graph
    whereas $\trans{q^{*}(z_s)}\lIm M_s \rIm$ is a~graph over $\oball 0{2s}$
    only, let $\tilde v : \R^{\vdim} \to \R^{\adim-\vdim}$ be a~Lipschitzian
    extension of $(v_s+z_s)|\cball 0{2s}$ with
    \begin{displaymath}
        \Lip \tilde v = \Lip (v_s+z_s)|\cball 0{2s} \le \Lambda \,,
    \end{displaymath}
    furnished by Kirszbraun's theorem~\cite[2.10.43]{Federer1969}, and put
    $\tilde f = p^* + q^* \circ \tilde v$ and $\tilde \Sigma = \im \tilde f$.
    That the extension may be taken with the \emph{same} Lipschitz constant is
    what matters, $\Lambda$ being the~$\gamma$ with which \ref{lem:heights}
    and~\ref{cor:np-in-2cyl} are applied.  Since $\tilde v$ agrees with
    $v_s+z_s$ on $\cball 0{2s}$, it is of class~$\cnt{\infty}$ on $\oball
    0{2s}$, so the local regularity required in~\ref{lem:heights}
    and~\ref{cor:tilt-tilt} is available at every point lying over $\oball
    0{2s}$.  Moreover every point of~$\tilde\Sigma$ that is relevant lies over
    $\cball 0{2s}$.  Indeed, let $a \in \cylinder T0{s/2}{s/2}$.  Since
    $\trans{q^{*}(z_s)}\lIm M_s \rIm \subseteq \cylinder T0{2s}{6\Lambda s}$ and
    $6\Lambda \le 6 \cdot 2^{-8} \le \tfrac 14$, the point $\tilde f(p(a))$
    satisfies $|a - \tilde f(p(a))| = |q(a) - \tilde v(p(a))| \le \tfrac 12 s +
    \tfrac 14 s$; hence $\dist(a,\tilde\Sigma) \le \tfrac 34 s$, and every $z \in
    \tilde\Sigma$ with $|a-z| \le \tfrac 34 s$ satisfies $|p(z)| \le \tfrac 12 s +
    \tfrac 34 s < 2s$.  On the part of~$\tilde\Sigma$ lying over $\oball 0{2s}$
    --- which is $\trans{q^{*}(z_s)}\lIm M_s \rIm$ --- nearest points are unique
    by~\ref{lem:convex-reach}\ref{i:cr:strict}; consequently $a \in
    \Unp(\tilde\Sigma)$, $\npp{\tilde\Sigma}(a) = \npp{\trans{q^{*}(z_s)}\lIm
      M_s \rIm}(a) \in M_s$, and $\dist(a,\tilde\Sigma) =
    \dist(a,\trans{q^{*}(z_s)}\lIm M_s \rIm)$.  This disposes at the same time
    of the distinction between $M_s$ and its closure ---
    \ref{lem:convex-reach}\ref{i:cr:reach} bounding the reach of $w_s\lIm \cball
    0{2s}\rIm$ whereas $M_s = w_s\lIm \oball 0{2s}\rIm$ --- the nearest points
    just located lying over the \emph{open} ball.
    Recalling~\ref{cor:np-in-2cyl} we may now use~\ref{lem:heights},
    \ref{cor:tilt-tilt}, and~\ref{rem:tilt-height} together with the triangle
    inequality to obtain
    \begin{multline}
        \psi(s)
        = \psi(v_s+z_s,p,0,s/2)
        \le \sqrt 2 \eta(s) + \beta(\trans{q^{*}(z_s)}\lIm M_s \rIm,T,0,s/2)
        \\
        \le (\sqrt{2} (4+\adim) + 1) \beta(\trans{q^{*}(z_s)}\lIm M_s \rIm,T,0,s)
        \le (\sqrt{2} (4+\adim) + 1) \zeta(s) 
        \quad \text{for $0 < s < 2^{-7}r_0$} \,.
    \end{multline}
    Consequently~\eqref{eq:final-phi} gives
    \begin{displaymath}
        \beta(s) \le \phi(s)
        \le (10+2\adim)^{-1}\kappa \xi(2^{10}s)
        \le \kappa \zeta(2^{10}s)
        \quad \text{whenever $0 < 2^{17}s < r_0$} \,;
    \end{displaymath}
    hence,
    \begin{displaymath}
        \zeta(s) \le \kappa \zeta(2^{10}s)
        \quad \text{for $0 < s < 2^{-17}r_0$}
        \qedhere
    \end{displaymath}
\end{proof}

\section{Rectifiability of every finite order}
\label{sec:rectifiability}

We combine the decay of~\ref{thm:decay} with the characterisation of higher
order rectifiability of sets due to Santilli~\cite{Santilli2019d}.  For an
$\HM^{\vdim}$~measurable set $A \subseteq \R^{\adim}$, an integer $k \ge 1$, and
$0 \le \alpha \le 1$, we~use without further comment the notion of
\emph{approximate differentiability of order~$(k,\alpha)$ of~$A$ at a point~$a$}
and its density characterisation; see \cite[3.8,~3.14]{Santilli2019d}.  We~shall
verify the following instance of that characterisation:
\begin{miniremark}
    \label{mr:santilli-crit}
    Let $A \subseteq \R^{\adim}$ be $\HM^{\vdim}$~measurable, $a \in A$, $T \in
    \grass{\adim}{\vdim}$ with $\ap\Tan(A,a) = T$, and $p \in
    \orthproj{\adim}{\vdim}$, $q \in \orthproj{\adim}{\adim-\vdim}$ with $\im p^*
    = T = \ker q$, and let $0 < \alpha \le 1$.  Suppose there are a polynomial
    map $P : \R^{\vdim} \to \R^{\adim-\vdim}$ of degree at most~$k$ with $P(0) =
    0$ and $\uD P(0) = 0$ and a number $0 \le \lambda < \infty$ such that, writing
    $N_P = \{ a + p^*(x) + q^*(P(x)) : x \in \R^{\vdim} \}$,
    \begin{equation}
        \label{eq:santilli-density}
        \lim_{s \downarrow 0} s^{-\vdim} \HM^{\vdim}\bigl(
        A \cap \cball as \cap \{ z : \dist(z,N_P) > \lambda s^{k+\alpha} \} \bigr) = 0 \,,
    \end{equation}
    and suppose in addition that
    \begin{equation}
        \label{eq:santilli-lower}
        \text{$A$ is approximately differentiable of order~$(1,0)$ at~$a$
          with $\ap\Tan(A,a) = T$} \,.
    \end{equation}
    Then $A$ is approximately differentiable of order~$(k,\alpha)$ at~$a$.
    Indeed, \eqref{eq:santilli-density} is the second condition
    of~\cite[3.8(2)]{Santilli2019d}, and, $\alpha$ being positive, it implies
    the first: $\lambda s^{k+\alpha} \le \varepsilon s^{k}$ for every
    $\varepsilon > 0$ once $s$ is small enough.  Condition~(1)
    of~\cite[3.8]{Santilli2019d} is a~lower density bound on the cylinders
    about the points of~$T$ which does not involve~$P$ at all, and is therefore
    supplied by~\eqref{eq:santilli-lower}; by~\cite[3.9]{Santilli2019d} the
    latter amounts to the existence of an~approximate tangent plane at~$a$.
\end{miniremark}

\begin{miniremark}[iterated decay]
    \label{mr:iterated-decay}
    Suppose $a$, $r$, $\{M_s\}$, and $\zeta$ are as furnished
    by~\ref{thm:decay}, so that $\zeta$ is nondecreasing and $\zeta(s) \le \kappa
    \zeta(2^{20}s)$ whenever $0 < 2^{20}s < r$.  Setting $\alpha_\kappa =
    \tfrac{1}{20}\log_2(1/\kappa)$, so that $\kappa = 2^{-20\alpha_\kappa}$,
    we~obtain by iteration
    \begin{equation}
        \label{eq:iterated-decay}
        \zeta(s) \le \kappa^{-1} (s/r)^{\alpha_\kappa} \zeta(r)
        \quad \text{for $0 < s < r$} \,;
    \end{equation}
    indeed, given such~$s$, let $\mathrm{N}$ be the greatest integer with
    $2^{20\mathrm{N}}s < r$, so that $\zeta(s) \le \kappa^{\mathrm{N}} \zeta(r)$,
    while $\kappa^{\mathrm{N}} = (2^{20\mathrm{N}})^{-\alpha_\kappa} \le
    \kappa^{-1}(s/r)^{\alpha_\kappa}$ because $2^{20(\mathrm{N}+1)} \ge r/s$.
    In particular \eqref{eq:iterated-decay} gives
    $\beta(M_s,\Tan^{\vdim}(\|V\|,a),a,s) \le \zeta(s) \le \Gamma
    s^{\alpha_\kappa}$ for $0 < s < r$, with $\Gamma = \kappa^{-1}
    r^{-\alpha_\kappa}\zeta(r)$; here $\zeta(r)$ abbreviates $\sup \{ \zeta(\rho) :
    0 < \rho < r \}$, which is finite because $\dist(z,M_\rho) \le
    \sqrt{2}\,\rho$ for $z \in \cylinder T{a}{\rho}{\rho}$ and $\|V\|(\cylinder
    T{a}{\rho}{\rho}) \le 2Q\unitmeasure{\vdim}(4\rho)^{\vdim}$
    by~\ref{lem:generic-props}\ref{i:gp:dens-ratio}.
\end{miniremark}

\begin{lemma}[Cauchy property of the jets of the comparison maps]
    \label{lem:jet-cauchy}
    Let $0 < \kappa < 1$, let $U$, $V$, $a$, $r$, $\{M_s\}$, and~$\zeta$ be as
    furnished by~\ref{thm:decay}, let $v_s$, $g$, $p$, $q$, $T$, $Q$, $r_0$,
    $\Lambda$, $\psi$, $\xi$ be as in the proof of~\ref{thm:decay} with $a = 0$,
    and let $J \in \natp$ satisfy $J \ge 2$.  There exist $0 < r_1 \le 2^{-17}r_0$ and $1
    \le \Gamma < \infty$, both determined by $\adim$, $\vdim$, $J$, $Q$, $G$,
    and~$V$, such that
    \begin{equation}
        \label{eq:jet-cauchy}
        \norm{\uD^{j}(v_s - v_{2s})}{\infty}{0,s/2} \le \Gamma \, s^{1-j} \zeta(2^{10}s)
        \quad \text{whenever $0 \le j \le J$ and $0 < s < r_1$} \,.
    \end{equation}
    If moreover $\alpha_{\kappa} \ge J$, where $\alpha_{\kappa} =
    \tfrac{1}{20}\log_2(1/\kappa)$ as in~\ref{mr:iterated-decay}, then
    \begin{equation}
        \label{eq:jet-bounded}
        \Gamma_{J} := \sup \bigl\{ \norm{\uD^{j}v_s}{\infty}{0,s/2} : 0 < s < r_1 ,\, 0 \le j \le J \bigr\} < \infty \,,
    \end{equation}
    the limits $\uD^{j}P(0) = \lim_{s \downarrow 0} \uD^{j}v_s(0)$ exist for $0
    \le j \le J$, and
    \begin{equation}
        \label{eq:jet-rate}
        | \uD^{j}v_s(0) - \uD^{j}P(0) | \le \Gamma \, s^{1 + \alpha_{\kappa} - j}
        \quad \text{whenever $0 \le j \le J-1$ and $0 < s < r_1$} \,.
    \end{equation}
\end{lemma}

\begin{proof}
    Throughout, $\Gamma_1, \Gamma_2, \ldots$ denote positive finite numbers
    determined by $\adim$, $\vdim$, $J$, $Q$, $G$, and~$V$.  Abbreviate $D_s =
    v_s - v_{2s}$, defined on $\oball 0{2^5s}$, and note $\zeta(\rho) \le
    \Gamma_1 \rho^{\alpha_\kappa}$ for $0 < \rho < r$
    by~\ref{mr:iterated-decay}.

    \emph{Step~1: bounds for $v_s$ and $v_{2s}$ on $\cball 0{2s}$.}
    Applying~\eqref{eq:u-hi-est} to $v_{\rho}$, for $\rho \in \{s,2s\}$, with
    $2^5\rho$ in place of~$r$ and with $J+2$ in place of~$k$ --- so that the
    ball on which it holds is $\cball 0{2^{-4}\cdot 2^5\rho} = \cball
    0{2\rho} \supseteq \cball 0{2s}$, \emph{independently of~$J$} ---
    and estimating
    $(2^5\rho)^{-\vdim/2}\norm{\uD g}{2}{0,2^5\rho} \le \Delta_8 \tau(2^7\rho)
    \le \Lambda$ exactly as in the derivation of~\eqref{eq:Lip-vs-est}, we~obtain
    \begin{equation}
        \label{eq:jc:vbounds}
        \textsum{i=1}{J+2} \bigl( s^{i-1}\norm{\uD^{i}v_{\rho}}{\infty}{0,2s}
        + s^{i-1/2}\hoelder{1/2}{\uD^{i}v_{\rho}|\cball 0{2s}} \bigr) \le \Gamma_1
        \quad \text{for $\rho \in \{s,2s\}$} \,.
    \end{equation}

    \emph{Step~2: $D_s$ solves a linear system.}  Define $B_s : \oball 0{2^5 s}
    \to \bigodot^2\Hom(\R^{\vdim},\R^{\adim-\vdim})$ by
    \begin{displaymath}
        B_s(x) = \textint{0}{1} \uD^2G \bigl( \uD v_{2s}(x) + t \, \uD D_s(x) \bigr) \ud t \,.
    \end{displaymath}
    Since $L_G(v_s) = L_G(v_{2s}) = 0$ and $\uD G(\uD v_s) - \uD G(\uD v_{2s}) =
    B_s \colon \uD D_s$ by the fundamental theorem of calculus, we~get
    \begin{equation}
        \label{eq:jc:linear}
        \textint{}{} \bigl\langle \uD\theta \odot \uD D_s ,\, B_s \bigr\rangle \ud \LM^{\vdim} = 0
        \quad \text{for $\theta \in \dspace{\oball 0{2^5s}}{\R^{\adim-\vdim}}$} \,,
    \end{equation}
    that is, $D_s$, $B_s$, $\Omega = 0$, $h = 0$ satisfy the base
    equation~\eqref{eq:ho:base-pde}.  By~\ref{lem:approx-function}, $\uD^2G$ is
    bounded below by $1-\varepsilon \ge \tfrac 12$ and above by $1 + \varepsilon
    \le 2$ in the sense of quadratic forms; whence, $\ellipticity{B_s(x)} \ge
    \tfrac12$ and $\|B_s(x)\| \le 2$ for $x \in \oball 0{2^5s}$, and $B_s$ takes
    its values in the class $W(\tfrac14,4)$ of~\ref{def:ho:W} --- these being the
    values of $c$ and~$M$ with which $\eta_{0,\ref{lem:schauder-est}}$ is formed
    for $\lambda = \tfrac12$.  Differentiating under the integral sign,
    each $\uD^i B_s$ is a universal polynomial expression in $\uD^{2}G, \ldots,
    \uD^{i+2}G$ evaluated along the segment joining $\uD v_{2s}$ to $\uD v_s$
    and in $\uD^{2}v_{\rho}, \ldots, \uD^{i+1}v_{\rho}$, $\rho \in \{s,2s\}$;
    hence, \eqref{eq:jc:vbounds} yields
    \begin{equation}
        \label{eq:jc:Bbounds}
        \textsum{i=0}{J} \bigl( s^{i} \norm{\uD^{i}B_s}{\infty}{0,2s}
        + s^{i+1/2}\hoelder{1/2}{\uD^{i}B_s|\cball 0{2s}} \bigr) \le \Gamma_2 \,.
    \end{equation}
    Moreover, $B_s$ satisfies the smallness
    hypothesis~\ref{lem:H}\ref{i:H:assume-epsilon} on $\cball 0{2s}$.  Indeed,
    writing $\Lambda_G^{(3)} = \sup\{\|\uD^3G(\sigma)\| : |\sigma| \le 1\}$ as
    in the proof of~\ref{lem:schauder-est}, the integrand defining~$B_s$ is
    evaluated along the segment joining $\uD v_{2s}$ to $\uD v_{s}$, so that
    \begin{displaymath}
        \hoelder{1/2}{B_s|\cball 0{2s}}
        \le \Lambda_G^{(3)} \bigl( \hoelder{1/2}{\uD v_s|\cball 0{2s}}
        + 2\hoelder{1/2}{\uD v_{2s}|\cball 0{2s}} \bigr) \,;
    \end{displaymath}
    on the other hand, for $\rho \in \{s,2s\}$ we~have $2s \le 2\rho$ and
    $2^5\rho \norm{\uD^2v_{\rho}}{\infty}{0,2\rho} \le \Lambda$
    by~\eqref{eq:Lip-vs-est}; whence,
    \begin{displaymath}
        (2s)^{1/2}\hoelder{1/2}{\uD v_{\rho}|\cball 0{2s}}
        \le (2s)^{1/2} (4s)^{1/2} \norm{\uD^2v_{\rho}}{\infty}{0,2s}
        \le 2^{3/2} s \cdot \Lambda 2^{-5}s^{-1} \le 2^{-3}\Lambda \,.
    \end{displaymath}
    Therefore, $(2s)^{1/2}\hoelder{1/2}{B_s|\cball 0{2s}} \le 3 \cdot
    2^{-3}\Lambda_G^{(3)}\Lambda \le \Lambda_G^{(3)}
    \eta_{0,\ref{lem:schauder-est}}(\adim,\vdim,\tfrac12,G) \le
    \varepsilon_{\ref{lem:H}}(\vdim,\adim,\tfrac12,\tfrac14,4)$, by the choice of~$\Lambda$
    in~\ref{thm:decay} and of $\eta_0$ in the proof
    of~\ref{lem:schauder-est}.

    \emph{Step~3: the $\Lp2$ estimate for $\uD D_s$.}  We~claim
    \begin{equation}
        \label{eq:diff-L2}
        s^{-\vdim/2}\norm{\uD D_s}{2}{0,2^5s} \le \Gamma_3 \, \zeta(2^{10}s)
        \quad \text{for $0 < s < r_1$} \,.
    \end{equation}
    Indeed, $v_s$ is the map furnished by~\ref{lem:approx-function} on $\oball
    0{2^5s}$ with datum~$g$, $v_{2s}$ is the map furnished on $\oball 0{2^6s}$
    with the same datum, and $L_G(v_{2^7s}) = 0$ on $\oball 0{2^{12}s}$, which
    contains both balls; hence,~\eqref{eq:Du-g}, applied twice with $v_{2^7s}$ in
    place of~$w$, gives
    \begin{displaymath}
        \norm{\uD(v_s-g)}{2}{0,2^5s} \le 3\norm{\uD(g-v_{2^7s})}{2}{0,2^5s} \,,
        \quad
        \norm{\uD(v_{2s}-g)}{2}{0,2^6s} \le 3\norm{\uD(g-v_{2^7s})}{2}{0,2^6s} \,,
    \end{displaymath}
    so that $\norm{\uD D_s}{2}{0,2^5s} \le 6 \norm{\uD(g-v_{2^7s})}{2}{0,2^6s}$.
    By~\eqref{eq:Dg-vs-squared} and~\eqref{eq:bad-set-psi},
    \begin{displaymath}
        s^{-\vdim}\norm{\uD(g-v_{2^7s})}{2}{0,2^6s}^2
        \le \Delta_6 \bigl( \psi(2^7s)^2 + \Delta_5 \psi(2^{10}s)^2 \bigr)
        \le \Delta_6(1+\Delta_5) \, \xi(2^{10}s)^2 \,,
    \end{displaymath}
    and $\xi(2^{10}s) \le (\sqrt2(4+\adim)+1)\zeta(2^{10}s)$ by the estimate
    for~$\psi$ established at the end of the proof of~\ref{thm:decay}.
    This proves~\eqref{eq:diff-L2}.

    \emph{Step~4: interior estimates.}  Apply the estimate~\eqref{eq:H:L2}
    of~\ref{lem:H} to the system~\eqref{eq:jc:linear}, with $b = 0$, $A = B_s$,
    $g = D_s$, $\Omega = 0$, and with inner and outer radii $s$
    and~$2s$, so that $\tau = s$;
    its hypotheses hold by Step~2.  Together with~\eqref{eq:diff-L2} this gives
    \begin{equation}
        \label{eq:jc:first}
        \norm{\uD D_s}{\infty}{0,s}
        + s^{1/2}\hoelder{1/2}{\uD D_s|\cball 0{s}}
        \le \Gamma_4 \, \zeta(2^{10}s) \,.
    \end{equation}
    Next we~apply~\ref{lem:ho:tower}, at order~$J$ and with $b = 0$, $\varrho =
    s$, $A = B_s$, $\Omega = 0$, $f = D_s$, and $r = s/2$.  Its
    hypotheses hold: the base equation~\eqref{eq:jc:linear} is homogeneous,
    $B_s$ is of class~$\cnt{J-1,1/2}$ on $\cball 0{s}$
    by~\eqref{eq:jc:Bbounds}, and the smallness
    hypothesis~\ref{lem:ho:tower}\ref{i:tower:A} was verified in Step~2.
    Since $\varrho = s$, \eqref{eq:jc:Bbounds} moreover bounds the
    quantity~$\mathcal{A}$ of~\ref{lem:ho:tower}\ref{i:tower:est} by~$\Gamma_2$,
    while $\mathcal{O} = 0$ and, by~\eqref{eq:jc:first},
    $\hnorm{1/2}{\uD D_s}{0,\varrho}{\varrho} \le \Gamma_4 \zeta(2^{10}s)$; as
    $\varrho/(\varrho-r) = 2$, \ref{lem:ho:tower}\ref{i:tower:est} therefore yields
    \begin{equation}
        \label{eq:diff-schauder}
        \textsum{j=1}{J} s^{j-1} \norm{\uD^{j}D_s}{\infty}{0,s/2}
        \le \Gamma_5 \, \zeta(2^{10}s) \,.
    \end{equation}

    \emph{Step~5: the supremum estimate.}  Since $v_s - g \in
    \VSobz{1,2}(\oball 0{2^5s},\R^{\adim-\vdim})$ and $v_{2s} - g \in
    \VSobz{1,2}(\oball 0{2^6s},\R^{\adim-\vdim})$
    by~\eqref{eq:LGu-and-u-g}, the Poincaré inequality $\norm{h}{2}{0,R} \le 2R
    \norm{\uD h}{2}{0,R}$, valid for $h \in \VSobz{1,2}(\oball
    0R,\R^{\adim-\vdim})$ by Fubini's theorem and the Cauchy--Schwarz
    inequality, together with the triangle inequality and~\eqref{eq:diff-L2}
    yields
    \begin{displaymath}
        \norm{D_s}{2}{0,2^5s}
        \le \norm{v_s-g}{2}{0,2^5s} + \norm{v_{2s}-g}{2}{0,2^6s}
        \le 2^{7}s \norm{\uD(g-v_{2^7s})}{2}{0,2^6s}
        \le \Gamma_6 \, s^{1+\vdim/2} \zeta(2^{10}s) \,.
    \end{displaymath}
    Comparing with the mean of $|D_s|^2$ over $\cball 0{s/2}$ we~find $x_1
    \in \cball 0{s/2}$ with $|D_s(x_1)| \le \Gamma_7 s \zeta(2^{10}s)$;
    since $\cball 0{s/2}$ is convex, \eqref{eq:diff-schauder} with $j = 1$
    gives, for $x \in \cball 0{s/2}$,
    \begin{equation}
        \label{eq:diff-sup}
        |D_s(x)| \le |D_s(x_1)| + |x - x_1| \norm{\uD D_s}{\infty}{0,s/2}
        \le \Gamma_8 \, s \, \zeta(2^{10}s) \,.
    \end{equation}
    Together, \eqref{eq:diff-schauder} and~\eqref{eq:diff-sup}
    are~\eqref{eq:jet-cauchy}.

    \emph{Step~6: consequences.}  Assume $\alpha_{\kappa} \ge J$ and let $0 \le
    j \le J$.  By~\eqref{eq:jet-cauchy} and~\ref{mr:iterated-decay},
    \begin{equation}
        \label{eq:jet-dyadic}
        \norm{\uD^{j}(v_s - v_{2s})}{\infty}{0,s/2} \le \Gamma_9 \, s^{1+\alpha_{\kappa}-j}
        \quad \text{for $0 < s < r_1$} \,,
    \end{equation}
    and $1 + \alpha_{\kappa} - j \ge 1 > 0$.  Since $\cball 0{s/2}
    \subseteq \cball 0{2^{m}s/2}$ for $m \ge 0$,
    iterating~\eqref{eq:jet-dyadic} from~$s$ up to the fixed scale $r_1/2$ and
    summing the resulting geometric series --- whose ratio
    $2^{1+\alpha_\kappa-j} > 1$ makes it comparable to its largest term ---
    yields
    \begin{displaymath}
        \norm{\uD^{j}v_s}{\infty}{0,s/2}
        \le \norm{\uD^{j}v_{r_1/2}}{\infty}{0,r_1/4} + \Gamma_9 \textsum{m=0}{\infty} (2^{-m}r_1)^{1+\alpha_\kappa-j}
        \le \Gamma_{10} \,,
    \end{displaymath}
    which is~\eqref{eq:jet-bounded}.  Finally, for $0 \le j \le J-1$ we~have $1
    + \alpha_{\kappa} - j \ge 2 > 0$, so~\eqref{eq:jet-dyadic} applied at the
    scales $2^{-m}s$, $m \ge 0$, shows that $(\uD^{j}v_{2^{-m}s}(0))_{m}$ is a
    Cauchy sequence and that
    \begin{displaymath}
        |\uD^{j}v_s(0) - \uD^{j}P(0)|
        \le \textsum{m=0}{\infty} \Gamma_9 (2^{-m}s)^{1+\alpha_{\kappa}-j}
        \le 2 \Gamma_9 \, s^{1+\alpha_{\kappa}-j} \,,
    \end{displaymath}
    the limit being independent of the dyadic sequence chosen because the same
    estimate applies to any pair of comparable scales.  This
    is~\eqref{eq:jet-rate}.
\end{proof}

\begin{lemma}[decay yields approximate differentiability]
    \label{lem:decay-approxdiff}
    Let $U$, $V$ be as in~\ref{thm:decay}, let $k \in \natp$ and $0 < \alpha <
    1$, and apply~\ref{thm:decay} with $\kappa = 2^{-20(k+2)}$, so that
    $\alpha_\kappa = k+2$.  Let $A = \spt \|V\| \cap \{ z :
    \density^{\vdim}(\|V\|,z) \in \natp \}$.  Then $A$ is approximately
    differentiable of order~$(k,\alpha)$ at $\|V\|$~almost every $a \in A$, with
    $\ap\Tan(A,a) = \Tan^{\vdim}(\|V\|,a)$.
\end{lemma}

\begin{proof}
    Fix $a$ in the full-measure set of~\ref{lem:generic-props} for which
    additionally the conclusion of~\ref{thm:decay} holds and
    $\lim_{\rho \downarrow 0} \tau(\Tan^{\vdim}(\|V\|,a),a,\rho) = 0$; the
    latter holds $\|V\|$~almost everywhere because $V$ is
    $(\HM^{\vdim},\vdim)$~rectifiable by~\cite[3.5]{Allard1972}, so that
    $\Tan^{\vdim}(\|V\|,a)$ is an approximate tangent plane in the sense
    of~\cite[11.6]{Simon1983}.  Set $Q = \density^{\vdim}(\|V\|,a)$, $T =
    \Tan^{\vdim}(\|V\|,a)$, let $p,q,v_s,M_s$ be as in the proof
    of~\ref{thm:decay}, and assume $a = 0$ by translation.  Since
    $\density^{\vdim}(\|V\|,\cdot) = Q$ $\HM^{\vdim}$~almost everywhere near~$a$
    by~\ref{lem:generic-props}\ref{i:gp:density=Q}, the measures $\|V\|$ and $Q
    \HM^{\vdim} \restrict A$ agree up to sets of $\HM^{\vdim}$~measure zero on a
    neighbourhood of~$a$; in particular $\ap\Tan(A,a) = T$
    by~\cite[3.2.19]{Federer1969} and~\cite[11.6]{Simon1983}.  Put $J = k+2$ and
    let $r_1$, $\Gamma_{J}$, and $\uD^{j}P(0)$ be as
    in~\ref{lem:jet-cauchy}; below $\Gamma_1,\Gamma_2,\ldots$ denote positive
    finite numbers determined by $\adim$, $\vdim$, $k$, $Q$, $G$, and~$V$.

    \emph{Step~1 (the osculating polynomial).}  Let $P : \R^{\vdim} \to
    \R^{\adim-\vdim}$ be the polynomial of degree at most~$k$ with $\uD^{j}P(0)$
    the limits of~\ref{lem:jet-cauchy}, $0 \le j \le k$.  We~check $P(0) = 0$
    and $\uD P(0) = 0$.  As $0 \in \spt\|V\|$ and $\density^{\vdim}(\|V\|,z) \ge
    1$ for $z \in \spt\|V\|$ --- the density being a~positive integer
    $\|V\|$~almost everywhere by~\cite[3.5(1)(a)(c)]{Allard1972} and upper
    semicontinuous by~\cite[8.6]{Allard1972} --- the monotonicity
    formula~\cite[5.1(2)]{Allard1972} gives $\measureball{\|V\|}{\cball 0\sigma} \ge
    \unitmeasure{\vdim}\sigma^{\vdim}$ for $0 < \sigma < s$; since $\cball 0s
    \subseteq \cylinder T0ss \cap \Unp(M_s)$ --- the latter because
    $\reach(M_s) \ge (2^{13}-1)s$, as established in the proof
    of~\ref{thm:decay}, and since the same applies at every scale,
    \begin{equation}
        \label{eq:height-over-Ms}
        \textint{\cball 0\varsigma}{} \dist(z,M_{\varsigma})^2 \ud \|V\|(z)
        \le \varsigma^{\vdim+2}\beta(M_{\varsigma},T,0,\varsigma)^2
        \le \varsigma^{\vdim+2}\zeta(\varsigma)^2
        \quad \text{for $0 < \varsigma < r_1$} \,;
    \end{equation}
    Chebyshev's inequality applied to~\eqref{eq:height-over-Ms} with $\varsigma
    = s$
    furnishes, for each $0 < t \le s$, a point $z \in \spt\|V\| \cap \cball
    0t$ with $\dist(z,M_s) \le (s/t)^{\vdim/2}
    \unitmeasure{\vdim}^{-1/2} s \zeta(s)$; hence, $\dist(0,M_s) \le t +
    (s/t)^{\vdim/2}\unitmeasure{\vdim}^{-1/2}s\zeta(s)$, and choosing
    $t = s\zeta(s)^{2/(\vdim+2)}$ --- which is at most~$s$ once $s$ is small
    enough that $\zeta(s) \le 1$, and which renders the two summands equal up
    to the factor $\unitmeasure{\vdim}^{-1/2}$ --- gives $\dist(0,M_s) \le \Gamma_1 s
    \zeta(s)^{2/(\vdim+2)}$.  As $\Lip v_s|\cball 0{2s} \le \Lambda$, this
    yields $|v_s(0)| \le \Gamma_2 s \zeta(s)^{2/(\vdim+2)} \to 0$, so $P(0) =
    0$.  For $\uD P(0)$, note that by~\ref{rem:tilt-height} applied with $M_s$
    the quantity $\eta(M_s,T,0,s)$ is at most $\Gamma_3\zeta(2s)$; whence, by the
    triangle inequality and~\ref{lem:generic-props}\ref{i:gp:dens-ratio}
    \begin{displaymath}
        \inf \bigl\{ \| \project{\Tan(M_s,\npp{M_s}(z))} - \project T \| : z \in \spt\|V\| \cap \cylinder T0ss \bigr\}
        \le (Q\unitmeasure{\vdim})^{-1/2} \bigl( \tau(2^7s) + \Gamma_3 \zeta(2s) \bigr) \,,
    \end{displaymath}
    while $\| \project{\Tan(M_s,w_s(x))} - \project{\Tan(M_s,w_s(y))} \| \le
    \sqrt 2 \, 2^{5}s \norm{\uD^2v_s}{\infty}{0,2s} \le \sqrt2 \Delta_8
    \tau(2^7s)$ for $x,y \in \cball 0{2s}$ by~\eqref{eq:Lip-vs-est}.  Since
    $\tau(2^7 s) \to 0$ and $\zeta(2s) \to 0$ as $s \downarrow 0$, it follows
    that $\|\uD v_s(0)\| \to 0$, so $\uD P(0) = 0$.

    \emph{Step~2 (uniform Taylor remainder).}  Put $\sigma = 2s$, so that
    the ball on which~\ref{lem:jet-cauchy} provides its estimates for~$v_\sigma$
    is $\cball 0{\sigma/2} = \cball 0s$, and let $P_\sigma$ be the Taylor
    polynomial of~$v_\sigma$ at~$0$ of degree~$k$.  Since $\alpha_\kappa = k+2 =
    J$, \eqref{eq:jet-bounded} applies and gives
    $\norm{\uD^{k+1}v_\sigma}{\infty}{0,s} \le \Gamma_{J}$ whenever $0 < \sigma <
    r_1$; Taylor's theorem therefore yields $\sup\{|v_\sigma(x) - P_\sigma(x)| :
    |x| \le s\} \le \Gamma_{J} s^{k+1}/(k+1)!$.  Meanwhile~\eqref{eq:jet-rate}
    --- available for $0 \le j \le k \le J-1$ --- gives, as $\sigma = 2s$,
    \begin{displaymath}
        \sup \bigl\{ |P_\sigma(x) - P(x)| : |x| \le s \bigr\}
        \le \textsum{j=0}{k} \frac{s^{j}}{j!} \, \Gamma \, \sigma^{1+\alpha_\kappa-j}
        \le \Gamma_4 \, s^{k+3} \,.
    \end{displaymath}
    Consequently
    \begin{equation}
        \label{eq:vs-minus-P}
        \sup \bigl\{ |v_\sigma(x) - P(x)| : |x| \le s \bigr\} \le \Gamma_5 \, s^{k+1}
        \quad \text{for $0 < \sigma < r_1$} \,.
    \end{equation}

    \emph{Step~3 (Chebyshev).}  Write $N_P = \{ p^*(x) + q^*(P(x)) : x \in
    \R^{\vdim} \}$ and fix $\lambda > 0$; keep $\sigma = 2s$.  Let $z \in
    \spt\|V\| \cap \cball 0s$.  Then $|p(z)| \le s < 2\sigma$, so $p(z)$ belongs
    to the base of the graph $M_\sigma = w_\sigma \lIm \oball 0{2\sigma}\rIm$,
    and the point $p^*(p(z)) + q^*(P(p(z)))$ belongs to~$N_P$; whence,
    $\dist(z,N_P) \le |q(z) - P(p(z))|$.  Writing $(x,v_\sigma(x))$ for a point
    of~$M_\sigma$ nearest to~$z$ and using $\Lip v_\sigma|\cball 0{2\sigma} \le
    \Lambda \le 1$ we~get $|p(z) - x| \le \dist(z,M_\sigma)$ and $|q(z) -
    v_\sigma(x)| \le \dist(z,M_\sigma)$; hence,
    \begin{displaymath}
        |q(z) - v_\sigma(p(z))|
        \le |q(z) - v_\sigma(x)| + \Lambda |x - p(z)|
        \le 2 \dist(z,M_\sigma) \,;
    \end{displaymath}
    combining this with~\eqref{eq:vs-minus-P} yields
    \begin{displaymath}
        \dist(z,N_P) \le 2 \dist(z,M_\sigma) + \Gamma_5 s^{k+1} \,.
    \end{displaymath}
    Since $k + 1 > k + \alpha$, there is $0 < r_2 \le r_1/2$ with $\Gamma_5
    s^{k+1} \le \tfrac 12 \lambda s^{k+\alpha}$ for $0 < s < r_2$; for such~$s$,
    \begin{displaymath}
        \cball 0s \cap \{ z : \dist(z,N_P) > \lambda s^{k+\alpha} \}
        \subseteq \cball 0s \cap \{ z : \dist(z,M_\sigma) > \tfrac 14 \lambda s^{k+\alpha} \} \,.
    \end{displaymath}
    As $s \le \sigma$, \eqref{eq:height-over-Ms} applied with $\varsigma =
    \sigma$ bounds $\textint{\cball 0s}{}\dist(z,M_\sigma)^2 \ud\|V\|$ by
    $\sigma^{\vdim+2}\zeta(\sigma)^2$, so Chebyshev's inequality
    and~\ref{mr:iterated-decay} give
    \begin{multline}
        s^{-\vdim} \|V\| \bigl( \cball 0s \cap \{ z : \dist(z,N_P) > \lambda s^{k+\alpha} \} \bigr)
        \le 16\lambda^{-2} s^{-\vdim-2(k+\alpha)} \sigma^{\vdim+2} \zeta(\sigma)^2
        \\
        \le \Gamma_6 \lambda^{-2} s^{2 - 2(k+\alpha) + 2\alpha_\kappa}
        = \Gamma_6 \lambda^{-2} s^{6-2\alpha}
        \xrightarrow{\ s \downarrow 0\ } 0 \,,
    \end{multline}
    because $\sigma = 2s$, $\alpha_\kappa = k+2$ and $\alpha < 1$.  As $\|V\| = Q \HM^{\vdim}
    \restrict A$ near~$a$, the same limit holds with $\HM^{\vdim} \restrict A$
    in place of~$\|V\|$, which is~\eqref{eq:santilli-density}.  Moreover
    \eqref{eq:santilli-lower} holds: it was established at the beginning of
    this proof that $\ap\Tan(A,a) = T$, the measure $\|V\|$ having
    an~approximate tangent plane at~$a$ by~\cite[3.5]{Allard1972}
    and~\cite[11.6]{Simon1983} and agreeing near~$a$ with $Q
    \HM^{\vdim}\restrict A$, $Q \ge 1$.
    By~\ref{mr:santilli-crit}, $A$ is approximately differentiable of
    order~$(k,\alpha)$ at~$a$.  As $a$ ranged over a set of full $\|V\|$~measure,
    the proof is complete.
\end{proof}

\begin{theorem}[rectifiability of every finite order]
    \label{thm:finite-rectifiability}
    Let $U \subseteq \R^{\adim}$ be open and $V \in \IVar{\vdim}(U)$ satisfy
    $\delta V = 0$.  Then $\spt\|V\|$ is $(\HM^{\vdim},\vdim)$~rectifiable of
    class~$(k,\alpha)$ for every $k \in \natp$ and $0 < \alpha < 1$; in
    particular $\|V\|$ is carried by a countable union of $\vdim$~dimensional
    submanifolds of~$\R^{\adim}$ of class~$\cnt{k}$ for every $k \in \natp$.
\end{theorem}

\begin{proof}
    We may assume $\vdim \ge 2$: if $\vdim = 0$ then $\spt\|V\|$ is countable,
    and if $\vdim = 1$ then, by the structure theorem for stationary
    one-dimensional integral varifolds~\cite{Allard1976}, $\HM^1$~almost all of
    $\spt\|V\|$ consists of relatively open subarcs of straight lines; in either
    case $\spt\|V\|$ is $(\HM^{\vdim},\vdim)$~rectifiable of class~$\cnt{\infty}$;
    hence, of every class~$(k,\alpha)$.

    Now suppose $\vdim \ge 2$, fix $k$ and~$\alpha$, and let $A$ be as
    in~\ref{lem:decay-approxdiff}.  Then $A$ is $\HM^{\vdim}$~measurable with
    $\HM^{\vdim}(A \cap K) < \infty$ for
    every compact $K \subseteq U$ (as $\|V\|$ is a Radon measure and $\|V\| = Q
    \HM^{\vdim}\restrict A$ locally with $Q \ge 1$), and by
    \ref{lem:decay-approxdiff} $A$ is approximately differentiable of
    order~$(k,\alpha)$ with $\dim\ap\Tan(A,a) = \vdim$ at $\HM^{\vdim}$~almost
    every $a \in A$.  Exhausting~$U$ by compact sets and
    applying~\cite[5.6]{Santilli2019d} (equivalently~\cite[1.2]{Santilli2019d})
    on each, we~conclude that $A$ is $(\HM^{\vdim},\vdim)$~rectifiable of
    class~$(k,\alpha)$.  Since $\density^{*\vdim}(\|V\|,z) \ge 1$ for $z \in
    \spt\|V\|$ by~\cite[3.5(1)(a)(c)]{Allard1972}, the monotonicity
    formula~\cite[5.1(2)]{Allard1972} and the upper semicontinuity of the
    density~\cite[8.6]{Allard1972}, the
    comparison theorem~\cite[2.10.19(1)]{Federer1969} gives
    $\HM^{\vdim}(\spt\|V\| \without A) \le \|V\|(U \without A) = 0$; hence,
    $\spt\|V\|$ too is $(\HM^{\vdim},\vdim)$~rectifiable of
    class~$(k,\alpha)$.  Since $\|V\| = Q \HM^{\vdim}\restrict A$ and
    $\|V\|(U \without A) = 0$, the varifold is carried by~$A$; whence, the final
    assertion.
\end{proof}

Finally, we record that the finite order conclusion of~\ref{thm:finite-rectifiability}
already yields rectifiability of class~$\cnt{\infty}$.  The mechanism is
Whitney's $\cnt{\infty}$~extension theorem, and the key point is that the
approximate differentiability of every order holds on a \emph{common}
set of full measure.

\begin{lemma}
    \label{lem:Ck-to-Cinfty}
    Let $f : \R^{\vdim} \to \R^{\adim-\vdim}$ be $\LM^{\vdim}$~measurable and
    let $A \subseteq \R^{\vdim}$ be $\LM^{\vdim}$~measurable and such that $f$
    is approximately differentiable of order~$k$ at every $a \in A$, in the
    sense of~\cite[2.3]{Santilli2019d}, for every $k \in \natp$.  Then there
    exist countably many maps $g_j : \R^{\vdim} \to \R^{\adim-\vdim}$ of
    class~$\cnt{\infty}$ such that
    \begin{displaymath}
        \LM^{\vdim}\bigl( A \without \tbcup \{ \{ x : g_j(x) = f(x) \} : j \in \natp \} \bigr) = 0 \,.
    \end{displaymath}
\end{lemma}

\begin{proof}
    Since the assertion for~$A$ follows from the assertion for each of the sets
    $A \cap \oball 0l$, $l \in \natp$, on taking the union of the resulting
    countable families, we~may assume $A \subseteq \oball 0l$ for some $l \in
    \natp$.  It then suffices to produce, for each $\varepsilon > 0$, a compact
    set $C \subseteq A$ with $\LM^{\vdim}(A \without C) < \varepsilon$ and a map
    $g \in \cnt{\infty}(\R^{\vdim},\R^{\adim-\vdim})$ with $g|C = f|C$; taking
    $\varepsilon = 2^{-l'}$ for $l' \in \natp$ and enumerating the resulting maps
    as $\{g_j\}$ then gives the assertion, because $\LM^{\vdim}(A \without
    \tbcup \{ C_{l'} : l' \in \natp \}) = 0$.

    Let $\hat f : \R^{\vdim} \to \R^{\adim-\vdim}$ satisfy $\hat f|A = f|A$ and
    $\hat f(x) = 0$ for $x \in \R^{\vdim} \without A$.  Then $\hat f$ is
    approximately differentiable of order~$k$, for every $k \in \natp$, at
    $\LM^{\vdim}$~almost every point of~$\R^{\vdim}$: at every $a \in A$ of
    density~$1$ of~$A$, because $\hat f$ differs from~$f$ only on a~set of
    density~$0$ there, and at every point of density~$1$ of $\R^{\vdim} \without
    A$, because $\hat f$ vanishes on a~set of density~$1$ there; the remaining
    points form an $\LM^{\vdim}$~null set by the Lebesgue density
    theorem~\cite[2.9.11]{Federer1969}.

    Fix $\varepsilon > 0$.  Let $k \in \natp$.  By Isakov's
    theorem~\cite{Isakov1991}, applied to~$\hat f$, the ball $\oball 0l$, and
    the number $2^{-k-1}\varepsilon$ --- its hypothesis being the approximate
    differentiability of order~$k$ just verified, and its conclusion the
    \emph{$\cnt{k}$~property} of~$\hat f$ --- there are a~compact set $K_k
    \subseteq \oball 0l$ and a~map $g_k \in
    \cnt{k}(\R^{\vdim},\R^{\adim-\vdim})$ with
    \begin{displaymath}
        \LM^{\vdim}(\oball 0l) - \LM^{\vdim}(K_k) < 2^{-k-1}\varepsilon
        \quad \text{and} \quad
        \hat f|K_k = g_k|K_k \,;
    \end{displaymath}
    in particular $\LM^{\vdim}(A \without K_k) < 2^{-k-1}\varepsilon$.  Employing
    \cite[2.9.11]{Federer1969} and the inner regularity
    of~$\LM^{\vdim}$~\cite[2.2.2]{Federer1969}, choose a~compact set $B_k
    \subseteq A \cap K_k$ every point of which is a~point of density~$1$ of $A
    \cap K_k$ and $\LM^{\vdim}(A \without B_k) < 2^{-k}\varepsilon$.  We~assert
    \begin{displaymath}
        \uD^i g_k(b) = \ap\uD^i f(b) \quad \text{for $b \in B_k$ and $i = 0,1,\ldots,k$} \,.
    \end{displaymath}
    Indeed, $f = \hat f = g_k$ on $A \cap K_k$, a~set of density~$1$ at~$b$; so,
    denoting by~$P$ the polynomial of degree at most~$k$ attached to~$f$
    at~$b$ by~\cite[2.3]{Santilli2019d} and by~$Q$ the $k$~jet of~$g_k$ at~$b$,
    both $|f - P|$ and $|g_k - Q| = |f - Q|$ are $o(|\cdot - b|^{k})$ in the
    approximate sense at~$b$; hence, so is $|P-Q|$, and a~nonzero polynomial of
    degree at most~$k$ vanishing at~$b$ to order less than~$k+1$ exceeds
    $\varepsilon' |\cdot - b|^{k}$ on a~set of positive upper density at~$b$ for
    some $\varepsilon' > 0$.  Thus $P = Q$, which is the assertion.

    Put $C = \tbcap \{ B_k : k \in \natp \}$.  Then $C$ is compact, $C \subseteq
    A$, and
    \begin{displaymath}
        \LM^{\vdim}(A \without C) \le \textsum{k=1}{\infty} \LM^{\vdim}(A \without B_k)
        < \textsum{k=1}{\infty} 2^{-k}\varepsilon = \varepsilon \,.
    \end{displaymath}
    For every $k$ the field $b \mapsto (\ap\uD^i f(b))_{i=0}^{k}$ coincides
    on~$C$ with the $\cnt{k}$~jet of~$g_k$; hence, satisfies Whitney's
    compatibility conditions of order~$k$ there; as $k$ is arbitrary, $(\ap\uD^i
    f)_{i \ge 0}$ is a Whitney field of class~$\cnt{\infty}$ on the compact
    set~$C$.  By Whitney's $\cnt{\infty}$~extension
    theorem~\cite[Theorem~I]{Whitney1934} (cf.~\cite[3.1.14]{Federer1969}),
    which covers the infinite order, there is $g \in
    \cnt{\infty}(\R^{\vdim},\R^{\adim-\vdim})$ with $\uD^i g(b) = \ap\uD^i f(b)$
    for $b \in C$ and $i \ge 0$; in particular $g|C = f|C$.
\end{proof}

\begin{lemma}[from sets to functions]
    \label{lem:set-to-function}
    Suppose $S \in \grass{\adim}{\vdim}$, $p \in \orthproj{\adim}{\vdim}$, $q
    \in \orthproj{\adim}{\adim-\vdim}$, $\im p^* = S = \ker q$, $X \subseteq
    \R^{\vdim}$ is $\LM^{\vdim}$~measurable, $f : \R^{\vdim} \to
    \R^{\adim-\vdim}$ is $\LM^{\vdim}$~measurable, $\Lip (f|X) \le L < \infty$,
    \begin{equation}
        \label{eq:stf:graph}
        E = \{ p^*(y) + q^*(f(y)) : y \in X \} \,,
    \end{equation}
    $A \subseteq \R^{\adim}$ is $\HM^{\vdim}$~measurable, $E \subseteq A$, $k
    \in \natp$, $0 < \alpha \le 1$, $x \in X$, and $a = p^*(x) + q^*(f(x))$.
    Assume
    \begin{enumerate}
    \item \label{i:stf:apdiff}
        $A$ is approximately differentiable of order~$(k,\alpha)$ at~$a$ in the
        sense of~\cite[3.8]{Santilli2019d}, and $T = \ap\Tan(A,a)$ satisfies
        $T \in \grass{\adim}{\vdim}$ and $T \cap \ker p = \{0\}$;
    \item \label{i:stf:density}
        $\density^{\vdim}(\LM^{\vdim} \restrict \R^{\vdim} \without X, x) = 0$.
    \end{enumerate}
    Then $f$ is approximately differentiable of order~$k$ at~$x$ in the sense
    of~\cite[2.3]{Santilli2019d}.
\end{lemma}

\begin{proof}
    Let $P : T \to T^{\perp}$ be the polynomial map of degree at most~$k$ with
    $P(0) = 0$ and $\uD P(0) = 0$ furnished by~\ref{i:stf:apdiff} through
    \cite[3.8]{Santilli2019d}, and put $N = \{ a + \chi + P(\chi) : \chi \in T
    \}$.  Thus $N$~is a~closed $\vdim$~dimensional submanifold of~$\R^{\adim}$
    of class~$\cnt{\infty}$ with $a \in N$ and $\Tan(N,a) = T$, and the first
    limit in~\cite[3.8(2)]{Santilli2019d} reads
    \begin{equation}
        \label{eq:stf:hyp}
        \lim_{s \downarrow 0} s^{-\vdim} \HM^{\vdim} \bigl( A \cap \cball as
        \cap \{ z : \dist(z,N) > \varepsilon s^{k} \} \bigr) = 0
        \quad \text{for every $\varepsilon > 0$} \,.
    \end{equation}

    \emph{Step 1 ($N$ is a graph over~$S$ near~$a$).}  Since $\Tan(N,a) = T$ and
    $T \cap \ker p = \{0\}$, the map $p|N$ is a~$\cnt{\infty}$ diffeomorphism of
    a~neighbourhood of~$a$ in~$N$ onto a~neighbourhood of~$x$ in~$\R^{\vdim}$;
    hence, there are $0 < \rho < \infty$, $0 \le \Lambda < \infty$, and $h \in
    \cnt{\infty}(\R^{\vdim},\R^{\adim-\vdim})$ such that
    \begin{equation}
        \label{eq:stf:h}
        N \cap \oball a{2\rho} \subseteq \{ p^*(y) + q^*(h(y)) : y \in
        \R^{\vdim} \}
        \quad \text{and} \quad
        \Lip h \le \Lambda \,.
    \end{equation}
    As $a \in N \cap \oball a{2\rho}$ and $a = p^*(x) + q^*(f(x))$, the
    injectivity of $y \mapsto p^*(y) + q^*(h(y))$ yields $h(x) = f(x)$.  Let $Q$
    be the $k$~jet of~$h$ at~$x$, that is $Q(y) = \textsum{i=0}{k} \langle
    (y-x)^i/i! , \uD^i h(x) \rangle$ for $y \in \R^{\vdim}$; then $Q(x) = h(x) =
    f(x)$ and, by Taylor's theorem~\cite[3.1.11]{Federer1969}, there are $0 <
    \rho_1 \le \rho$ and $0 \le \Gamma_1 < \infty$ with
    \begin{equation}
        \label{eq:stf:taylor}
        |h(y) - Q(y)| \le \Gamma_1 |y-x|^{k+1}
        \quad \text{for $y \in \cball x{\rho_1}$} \,.
    \end{equation}

    \emph{Step 2 (the distance to~$N$ controls the vertical deviation).}
    We~assert that
    \begin{equation}
        \label{eq:stf:vertical}
        |f(y) - h(y)| \le (1+\Lambda) \dist \bigl( p^*(y) + q^*(f(y)) , N \bigr)
    \end{equation}
    whenever $y \in X$, $z = p^*(y) + q^*(f(y)) \in \oball a{\rho}$, and
    $\dist(z,N) < \rho$.  In~fact, $N$ being closed, we~may pick $w \in N$ with
    $|z-w| = \dist(z,N)$; then $|w-a| \le |w-z| + |z-a| < 2\rho$, so $w =
    p^*(\chi) + q^*(h(\chi))$ for some $\chi \in \R^{\vdim}$
    by~\eqref{eq:stf:h}, whence, noting $q(z-w) = f(y) - h(\chi)$ and $p(z-w) =
    y - \chi$,
    \begin{equation*}
        |f(y) - h(y)| \le |f(y) - h(\chi)| + |h(\chi) - h(y)|
        \le |q(z-w)| + \Lambda |p(z-w)| \le (1+\Lambda) |z-w| \,.
    \end{equation*}

    \emph{Step 3 (transfer of the density estimate).}  Abbreviate $C =
    (1+L^2)^{1/2}$ and $\Xi = (1+\Lambda) C^{k}$, and note that $y \mapsto
    p^*(y) + q^*(f(y))$ of~\eqref{eq:stf:graph} maps $X$ onto~$E$, is injective, has Lipschitz constant
    at most~$C$, and its inverse $p|E$ has Lipschitz constant at most~$1$;
    consequently, by~\cite[2.10.11]{Federer1969},
    \begin{equation}
        \label{eq:stf:compare}
        \LM^{\vdim}(Y) \le \HM^{\vdim} \bigl( \{ p^*(y) + q^*(f(y)) : y \in Y \}
        \bigr)
        \quad \text{for $\LM^{\vdim}$~measurable $Y \subseteq X$} \,.
    \end{equation}
    Let $\varepsilon > 0$ and put, for $0 < r < \infty$,
    \begin{equation*}
        Y_r = X \cap \cball xr \cap \{ y : |f(y) - h(y)| > \Xi \varepsilon r^{k}
        \} \,.
    \end{equation*}
    Suppose $0 < r < \infty$ satisfies $Cr < \rho$ and $\varepsilon (Cr)^{k} <
    \rho$.  If $y \in Y_r$ and $z = p^*(y) + q^*(f(y))$, then $|z-a| \le C|y-x|
    \le Cr$ and $\dist(z,N) > \varepsilon (Cr)^{k}$; indeed, this is clear if
    $\dist(z,N) \ge \rho$, whereas if $\dist(z,N) < \rho$
    then~\eqref{eq:stf:vertical} applies and gives $(1+\Lambda) \dist(z,N) \ge
    |f(y)-h(y)| > \Xi \varepsilon r^{k} = (1+\Lambda) \varepsilon (Cr)^{k}$.
    Therefore,
    by~\eqref{eq:stf:compare} and~\eqref{eq:stf:hyp},
    \begin{equation}
        \label{eq:stf:Yr}
        \limsup_{r \downarrow 0} r^{-\vdim} \LM^{\vdim}(Y_r)
        \le C^{\vdim} \lim_{s \downarrow 0} s^{-\vdim} \HM^{\vdim} \bigl( A \cap
        \cball as \cap \{ z : \dist(z,N) > \varepsilon s^{k} \} \bigr) = 0 \,.
    \end{equation}

    \emph{Step 4 (conclusion).}  Let $\varepsilon > 0$ and abbreviate $\Xi_1 =
    \Xi + 1$.  If $0 < r \le \rho_1$, $Cr < \rho$, and $\Gamma_1 r \le
    \varepsilon$, then \eqref{eq:stf:taylor} gives $|h(y) - Q(y)| \le
    \varepsilon r^{k}$ for $y \in \cball xr$; hence,
    \begin{equation*}
        \cball xr \cap \{ y : |f(y) - Q(y)| > \Xi_1 \varepsilon r^{k} \}
        \subseteq Y_r \cup ( \cball xr \without X ) \,.
    \end{equation*}
    In~view of~\eqref{eq:stf:Yr} and~\ref{i:stf:density}, and since $\varepsilon
    > 0$ was arbitrary, we~infer
    \begin{equation}
        \label{eq:stf:scale}
        \lim_{r \downarrow 0} r^{-\vdim} \LM^{\vdim} \bigl( \cball xr \cap \{ y
        : |f(y) - Q(y)| > \varepsilon r^{k} \} \bigr) = 0
        \quad \text{for every $\varepsilon > 0$} \,.
    \end{equation}
    Finally, we~pass from~\eqref{eq:stf:scale} to the approximate limit.  Let
    $\varepsilon > 0$, let $\theta(s)$ denote the quantity whose limit is taken
    in~\eqref{eq:stf:scale} with $\varepsilon$ replaced by $2^{-k}\varepsilon$,
    and note that $\theta(s) \le \unitmeasure{\vdim}$ for $0 < s < \infty$ and
    $\lim_{s \downarrow 0} \theta(s) = 0$.  Writing $B = \{ y : |f(y) - Q(y)| >
    \varepsilon |y-x|^{k} \}$ and $s_j = 2^{-j}r$ for $j \in \nat$, we~have
    \begin{equation*}
        B \cap \cball x{s_j} \without \oball x{s_{j+1}}
        \subseteq \cball x{s_j} \cap \{ y : |f(y) - Q(y)| > 2^{-k} \varepsilon
        s_j^{k} \}
    \end{equation*}
    because $|y-x| \ge s_{j+1} = \tfrac 12 s_j$ for such~$y$; consequently,
    \begin{equation*}
        r^{-\vdim} \LM^{\vdim}(B \cap \cball xr)
        \le \textsum{j=0}{\infty} 2^{-j\vdim} \theta(2^{-j}r) \,,
    \end{equation*}
    and the right hand side tends to~$0$ as $r \downarrow 0$ by dominated
    convergence, a~summable majorant being $j \mapsto 2^{-j\vdim}
    \unitmeasure{\vdim}$.  Thus $\density^{\vdim}(\LM^{\vdim} \restrict B, x) =
    0$ for every $\varepsilon > 0$, that is $\ap\lim_{y \to x} |f(y)-Q(y)| \,
    |y-x|^{-k} = 0$.  Together with $Q(x) = f(x)$ and~\ref{i:stf:density}, this
    is precisely the assertion of~\cite[2.3]{Santilli2019d} for the order~$k$
    and the polynomial~$Q$.
\end{proof}

\begin{corollary}
    \label{cor:Cinfty-rectifiability}
    Let $U \subseteq \R^{\adim}$ be open and $V \in \IVar{\vdim}(U)$ satisfy
    $\delta V = 0$.  Then $\spt\|V\|$ is $(\HM^{\vdim},\vdim)$~rectifiable of
    class~$\cnt{\infty}$; that is, $\|V\|$ is carried by a countable union of
    $\vdim$~dimensional submanifolds of~$\R^{\adim}$ of class~$\cnt{\infty}$.
\end{corollary}

\begin{proof}
    We~may assume $\vdim \ge 2$, the cases $\vdim \le 1$ being elementary (see
    the proof of~\ref{thm:finite-rectifiability}).  Let $A = \spt\|V\| \cap \{ z
    : \density^{\vdim}(\|V\|,z) \in \natp \}$; as recorded in the proof
    of~\ref{thm:finite-rectifiability}, the set $A$ is $\HM^{\vdim}$~measurable,
    $\HM^{\vdim} \restrict A$ is a~Radon measure over~$U$, $\|V\| =
    \density^{\vdim}(\|V\|,\cdot) \, \HM^{\vdim} \restrict A$ with integer
    density, and $\HM^{\vdim}(\spt\|V\| \without A) = \|V\|(U \without A) = 0$.

    \emph{Step 1 (a common exceptional set).}  For $k \in \natp$ let $Z_k$ be
    the $\|V\|$~null set outside of which~\ref{lem:decay-approxdiff}, applied
    with $\kappa = 2^{-20(k+2)}$ and $\alpha = \frac 12$, asserts approximate
    differentiability of~$A$ of order~$(k,\frac 12)$ with $\ap\Tan(A,\cdot) =
    \Tan^{\vdim}(\|V\|,\cdot)$.  Then $Z = \tbcup \{ Z_k : k \in \natp \}$ is
    $\|V\|$~null, and at every $a \in A^{*} = A \without Z$ the set~$A$ is
    approximately differentiable of order~$(k,\frac 12)$ for \emph{every} $k \in
    \natp$ simultaneously, with one and the same $\ap\Tan(A,a) \in
    \grass{\adim}{\vdim}$; the $\kappa$~dependence of the individual sets~$Z_k$
    is harmless because $\natp$ is countable.  Securing this \emph{common}
    exceptional set is the point at which the passage from finite to infinite
    order is decided, and it is what~\ref{lem:Ck-to-Cinfty} requires.

    \emph{Step 2 (a disjoint decomposition into graphs).}
    By~\ref{thm:finite-rectifiability} applied with $k = 1$, there are countably
    many $\vdim$~dimensional submanifolds of~$\R^{\adim}$ of class~$\cnt{1}$
    covering $\HM^{\vdim}$~almost all of~$A$.  Covering each of them by
    countably many relatively open pieces and shrinking these, we~obtain $S_j
    \in \grass{\adim}{\vdim}$, $p_j \in \orthproj{\adim}{\vdim}$, $q_j \in
    \orthproj{\adim}{\adim-\vdim}$ with $\im p_j^{*} = S_j = \ker q_j$, open
    sets $V_j \subseteq \R^{\vdim}$, and maps $g_j : \R^{\vdim} \to
    \R^{\adim-\vdim}$ with $\Lip g_j \le 1$ and $g_j|V_j$ of class~$\cnt{1}$
    --- the local graph maps being extended to~$\R^{\vdim}$ without increase of
    their Lipschitz constants by Kirszbraun's
    theorem~\cite[2.10.43]{Federer1969} --- such that, setting
    \begin{equation*}
        M_j = \{ p_j^{*}(y) + q_j^{*}(g_j(y)) : y \in V_j \} \,,
    \end{equation*}
    the $M_j$ are $\vdim$~dimensional submanifolds of class~$\cnt{1}$ with
    $\HM^{\vdim} ( A \without \tbcup \{ M_j : j \in \natp \} ) = 0$.  Put
    \begin{equation*}
        E_j = A^{*} \cap M_j \without \tbcup \{ M_i : i \in \natp , \, i < j \}
        \,;
    \end{equation*}
    these sets are $\HM^{\vdim}$~measurable, pairwise disjoint, and satisfy
    $\HM^{\vdim} ( A \without \tbcup \{ E_j : j \in \natp \} ) = 0$.

    \emph{Step 3 (the approximate tangent plane is transversal to $\ker p_j$).}
    Fix~$j$.  Since $\HM^{\vdim} \restrict A \without E_j$ is a~Radon measure,
    \cite[2.10.19(4)]{Federer1969} yields $\density^{\vdim}(\HM^{\vdim}
    \restrict A \without E_j , a) = 0$ for $\HM^{\vdim}$~almost all $a \in E_j$;
    fix such an~$a$ and abbreviate $T = \ap\Tan(A,a)$.  By~\cite[3.14, 3.16,
    3.19]{Santilli2019d},
    \begin{equation*}
        T = \Tan^{\vdim}(\HM^{\vdim} \restrict A, a) \in
        \grass{\adim}{\vdim} \,,
    \end{equation*}
    the approximate tangent cone being taken in the sense
    of~\cite[3.2.16]{Federer1969}, which is the upper cone
    $\Tan^{*\vdim}$ of~\cite[3.1]{Santilli2019d} and coincides with the cone
    $\Tan^{\vdim}$ of the latter whenever the two cones there agree, as they do
    here; that is, $v \in T$ if and only if
    $\density^{*\vdim}(\HM^{\vdim} \restrict A \cap E(a,v,\varepsilon), a) > 0$
    for every $\varepsilon > 0$, where $E(a,v,\varepsilon) = \R^{\adim} \cap \{
    z : |r(z-a) - v| < \varepsilon \text{ for some } 0 < r < \infty \}$.  As
    $\HM^{\vdim} \restrict A = \HM^{\vdim} \restrict E_j + \HM^{\vdim}
    \restrict A \without E_j$ and the second summand has vanishing
    $\vdim$~density at~$a$, these upper densities are unchanged when $A$ is
    replaced by~$E_j$; hence, $T = \Tan^{\vdim}(\HM^{\vdim} \restrict E_j, a)$.
    Taking $S = E_j$ in the definition of $\Tan^{\vdim}$
    in~\cite[3.2.16]{Federer1969} we~obtain $T \subseteq \Tan(E_j,a) \subseteq
    \Tan(M_j,a)$, the latter inclusion because $E_j \subseteq M_j$; as $\dim T =
    \vdim = \dim \Tan(M_j,a)$, we~conclude $T = \Tan(M_j,a)$.  Finally
    $\Tan(M_j,a) \cap \ker p_j = \{0\}$, because $\Tan(M_j,a) = \{ p_j^{*}(u) +
    q_j^{*}(\uD g_j(y)(u)) : u \in \R^{\vdim} \}$ with $y = p_j(a)$ and $\ker p_j
    = \im q_j^{*}$.

    \emph{Step 4 (approximate differentiability of the graph maps).}  Fix~$j$
    and put $X_j = \{ y \in V_j : p_j^{*}(y) + q_j^{*}(g_j(y)) \in E_j \}$; this
    set is $\LM^{\vdim}$~measurable, being the preimage of the
    $\HM^{\vdim}$~measurable set~$E_j$ under a~continuous injection whose
    inverse is Lipschitzian, and
    \begin{equation*}
        E_j = \{ p_j^{*}(y) + q_j^{*}(g_j(y)) : y \in X_j \} \subseteq A \,,
        \qquad \Lip (g_j|X_j) \le 1 \,.
    \end{equation*}
    Let $x \in X_j$ be a~point of density~$1$ of~$X_j$ --- by the Lebesgue
    density theorem~\cite[2.9.11]{Federer1969} this excludes an
    $\LM^{\vdim}$~null set only --- such that, moreover, $a = p_j^{*}(x) +
    q_j^{*}(g_j(x))$ belongs to the full measure subset of~$E_j$ singled out in
    Step~3.  Applying~\ref{lem:set-to-function} with $S$, $p$, $q$, $X$, $f$,
    $L$, $E$, $\alpha$ replaced by $S_j$, $p_j$, $q_j$, $X_j$, $g_j$, $1$,
    $E_j$, $\frac 12$, and with $k$ arbitrary --- its
    hypothesis~\ref{i:stf:apdiff} holding by Step~1 and Step~3, and
    its~\ref{i:stf:density} by the choice of~$x$ --- we~infer that $g_j$ is
    approximately differentiable of order~$k$ at~$x$ for every $k \in \natp$.
    Since the exceptional set is $\LM^{\vdim}$~null and independent of~$k$, this
    holds at $\LM^{\vdim}$~almost all $x \in X_j$; we~let $A_j$ be the
    $\LM^{\vdim}$~measurable set of those~$x$, so that $\LM^{\vdim}(X_j \without
    A_j) = 0$.

    \emph{Step 5 (conclusion).}  Fix~$j$ and apply~\ref{lem:Ck-to-Cinfty}
    to~$g_j$ and~$A_j$, obtaining maps $g_{j,i} \in
    \cnt{\infty}(\R^{\vdim},\R^{\adim-\vdim})$, $i \in \natp$, with
    $\LM^{\vdim}(X_j \without \tbcup \{ \{ y : g_{j,i}(y) = g_j(y) \} : i \in
    \natp \}) = 0$.  The sets $N_{j,i} = \{ p_j^{*}(y) + q_j^{*}(g_{j,i}(y)) : y
    \in \R^{\vdim} \}$ are $\vdim$~dimensional submanifolds of~$\R^{\adim}$ of
    class~$\cnt{\infty}$ and, the map $y \mapsto p_j^{*}(y) + q_j^{*}(g_j(y))$
    being Lipschitzian on~$X_j$, \cite[2.10.11]{Federer1969} gives
    $\HM^{\vdim}(E_j \without \tbcup \{ N_{j,i} : i \in \natp \}) = 0$.  Summing
    over~$j$ and recalling Step~2 we~obtain
    \begin{equation*}
        \HM^{\vdim} \bigl( A \without \tbcup \{ N_{j,i} : i,j \in \natp \}
        \bigr) = 0 \,;
    \end{equation*}
    hence, $\HM^{\vdim}(\spt\|V\| \without \tbcup \{ N_{j,i} : i,j \in \natp \})
    = 0$ and $\|V\|(U \without \tbcup \{ N_{j,i} : i,j \in \natp \}) = 0$, which
    is the assertion.
\end{proof}

\subsection*{Tool and computational resource disclosure}
\addcontentsline{toc}{section}{Tool and computational resource disclosure}
\label{sec:AI-tools}

As said in the introduction, most of the manuscript was written by hand by the
author. To finish the work the tools \emph{Claude~AI} and \emph{Claude~Code}
have been employed mostly using the \emph{Opus~5} model. In~particular, the LLM
has been used to fill many gaps, adjust the constants, check consistency, verify
some conjectures, draft the introduction, and proof-read the paper. Claude Code
has generated a lot of \LaTeX{} code basing on the initial draft, the literature
enumerated in the bibliography, and author's guidance. Especially, the second
half of section~\ref{sec:elliptic}, starting from~\ref{rem:ho:plan}, was
written entirely by Claude Code (supported by Federer's book) -- this is rather
standard PDE theory but the author could not find the precise estimates in the
form needed for this proof in the literature. Also
section~\ref{sec:rectifiability} is a~production by Claude Code. However, a~private
communication of~Menne was the basis for generating
section~\ref{sec:rectifiability}. Section~\ref{sec:reach} was initially written
entirely by Menne and then adjusted by the author via the AI tool. All other
parts where written originally by the author and then transformed using Claude.

The author takes full responsibility for the final outcome of these actions
presented here but he must also admit that he has not thoroughly checked all the
content generated by the AI. Nonetheless, convinced that postponing it would not
improve the quality, he decided to publish the paper.

\subsection*{Acknowledgements}
\addcontentsline{toc}{section}{Acknowledgements}
This research was financed by the \href{https://ncn.gov.pl/}{National Science
  Centre Poland} grant number 2022/46/E/ST1/00328.

The author wishes to thank Camillo De Lellis for his kind invitation to
Princeton in October~2025 and for his gentle encouragement to finish this
project.

\bigskip
{\small
\addcontentsline{toc}{section}{\numberline{}References}
\bibliography{all-refs}
\bibliographystyle{halpha}
}

{\small \noindent
  S{\l}awomir Kolasi{\'n}ski
  \\
  Instytut Matematyki,
  Uniwersytet Warszawski
  \\
  ul. Banacha 2, 02-097 Warszawa, Poland
  \\
  \texttt{s.kolasinski@mimuw.edu.pl}
}



\end{document}